\documentclass{amsart}
\usepackage{amssymb,euscript,tikz,units,mathrsfs,stmaryrd, amsmath}
\usepackage[colorlinks,citecolor=blue,linkcolor=blue]{hyperref}
\usepackage{verbatim}
\usepackage[nameinlink]{cleveref}
\usepackage{colonequals}
\usepackage{tikz-cd}
\usetikzlibrary{arrows.meta, positioning, calc}
\usepackage[titletoc]{appendix}
\usepackage{yhmath}
\usepackage{enumerate}
\usepackage{hyperref}
\usepackage{url}
\usepackage{makecell} % 用于文字换行
\usepackage[margin=1in]{geometry}
\usepackage{xcolor}
\usepackage{enumitem}
\usepackage{fancyhdr} % Headers and footers
\usepackage{lastpage}
\usepackage[all]{xy}
\theoremstyle{plain}
\newtheorem{theorem}{Theorem}[section]

\newtheorem{proposition}[theorem]{Proposition}
\newtheorem{lemma}[theorem]{Lemma}

\newtheorem{question}[theorem]{Question}
\newtheorem{remark}[theorem]{Remark}

\newcommand{\be}{\begin{equation}}
\newcommand{\ene}{\end{equation}}
\newcommand{\br}{\begin{remark}}
\newcommand{\er}{\end{remark}}
\newcommand{\bl}{\begin{lem}}
\newcommand{\el}{\end{lem}}
\newcommand{\bcor}{\begin{cor}}
\newcommand{\ecor}{\end{cor}}
\newcommand{\bpro}{\begin{pro}}
\newcommand{\epro}{\end{pro}}
\newcommand{\ben}{\begin{enumerate}}
\newcommand{\een}{\end{enumerate}}
\newcommand{\bp}{\begin{proof}}
\newcommand{\ep}{\end{proof}}
\newcommand{\bpo}{\begin{pro}}
\newcommand{\epo}{\end{pro}}
\newcommand{\beq}{\begin{equation*}}
\newcommand{\eeq}{\end{equation*}}
\newcommand{\bear}{\begin{eqnarray}}
\newcommand{\eear}{\end{eqnarray}}
\newcommand{\beqar}{\begin{eqnarray*}}
\newcommand{\eeqar}{\end{eqnarray*}}
\newcommand{\bt}{\begin{theorem}}
\newcommand{\et}{\end{theorem}}
\newcommand{\bex}{\begin{excer}}
\newcommand{\eex}{\end{excer}}

\theoremstyle{definition}

\theoremstyle{remark}

\newtheorem*{con*}{Construction}
\newtheorem*{rem*}{Remark}
\newtheorem*{exam*}{Example}
\newtheorem*{exams*}{Examples}
\newtheorem*{thm*}{\bf Theorem}
\newtheorem*{que*}{Question}

\newtheorem*{Def*}{Definition}
\newtheorem*{Cons*}{Construction}
\newtheorem*{Lem*}{Lemma}
\newtheorem*{Conj*}{\bf Conjecture}

\newtheorem{Def}{Definition}[section]

\numberwithin{equation}{section} \numberwithin{figure}{section}
\usepackage{tikz}
\usetikzlibrary{cd}

\begin{document}

\title{Existence and nonrelativistic limit of ground states to nonlinear Dirac equation
}
 
\author{Pan Chen$^{1}$, Yanheng Ding$^{3,4}$, Qi Guo$^{2}$ }

\maketitle

%%==================================%%
%% sample for unstructured abstract %%
%%==================================%%

\begin{center}
{\small $^1$ School of Mathematical Sciences, Shanghai Jiao Tong University, Shanghai, 200240, P. R. China       }

{\small $^{2}$ School of Mathematics, Renmin University of China, Beijing,  100872, P. R. China  }

  {\small $^3$ Academy of Mathematics and Systems Science, Chinese Academy of Sciences, Beijing, 100190, P. R. China  }

{\small $^4$ School of Mathematics, Jilin University, Changchun, 130012, P. R. China }
\end{center}

\noindent{ \bf Abstract  }:
{\small  This paper explores the existence and properties of ground states, including both energy and action ground states, for nonlinear Dirac equations with power-type potentials.
\begin{equation*}
-i c\sum\limits_{k=1}^3\alpha_k\partial_k u  +mc^2 \beta {u}- |{u}|^{p-2}{u}=\omega {u}.
\end{equation*}
We establish the existence of energy ground states and demonstrate that as the speed of light approaches infinity, both energy and action ground states converge to their counterparts in the nonlinear Schr\"odinger equation. Furthermore, we characterize the convergence rate of the ground state energy and investigate the equivalence between action and energy ground states.

% In this paper, we are concerned with the existence and nonrelativistic limit of energy and action ground states of the following nonlinear Dirac equation with power-type potentials 
% \begin{equation*}
% -i c\sum\limits_{k=1}^3\alpha_k\partial_k u  +mc^2 \beta {u}- |{u}|^{p-2}{u}=\omega {u}.
% \end{equation*}
% We first prove the existence of energy ground states. 
% We then show that, as $c\to \infty$, both the energy and action ground states 
% converge to the corresponding ground states of the nonlinear Schr\"odinger equation, and we characterize the convergence rate of the ground state energy.
% Furthermore, we investigate the relationship between
%  action and energy ground states of the nonlinear Dirac equation.

}


\numberwithin{equation}{section}
\tableofcontents

\section{Introduction}

The nonlinear Dirac equation provides a relativistic framework for high-velocity fermions (e.g., electrons), with applications across atomic physics, condensed matter, and quantum field theory. Understanding ground states is essential for predicting their stability and dynamics \cite{ED11}. This paper investigates the existence and properties of ground states for nonlinear Dirac equations with power-type potentials, and bridges relativistic and nonrelativistic descriptions by examining the nonrelativistic limit and its connection to the nonlinear Schr\"odinger equation. In this work, the nonlinear Dirac equation under consideration takes the form
\begin{align}\label{Dirac}\tag{$\mathrm{NDE}_{\mathrm{\omega_c}}$}
\mathscr{D}_c u-|u|^{p-2}u=\omega_c  u,
\end{align}
where $\mathscr{D}_c:=-ic\alpha\cdot \nabla +mc^2\beta$ is the free Dirac operator, 
$u:\mathbb{R}^3\rightarrow \mathbb{C}^4$ is the Dirac wave function, $\omega_c\in \mathbb{R}$ is the frequency of the wave function, $c$ is the speed of light. Throughout the paper, we will consistently assume that the exponent $p$ lies within the interval $(2,3)$, which corresponds to the Sobolev subcritical case. Notably, when $p = 8/3$, this nonlinear term aligns with the exchange-correlation potential in the Relativistic Density Functional Theory at the Lieb-Oxford bound \cite{LLS22, LiebOxford} and is also referred to as the mass-critical exponent.
It is well known that the free Dirac operator $\mathscr{D}_c$ is self-adjoint in $L^2(\mathbb{R}^3,\mathbb{C}^4)$, with domain $H^1(\mathbb{R}^3,\mathbb{C}^4)$ and formal domain $H^{1/2}(\mathbb{R}^3,\mathbb{C}^4)$. Its spectrum is $(-\infty, -mc^2]\cup [mc^2,\infty)$, then  our working space $E_c$ can be defined by the completion of $\text{dom}(|\mathscr{D}_c|^{1/2})$ under the following inner product
$$
(u_1,u_2)_c:= \left(|\mathscr{D}_c|^{1/2}u_1, |\mathscr{D}_c|^{1/2}u_2\right)_{L^2},
$$
The induced norm is denoted by $\| u\|_c:=(u,u)_c^{1/2}$. 

\subsection{Ground States}
In physics, the ground states typically refer to the lowest-energy states among all  positive-energy configurations. There are two main types: action ground states, which are classical trajectories or instantons that extremize the action functional, and energy ground states, which are static eigenstates that minimize the Hamiltonian or energy functional.

For nonlinear Dirac equations, the study of action ground states involves fixing the frequency  $\omega_c$ within the spectral gap $(-mc^2, mc^2)$ and identifying the critical points of the associated functional on an appropriate constraint set. Problems formulated in this manner are known as {\it fixed-frequency problems}. Unlike the case of the Schr\"odinger equation and the pseudo-relativistic equation, the action functional 
\[
\mathcal{J}_{\omega_c}^c(u)=\|u^+\|_c^2-\|u^-\|_c^2-\omega_c\|u\|_{L^2}^2
-\frac{2}{p}\int_{\mathbb{R}^3 } |u|^p dx
\]
associated with \eqref{Dirac} is strongly indefinite, where $u^\pm$ is the projection of $u$ 
onto the subspace $E_c^\pm$. Since the negative space is infinite-dimensional, we consider its restriction to the following  reduced Nehari manifold
\[\mathcal{N}_{\omega_c}^c:=\left\{ u\in E_c^+: d\mathcal{J}_{\omega_c,red}^c(u) [u]=0 \right\},\]
where the reduced functional is defined as
$$
\mathcal{J}_{\omega_c,red}^c(u):=
\sup_{v\in E_c^-}\mathcal{J}_{\omega_c}^c(u+v),\quad u\in E_c^+,
$$
see Section \ref{section5} for more details.
An alternative approach to proving the existence of solutions to nonlinear Dirac equations involves studying the critical points of the following energy functional
\[
\mathcal{I}^{c}(u) := \|u^+\|_c^2 - \|u^-\|_c^2 - \frac{2}{p} \int_{\mathbb{R}^3} |u|^p  dx 
\]  
constrained to the $L^2$-sphere  
\[
\mathcal{S} := \left\{ u \in H^{1/2}(\mathbb{R}^3,\mathbb{C}^4) : \|u\|_{L^2} = 1 \right\}.
\]  
Such problems are known as {\it prescribed mass problems} or {\it normalized problems}. 
In this paper, we adopt alternative definitions of action and energy ground states from the perspective of variational methods.

 \begin{Def}[Ground states of Dirac equation]Using the notations introduced above,
  \begin{enumerate}
  \item (action ground state) A function $u_c\in E_c$ is called a (relativistic) action ground state if $u_c^+$ is a minimizer of the following minimization problem:
  \[e_{\omega_c,act}^c:=\inf_{u\in\mathcal{N}_{\omega_c}^c}\mathcal{J}_{\omega_c,red}^c(u),\]
  and $u_c^-$ is a solution of the maximization
  problem:
  $$
\mathcal{J}_{\omega_c,red}^c(u_c^+):=
\sup_{v\in E_c^-}\mathcal{J}_{\omega_c}^c(u_c^++v).
$$
  \item (energy ground state) A function $u_c\in \mathcal{S}$ is 
  called an (relativistic) energy ground state if 
$u_c$ is a solution of the following minimization problem:
\begin{equation*}\label{new_a}
   e_{ene}^c:=\inf \left\{ \mathcal{I}^{c}(u): \|u\|^2_{L^2}=1,~\mathcal{I}^c(u)>0,  ~d\mathcal{I}^{c}|_{\mathcal{S}}(u)=0  \right\},
\end{equation*}
  \end{enumerate}
\end{Def}
Based on the above definition, the energy ground state $u_c$ for {\it normalized problem} must satisfy the following Dirac equation,
\begin{equation}\label{Dirac_e}\tag{$\mathrm{NDE}_{\mathrm{ene}}$}
    \begin{cases}
		 \mathscr{D}_cu_c- |u_c|^{p-2}u_c= \omega_cu_c
      \\
      \displaystyle\int_{\mathbb{R}^3}|u_c|^2 =1,
    \end{cases}
\end{equation}
with Lagrange multiplier $\omega_c$. 
 The energy ground state is a critical point of the energy functional on the constraint set, where the energy reaches its positive minimum. Whether the positivity condition can be removed is open, so we maintain this assumption.

In the nonrelativistic case, we introduce the following Schr\"odinger equation:
\begin{align}\label{laplace}\tag{$\mathrm{NSE}_{\mathrm{\lambda}}$}
- \frac{\Delta}{2m} f +\lambda f=|f|^{p-2}f .
\end{align}
Here, $\lambda>0$ is the frequency of the wave function $f:\mathbb{R}^3\rightarrow \mathbb{C}^2$.  
% The wave functions in \eqref{laplace} and \eqref{Dirac} differ primarily in their phase spaces, i.e. the phase space of $f$ is $\mathbb{C}^2$, whereas the phase space of $u$ is $\mathbb{C}^4$. 
Similarly, one can study the action ground state of \eqref{laplace}
by minimizing the action functional 
% $\mathcal{J}^\infty_\lambda: H^{1}(\mathbb{R}^3, \mathbb{C}^2)\to \mathbb{R}$:
\begin{equation*}
	\mathcal{J}^\infty_\lambda(f):= \frac{1}{2m}\int_{\mathbb{R}^3} |\nabla
	f|^2 dx + \lambda\int_{\mathbb{R}^3} | f|^2 dx-\frac{2}{p}\int_{\mathbb{R}^3 } |f|^p dx
\end{equation*}
on the following Nehari manifold
\begin{equation*}
	\mathcal{N}^\infty_\lambda:= \left\{f\in H^{1}(\mathbb{R}^3,\mathbb{C}^2) \setminus
	\{0\}: d\mathcal{J}^\infty_\lambda(f)[f] = 0 \right\}.
\end{equation*}
% $\mathcal{I}^\infty : H^{1}(\mathbb{R}^3, \mathbb{C}^2)\to \mathbb{R}$
The corresponding {\it normalized problem} for the nonlinear Schr\"odinger equation can be obtained by studying the following energy functional
\begin{equation*}
	\mathcal{I}^\infty(f):= \frac{1}{2m}\int_{\mathbb{R}^3} |\nabla
	f|^2 dx -\frac{2}{p}\int_{\mathbb{R}^3 } |f|^p dx
\end{equation*}
on the $L^2$-sphere
\begin{equation*}
	    \mathcal{S}'= \left\{f\in H^{1}(\mathbb{R}^3, \mathbb{C}^2):
	    \|f\|_{L^2}=1\right\}.
\end{equation*}
This minimization problem is well defined since we restrict the index $p$ to the interval $(2, 3)$. Many references study action and energy ground states for the nonlinear Schrödinger equation, as well as the equivalence between them. We refer interested readers to \cite{DST23} as a starting point for exploring related results.
\begin{Def}[Ground states of Schrödinger equation]
  Using the notations introduced above,
  \begin{enumerate}
\item (action ground state) A function $f$ is
  called a (nonrelativistic) action ground state of
\eqref{laplace} if
  \begin{equation*}
    \mathcal{J}^\infty_{\lambda}(f)=e_{\lambda,act}^\infty:=\inf_{v \in \mathcal{N}^\infty_{\lambda}}\mathcal{J}^\infty_{\lambda}(v).
  \end{equation*}
  \item (energy ground state) A function $f\in \mathcal{S}'$ is
  called an (nonrelativistic) energy ground state if 
$f$ is a solution of the following minimization problem:
\begin{equation*} 
  \mathcal{I}^\infty(f) =e_{ene}^\infty:= \inf_{v\in \mathcal{S}'}\mathcal{I}^\infty(v).
\end{equation*}
  \end{enumerate}
\end{Def}
In the nonrelativistic case, energy ground state $f$ of Schr\"odinger equations must satisfy the following equation
\begin{equation}
	\label{laplace_e}\tag{$\mathrm{NSE}_{\mathrm{ene}}$}
    \begin{cases}
      - \frac{\Delta}{2m} f +\lambda f=|f|^{p-2}f 
      \\
      \displaystyle\int_{\mathbb{R}^3}|f|^2 =1,
    \end{cases}
\end{equation}
with Lagrange multiplier $\lambda$. 
%The existence of action ground states and energy ground states for \eqref{laplace}, 
%as well as the existence of action ground state for \eqref{Dirac}, can be established
% through standard minimization methods on (generalized) Nehari manifold 
%and constrained variational techniques; we refer to \cite{dyh, MR2768820,MR1430506} for details and omit them here.
Both action and energy ground states play central roles in quantum theory, and natural questions are whether such ground states exist in relativistic and nonrelativistic models and how the different notions of ground state relate to each other.
% In the spirit of the key questions asked in the relativistic models, it is natural to ask 
% whether \eqref{Dirac}(or \eqref{Dirac_e}) admits an action (or energy) ground state and, 
% if so, what's the relationship between relativistic 
% action (energy) ground states  and   nonrelativistic action (energy) ground states. 
More precisely, one may ask the following questions:
\begin{question}\label{Q1}
  \normalfont % 或者使用 \rmfamily 或 \upshape
  \quad
  \begin{enumerate}
    \item Do action ground states exist for the fixed-frequency problem \eqref{Dirac}, and do they converge to those of \eqref{laplace} in the nonrelativistic limit?
    \item Regarding the prescribed mass problem \eqref{Dirac_e}, do energy ground states exist and converge to those of \eqref{laplace_e} in the nonrelativistic limit?
    \item What is the convergence rate of the ground state energies \(e^{c}_{\omega_c,{act}}\) and \(e^{c}_{{ene}}\) to their nonrelativistic counterparts as \(c \to \infty\)?
    \item If an energy ground state exists for \eqref{Dirac_e}, how is it related to the action ground states of \eqref{Dirac}?
    \end{enumerate}
\end{question}

 \begin{remark}
     \begin{itemize}
         \item[(a)] Regarding question (1), the existence of action ground states for \eqref{Dirac} can be established using standard minimization methods on the reduced Nehari manifold, as detailed in \cite{dyh}. Currently, no literature addresses the relationship between relativistic and nonrelativistic action ground states. To explore this connection, we'll use a concentration-compactness argument as outlined in \cite{MR778970}.
         
         \item[(b)] Regarding question (2), Coti Zelati and Nolasco established the existence of energy ground states for the nonlinear Dirac equation with $2<p\leq 8/3$, as well as for Hartree-type nonlinearity. Their findings are detailed in \cite{CN19, CN25, Nolasco}. In their framework, the speed of light is set to \(c=1\) with the requirement that the nonlinear term is sufficiently small, which aligns with our conditions as explained by the following statement. For any $u\in H^{1/2}(\mathbb{R}^3, \mathbb{C}^4)$, if we set $\tilde{u}(x)=c^{-3/2}u(c^{-1}x)$, then $$\mathcal{I}^{1,c^{-1}}(\tilde{u}):=\|\tilde{u}^+\|^2_1-\|\tilde{u}^-\|_1^2-2p^{-1}c^{3p/2-5}\int_{\mathbb{R}^3 } |\tilde{u}|^p dx=c^{-2} \mathcal{I}^c(u). $$  
         Recent research has extensively explored the relationship between energy ground states of the pseudo-relativistic (Hartree) and nonrelativistic Schr\"odinger equations (see \cite{arXiv:2505.05917, MR4430585, MR2561169}). Some work also addresses the connection between Dirac and Schr\"odinger ground states under $L^2$-critical or subcritical growth (see \cite{CDGW24, Dolbeault24, ES01}). However, for the $L^2$-supercritical case $8/3 < p < 3$, the existence of energy ground states and their relation to Schr\"odinger ground states remain open. Tackling this gap is a primary goal of this paper.  
         
         \item[(c)] Regarding question (3), in \cite{arXiv:2505.05917}, the authors showed that the convergence rate of the ground state energy in the nonrelativistic limit for the pseudo-relativistic Hartree equation is $1/c^2$, based on a Taylor expansion of the pseudo-relativistic operator: \[\sqrt{-c^2\Delta+m^2c^4}-mc^2=-\frac{\Delta}{2m}+ \mathcal{O}\left(\frac{1}{c^2}\right).\]  
         Nonlinear Dirac problems are more challenging due to the unbounded spectrum of the Dirac operator. In \cite{arXiv:2503.21405}, refined projection estimates were used to establish, for the first time, a convergence rate of \(1/c^2\) for the Dirac-Fock ground state energy. Similarly, this paper shows the same \(1/c^2\) rate for \eqref{Dirac} and \eqref{Dirac_e}, based on linearizing the energy functional around the ground state.

        \item[(d)]  Regarding question (4), in the nonrelativistic case, the relationship between the action  ground states of \eqref{laplace} and the energy ground states of \eqref{laplace_e} has been studied in \cite{DST23}. Inspired by this work, this paper investigates the relationship between the action ground states of \eqref{Dirac} and the energy ground states of \eqref{Dirac_e}.
     \end{itemize}
 \end{remark}

  The purpose of this paper is to complete the following diagram, where the black arrows represent existing results, and the red arrows indicate the problems to be addressed in this paper.

\begin{figure}[htbp]
  \centering
  \begin{tikzpicture}[
    every node/.style={text width=3cm, align=center}
  ]
  
    % 定义节点位置
    \node (tl) at (0,0) {Action Ground States of \eqref{Dirac}};
    \node (tr) at (7,0) {Action Ground States of \eqref{laplace}};
    \node (bl) at (0,-4) {Energy Ground States of \eqref{Dirac_e}};
    \node (br) at (7,-4) {Energy Ground States of \eqref{laplace_e}};
    \node (center) at (3.5,-2) {Existence};
    
    % 绘制原有箭头
    \draw[->, red, thick] (tl) -- node[above, black] {Nonrelativistic Limits} (tr);
    \draw[->, red, thick] (bl) -- node[below, black] {Nonrelativistic Limits} (br);
    
    % 左侧等价关系 - 使用双向箭头 ⇔
    \draw[<->, red, thick, double, double distance=1.5pt] (tl) -- node[left, black] {} (bl);
    
    % 右侧等价关系 - 使用双向箭头 ⇔
    \draw[<->, thick, double, double distance=1.5pt] (tr) -- node[right, black] {} (br);
    
    % 绘制从中心到四个角的箭头
    \draw[->, thick] (center) -- (tl);
    \draw[->, thick] (center) -- (tr);
    \draw[->, red, thick] (center) -- (bl);
    \draw[->, thick] (center) -- (br);
  \end{tikzpicture}
  \caption{Diagram showing the relationships between the ground states. The double-headed arrows represent equivalences}
  \label{fig:diagram}
  \end{figure}
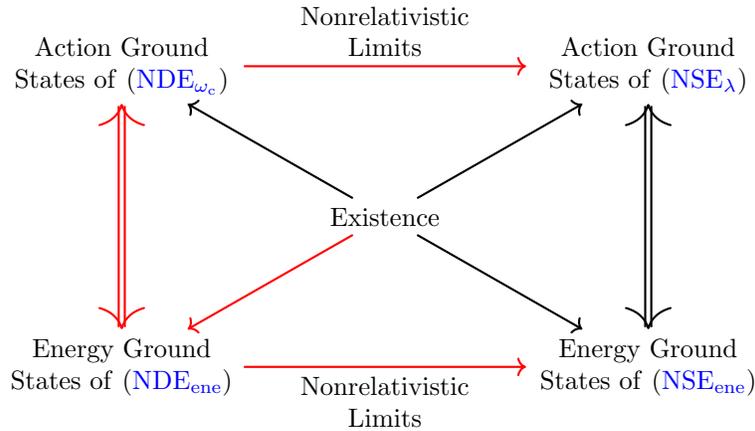

\subsection{Main Results}
We address Question \ref{Q1} with four main theorems: first, establishing the existence of energy ground states for \eqref{Dirac_e}; second and third, analyzing the relationship and convergence between relativistic and nonrelativistic energy (and action) ground states; and fourth, examining the consistency between the action ground states of \eqref{Dirac} and the energy ground states of \eqref{Dirac_e}. Our initial result is as follows:

% We address Question \ref{Q1} through four main theorems. First, 
% we establish the existence of energy ground states for \eqref{Dirac_e}. 
% The second and third main theorems explore the relationship between relativistic 
% and nonrelativistic energy (and action) ground states. This includes a comparison 
% of their roles in describing physical phenomena, as well as the conditions under 
% which they converge or approximate one another. The fourth main theorem investigates
%  the consistency between the action ground states of \eqref{Dirac}
%  and the energy ground states of \eqref{Dirac_e}. Our first result is stated as follows:
\begin{theorem}[Existence and properties of energy ground states]\label{them:1.1}
  \normalfont
  Let $ p\in (2,3) $, then there exists $c_0>0$, such that for $c>c_0$, the following results are valid.
  \begin{enumerate}
      \item (Existence) There exists $\omega_c\in (0,mc^2)$ and a function $ u_c\in  H^{1/2}(\mathbb{R}^3,\mathbb{C}^{4})$ solves \eqref{Dirac_e}. Moreover,
  $u_c$ is a energy ground state of  \eqref{Dirac_e}.
In addition, there holds
  $$-\infty < \liminf\limits_{c\to \infty} (\omega_c-mc^2)\leq \limsup\limits_{c\to \infty} (\omega_c-mc^2)< 0.$$
  \item (Exponential Decay)
 There exist $\delta>0$ and $C(\delta) >0$ independent of $c$,  such that
  $$
|P_\infty^+u_c(x)|\leq C(\delta)e^{-\delta|x|}, \quad
|P_\infty^-u_c(x)|\leq \frac{C(\delta)}{c}e^{-\delta|x|},
$$
where $P_\infty^{+}$ $(P_{\infty}^-)$  project onto the first two (last two) components, respectively.
\item (Uniqueness of Lagrange multiplier) The  multiplier $\omega_c$ associated with the same energy ground state is unique, except for a countable set $\Xi$ of $c$.
  \end{enumerate}
\end{theorem}
 
\begin{remark}
\begin{itemize}
    \item [(1)] The existence of Theorem \ref{them:1.1} for \( p \in (2, 8/3] \) follows directly from recent results by Coti Zelati and Nolasco in \cite{CN25}. Therefore, for the existence results, we only need to discuss \( p \in (8/3, 3) \), which is the case of $L^2$-supercritical.
    \item[(2)] The proof of uniqueness of the multiplier \( \omega_c \) 
    is inspired by Lenzmann \cite{MR2561169}, 
    we show that the uniqueness of \( \omega_c \) in \eqref{Dirac_e} after 
    removing a countable set. Furthermore, Guo and Zeng in \cite{GuoZeng20} eliminated this condition for pseudo-relativistic Hartree equations using nondegeneracy of ground state, which we lack for \eqref{Dirac_e}, thus, we cannot remove this condition here.
\end{itemize}
 
\end{remark}

Our second main theorem reveals the nonrelativistic limit of the energy ground state \(u_c\) of \eqref{Dirac_e} as \(c \to \infty\). The positive and negative parts are \(u_c^+ = P_c^+ u_c\) and \(u_c^- = P_c^- u_c\), while the first two and last two components are \(f_c = P_\infty^+ u_c\) and \(g_c = P_\infty^- u_c\).

\begin{theorem}[Nonrelativistic limit of  energy ground states]\label{them:1.2}
  \normalfont
  Under the assumptions of Theorem \ref{them:1.1}, for every \(c > c_0\), up to a subsequence and 
   translation, the following asymptotic properties hold. 
   \begin{enumerate}
    \item     
\[\left\|u_c^+ - (f_\infty, 0)^T\right\|_{H^1} \to 0, \qquad \left\|u_c^-\right\|_{H^1} = \mathcal{O}\left(\frac{1}{c^2}\right),\] 
where \(f_\infty\in H^1(\mathbb{R}^3, \mathbb{C}^2)\) is an energy ground state of \eqref{laplace_e}.
  \item \[
  \left\|f_c - f_\infty\right\|_{H^1} \to 0, \qquad \left\|g_c\right\|_{H^1} = \mathcal{O}\left(\frac{1}{c}\right),
  \] 
  \item  $$e_{ene}^c= e^\infty_{ene}+mc^2+ \mathcal{O}\left(\frac{1}{c^2}\right).$$
  \end{enumerate}
 \end{theorem}

The above theorem shows that, in the nonrelativistic limit, the positive spectral part or the first two components of the Dirac energy ground states converge to those of the nonlinear Schr\"odinger equation, while the negative part and the last two components vanish. This indicates that the nonlinear Dirac equation can be viewed as a relativistic extension, with both equations giving the same results as the particle velocity approaches zero. It also reveals the rate of nonrelativistic correction terms. 

Similarly, we obtain the following result for the nonrelativistic limit of the action ground states of \eqref{Dirac}, and we will denote \(\{u_c\}\)  by the action ground states of \eqref{Dirac}.

\begin{theorem}[Nonrelativistic limit of  action ground states]\label{them:1.3}
  \normalfont
  For each $\lambda>0$, $\omega_c= mc^2-\lambda$, $c>\sqrt{\frac{\lambda}{2m}}$, up to a subsequence and translation, the following asymptotic properties hold.
   \begin{enumerate}
    \item 
   \[
  \left\|u_c^+ - (f_\infty, 0)^T\right\|_{H^1} \to 0, \qquad \left\|u_c^-\right\|_{H^1} = \mathcal{O}\left(\frac{1}{c^2}\right),
  \] 
  where \(f_\infty\in H^1(\mathbb{R}^3, \mathbb{C}^2)\) is an action ground state of \eqref{laplace}.
  \item 
  \[
  \left\|f_c - f_\infty\right\|_{H^1} \to 0, \qquad \left\|g_c\right\|_{H^1} = \mathcal{O}\left(\frac{1}{c}\right).
  \]
  \item 
  $$
  e_{\omega_c, act}^c = e_{\lambda, act}^\infty+ \mathcal{O}\left( \frac{1}{c^2}\right).
  $$
  \end{enumerate}

 \end{theorem}

 \begin{remark}\label{rem_1.7}
\begin{itemize}
  \item [(a)] Theorem \ref{them:1.2} and Theorem \ref{them:1.3} describes the asymptotic behavior of the first (last) 
  two components and the positive (negative) energy parts of the energy ground 
  state solutions, respectively. From the perspective of Lemma \ref{nar}, 
  these two types of asymptotic behavior are in fact equivalent provided  
  \[
  \sup_{c > c_0} \|u_c\|_{H^2} < \infty.
  \]
  \item[(b)] When analyzing the asymptotic behavior of \(u_c^+\), 
  which is a four-component function, it is necessary to extend
   \(f_\infty\) to a four component vector \((f_\infty, 0)^T\).
   \item[(c)]Owing to the fact that the action ground state $f$ of equation \eqref{laplace} can be constructed from the unique positive radial solution \(v\) of the scalar field equation  
\[
-\Delta v + v = v^{p-1}
\]  
via the action of \(SU(2)\), see appendix \ref{a_B}, the convergence of \(f_c\) in Theorem \ref{them:1.3} (2) is given by  
\[
\|f_c(\cdot + x_c) - \gamma_c \cdot (v, 0)^T\|_{H^1} \to 0,
\]  
for some \(\gamma_c \in SU(2)\) and translation parameters \(x_c \in \mathbb{R}^3\).
\item[(d)] The convergence rate of the ground state energy for the pseudo-relativistic Hartree (Schrödinger) equation can be directly derived from the Taylor expansion of the pseudo-relativistic operator, as shown in \cite{arXiv:2505.05917}. In contrast, establishing this rate for the Dirac equation is more complex due to its unbounded spectrum. 
% We first need a maximization mapping to reduce the energy functional to a semi-definite form, which complicates the variational characterization of ground states and makes proving the convergence rate significantly more challenging.

% The convergence rate of the ground state energy for the pseudo-relativistic Hartree (Schrödinger) equation can be directly obtained via the Taylor expansion of the pseudo-relativistic operator, as seen in \cite{arXiv:2505.05917}. In contrast, proving the convergence rate of the ground state energy for the Dirac equation is more intricate. The primary difficulty arises from the fact that the spectrum of the Dirac operator is unbounded from below., we first require a maximization mapping to reduce the energy functional to a semi-definite one. The complexity of this maximization mapping complicates the variational characterization of the ground state solutions for the Dirac equation, thereby making the proof of the convergence rate of the ground state energy considerably more challenging.
\end{itemize}
   
\end{remark}

Our final theorem addresses the consistency between the action ground state of \eqref{Dirac} and the energy ground state of \eqref{Dirac_e}. Let \(EG_c\) be the set of all energy ground states of \eqref{Dirac_e} with Lagrange multiplier \(\omega_c \in (0, mc^2)\), and \(AG_{\omega_c}\) the set of action ground states of \eqref{Dirac}.
  
\begin{theorem}\label{them:1.4}
  \normalfont
  Under the assumptions of Theorem \ref{them:1.1}, for each $c\in (c_0, + \infty)\backslash \Xi$, let $\omega_c$ be the Lagrange multiplier determined in \eqref{Dirac_e}. Then we have
  $$EG_c= AG_{\omega_c}.$$
\end{theorem}

Our relativistic approach to the energy ground state of \eqref{Dirac_e} offers a new framework for normalized Dirac wave functions. Extending \cite{CDGW24} and \cite{CN25}, we cover the full Sobolev subcritical range $p\in (2,3)$ (previously $p\in(2,8/3]$). For $p\in (8/3,3)$ we cannot reduce the energy to the whole \(E_c^+\cap\mathcal{S}\); instead, we restrict to a suitable open subset where the reduced functional is bounded below. The key to compactness of minimizing sequences is lowering the reduced functional’s minimal energy below \(mc^2\). Inspired by the nonrelativistic limit, we use the negative-energy ground state of the limit equation \eqref{laplace_e} and take $c$ large to achieve this. The concentration compactness principle then yields a local minimizer, which we further show is global.

% For the existence of energy ground state of \eqref{Dirac_e}, 
% unlike previous methods, our analysis starts from a relativistic perspective,
%  providing a new approach for studying normalized Dirac wave functions.
%   Compared to previous works \cite{CDGW24} and \cite{CN25}, 
%   our results cover all Sobolev subcritical cases, 
%   expanding the range of $p$ from $(2,8/3]$ ($L^2$-subcritical)
%    to $(2,3)$ (Sobolev subcritical).For the case where $p$ is
%     between $8/3$ and $3$, unlike \cite{CN25}, we cannot directly
%      reduce the energy functional to $E_c^+\cap\mathcal{S}$.
%       Instead, we have found a suitable open subset on $E_c^+\cap\mathcal{S}$ 
%       for reduction, and the reduced functional is bounded from below on this 
%       open subset.
% To obtain the compactness of the minimizing sequence of the reduced 
% functional on this open subset, the key step is to lower the minimal 
% energy of the reduced functional below $mc^2$. Unlike the approach in \cite{CN25},
% our method, inspired by the nonrelativistic limit, utilizes the energy ground state 
%  with negative energy of the limit equation \eqref{laplace_e}, and leverages the condition 
%  that $c$ is sufficiently large to lower the minimal energy of the reduced functional
%   below $mc^2$. We then use the concentration-compactness principle to establish the 
%   existence of a local minimizer. Furthermore, we have proven that this local minimizer
%    is a global minimizer. Whether this minimizer is unique remains an open question. 
  We establish an a priori relation between the ground-state energy of \eqref{Dirac_e} and that of \eqref{laplace_e} (Proposition \ref{mulp}). Because the energy functional for \eqref{Dirac_e} is not weakly lower semicontinuous (due to an attractive potential), standard nonrelativistic arguments (e.g. \cite{ES01}) do not apply. Using the relation from Proposition \ref{mulp}, we show that any energy ground state of \eqref{Dirac_e} gives a minimizing sequence for the functional \(\mathcal{I}^\infty\) on \(\mathcal{S}'\).  
  
   % In proving the existence of the energy ground state of \eqref{Dirac_e},
   %  we have established an a priori relationship between the ground state
   %   energy of \eqref{Dirac_e} and that of \eqref{laplace_e}, as detailed in Proposition \ref{mulp}. 
   %   Unlike the nonrelativistic limit of the ground state of Dirac-Fock 
   %   equation in \cite{ES01}, the energy functional corresponding to \eqref{Dirac_e}
   %   is not weakly lower semicontinuous due to the presence of an attractive 
   %   potential term. Utilizing the a priori relationship of the ground state 
   %   energy from Proposition \ref{mulp}, we demonstrate that the energy ground state
   %   of \eqref{Dirac_e} constitutes a minimizing sequence for functional $\mathcal{I}^\infty$ on $\mathcal{S}'$. 
   %   Combined with the concentration-compactness principle, we further prove
   %    the convergence of the energy ground state of \eqref{Dirac_e}.
To the author's knowledge, the nonrelativistic limit for action ground states of \eqref{Dirac} is new. Our strategy parallels the energy ground states. We first establish a priori relation between the ground-state energies of \eqref{Dirac} and \eqref{laplace}, then obtain convergence. Unlike the study of energy ground state, this requires a refined estimate of the distance between the (reduced) Nehari manifolds for \eqref{Dirac} and \eqref{laplace}.

% Regarding the nonrelativistic limit of the action ground state \eqref{Dirac}, 
% to the best of the author's knowledge, no relevant results
%  have been reported. Our approach is similar to the methodology used in 
%  proving the nonrelativistic limit of the  energy ground state of \eqref{Dirac_e}.
%   Specifically, we first establish an a priori relationship
%   between the ground state energy of \eqref{Dirac} and that of \eqref{laplace}, 
%   and then employ the concentration-compactness principle to derive the
%    convergence of the solutions. However, unlike the convergence of the
%     energy ground state of \eqref{Dirac_e}, this process requires a refined estimation
%      of the distance between the (generalized) Nehari manifolds corresponding to the \eqref{Dirac} and \eqref{laplace}.

 % {\color{red} why introduce a family of functionals. Previous methods need $L^2$ norm small, but (RKS) cannot ask that, so introduce a family of functional to overcome that problem.}

 % {\color{red} Dirac operator has no scaling invariant, so there is no Pohozaev manifold.}
 \medskip

\noindent{\bf Outline of the paper.} The paper is organized as follows. 
In Section \ref{section2}, we recall some basic facts and useful lemmas. 
In Section \ref{section3}, we establish the existence of energy ground states of \eqref{Dirac_e} and prove Theorem \ref{them:1.1}. 
In Section \ref{section4} and \ref{section5}, we prove the nonrelativistic limit of 
energy and action ground states of nonlinear Dirac equations.
In Section \ref{section6}, we discuss the equivalence of action and energy ground states.
\medskip

\noindent{\bf Notations.} Throughout this paper, we make use of the following notations.
\begin{itemize}
\item $  C  $ is some positive constant that may change from line to line;
 \item For any $R>0$, $B_R$ denotes the ball of radius $R$ centered at the origin;
    \item $\|\cdot\|_{L^q}$ denotes the usual norm of the space $L^q\left(\mathbb{R}^3, \mathbb{C}^M\right)$, $M\in \{2,4\}$;
    \item $\|\cdot\|_{H^s}$ denotes the usual norm of the space $H^s\left(\mathbb{R}^3, \mathbb{C}^M\right)$, $M\in \{2,4\}$;
\item %$(\cdot, \cdot)_{\mathbb{C}^4}$ 
$\langle \cdot,\cdot\rangle$
denotes the usual complex inner product in $\mathbb{C}^M$, $M\in \{2,4\}$;
\item  $a\lesssim  b$ means that $a \leq     C   b$;
\item $\alpha\cdot\nabla$ (or $\sigma\cdot\nabla$) means that $\sum\limits_{k=1}^3 \alpha_k\partial_k$ (or $\sum\limits_{k=1}^3 \sigma_k\partial_k$);
\item $\Re(\Im)$ stands for the real part (imaginary part) of a  complex valued function;
\item $o_n(1)$ (or $o_c(1)$)   denotes a quantity that tends to zero as $n\to \infty$ (or $c\to \infty$);
\item $P(u)$ denotes the $L^2$-normalization of $u$, defined as $P(u) = u / \|u\|_{L^2}$.
    \end{itemize}

\section{Preliminary Results}\label{section2}

In this section, we introduce our basic working space and variational setting. In addition, we will present some basic inequalities and lemmas required for the proofs of our main theorems. The following notations are used consistently throughout the entire paper: $\alpha\cdot\nabla=\sum\limits_{k=1}^3 \alpha_k\partial_k$, where $\partial_k=\frac{\partial}{\partial x_k}$,
and  $\alpha_1$, $\alpha_2$, $\alpha_3$, $\beta$ are $4\times 4$ Dirac matrices,
 \begin{align*}
\alpha_k=\begin{pmatrix}
0 &\sigma_k\\
\sigma_k &0
\end{pmatrix},\beta=\begin{pmatrix}
I_2 &0\\
0 &-I_2
\end{pmatrix},
\end{align*}
with
\begin{align*}
\sigma_1=\begin{pmatrix}
0 &1\\
1&0
\end{pmatrix}, \sigma_2=\begin{pmatrix}
0 &-i\\
i&0
\end{pmatrix}, \sigma_3=\begin{pmatrix}
1&0\\
0&-1
\end{pmatrix}.
\end{align*}
We denote $\mathcal{F}(u)$ or $\hat{u}$ the Fourier transform of $u$, which is defined by
$$
\hat{u}(\xi)=\frac{1}{(2 \pi)^{3/2}} \int_{\mathbb{R}^3} e^{-i \xi \cdot x} u(x) d x .
$$
For $u, v\in  H^{1/2}(\mathbb{R}^3,\mathbb{C}^4)$, the inner product in $H^{1/2}(\mathbb{R}^3,\mathbb{C}^4)$  is defined by
$$
(u, v)_{H^{1/2}}:=\Re \int_{\mathbb{R}^3} \sqrt{|\xi|^2+1}\langle\hat{u}(\xi), \hat{v}(\xi)\rangle d \xi.
 $$
For the free Dirac operator $
\mathscr{D}_c 
$, 
it is evident that it is self-adjoint on $L^2(\mathbb{R}^3,\mathbb{C}^4)$ with domain $\text{dom}(\mathscr{D}_c)\cong H^{1}(\mathbb{R}^3,\mathbb{C}^4)$ for any $c>0$, and we have
$$\sigma(\mathscr{D}_c-\omega)=(-\infty,-mc^2-\omega]\cup[mc^2-\omega,+\infty),$$  where $\sigma(\cdot)$ is the spectrum
of the linear operator.
Therefore, the Hilbert space $L^2(\mathbb{R}^3, \mathbb{C}^4)$ possesses the following orthogonal decomposition
$$
L^2(\mathbb{R}^3,\mathbb{C}^4)=L^{c, -}\oplus L^{c, +},
$$
where $\mathscr{D}_c$ is positive defined on $L^{c, +}$ and negative defined on $L^{c, -}$. %Let $|\mathscr{D}_c|$ denote the absolute value of $\mathscr{D}_c$ and $|\mathscr{D}_c|^{\frac{1}{2}}$ denote its square root.
Let $E_c$ be the completion of $\text{dom}(|\mathscr{D}_c|^{1/2})$ under the following inner product
$$
(u_1,u_2)_c:= \left(|\mathscr{D}_c|^{1/2}u_1, |\mathscr{D}_c|^{1/2}u_2\right)_{L^2},
$$
The induced norm is denoted by $\| u\|_c:=(u,u)_c^{1/2}$. 
Moreover, we have
$$
\|u\|_c^2 = \int_{\mathbb{R}^3}\sqrt{c^2|\xi|^2 + m^2c^4}|\hat{u}(\xi)|^2 d\xi.
$$
It is clear that $E_c\cong H^{1/2}(\mathbb{R}^3,\mathbb{C}^{4})$ for any $c>0$. In fact, for any $u\in E_c$, we have 
$$
mc^2\|u\|_{L^2}^2\leq \| u\|_c^2, \quad c\|(-\Delta)^{\frac{1}{4}}u\|_{L^2}^2\leq \|u\|_c^2
$$
and
\[\min\{mc^2,c\}\|u\|_{H^{1/2}}^2  \leq  \| u\|_c^2\leq \max\{mc^2,c\}\|u\|_{H^{1/2}}^2 .\]
It is clear that the linear space $E_c$ possesses the following decomposition
$$
E_c=E_c^{-} \oplus E_c^{+}, \quad \text { where } E_c^{\pm}:=E_c \cap L^{c, \pm}.
$$
Denote $P_c^\pm=\frac{1}{2}\left(I \pm \frac{\mathscr{D}_c}{\left|\mathscr{D}_c\right|}\right)$ the orthogonal projections on $E_c$ with kernel $E_c^\mp$. 
In the Fourier domain, for each $\xi \in \mathbb{R}^3$, the symbol matrix
$$
\hat{\mathscr{D}}_c(\xi):=   \mathcal{F} \mathscr{D}_c \mathcal{F}^{-1} =\left(\begin{array}{cc}
m c^2 I_2 & c \sigma \cdot \xi \\
c \sigma \cdot \xi & -m c^2 I_2
\end{array}\right)
$$
  is a Hermitian $4\times 4$-matrix with eigenvalues
$\pm \sqrt{m^2 c^4+c^2|\xi|^2}$.
It follows from a direct calculation that the unitary transformation $\mathbf{U}(\xi)$  diagonalizing $\hat{\mathscr{D}_c}(\xi)$
is given  by
$$
\mathbf{U}(\xi)=\frac{\left(m c^2+\lambda_c(\xi)\right) I_4+\beta c \alpha \cdot \xi}{\sqrt{2 \lambda_c\left(m c^2+\lambda_c(\xi)\right)}}=\Upsilon_{+} I_4+\Upsilon_{-} \beta \frac{\alpha \cdot \xi}{|\xi|},
$$
where $\lambda_c(\xi)=\sqrt{m^2 c^4+c^2|\xi|^2}, $ $\Upsilon_{\pm}=\sqrt{\frac{1}{2}\left(1 \pm m c^2 / \lambda_c(\xi)\right)}$.
Consequently, it is easy to see that
$$
\mathbf{U}^{-1}(\xi)=\frac{\left(m c^2+\lambda_c\right) I_4-\beta c \alpha \cdot \xi}{\sqrt{2 \lambda_c\left(m c^2+\lambda_c(\xi)\right)}}=\Upsilon_{+} I_4-\Upsilon_{-} \beta \frac{\alpha \cdot \xi}{|\xi|},
$$
$$
\mathbf{U}(\xi) \hat{\mathscr{D}_c}(\xi) \mathbf{U}^{-1}(\xi)=\lambda_c \beta.
$$
Therefore, we have 
\begin{align}\label{proj}
\widehat{P_c^{\pm}u}(\xi)
=\frac{1}{2}\mathbf{U}^{-1}(\xi)(I_4\pm \beta)\mathbf{U}(\xi)\hat{u}(\xi)=\frac{1}{2}\left(I_4\pm\frac{mc^2}{\lambda_c}\beta\pm\frac{c}{\lambda_c}\alpha\cdot \xi\right)\hat{u}(\xi).
\end{align}
Let us recall  the Foldy-Wouthuysen transformation is defined by $\mathbf{U}_{FW}:= \mathcal{F}^{-1}\mathbf{U}\mathcal{F}$,  which transforms
the free Dirac operator into the $2\times 2$-block form
$$
\mathbf{U}_{\mathrm{FW}} \mathscr{D}_c \mathbf{U}_{\mathrm{FW}}^{-1}=\left(\begin{array}{cc}
\sqrt{-c^2 \Delta+m^2 c^4}I_2 & 0 \\
0 & -\sqrt{-c^2 \Delta+m^2 c^4}I_2
\end{array}\right)=\beta\left|\mathscr{D}_c\right| .
$$

Since $P_c^\pm=\frac{1}{2}\left(I \pm \frac{\mathscr{D}_c}{\left|\mathscr{D}_c\right|}\right)$, then the following commutation relation between the Dirac operator and the projection operator is clear, which will be used extensively throughout the paper.
\begin{lemma}
    For $u\in  H^{1}(\mathbb{R}^3,\mathbb{C}^4)$, there
    holds $$\mathscr{D}_cP_c^\pm u = \pm |\mathscr{D}_c|P_c^\pm u. $$
\end{lemma}

It naturally arises to ask, what is the relationship between \( E_{c'}^\pm \) and \( E_c^\pm \) for different values of \( c' \) and $c $ ? For convenience, we set $c'=1$.
\begin{lemma}\label{scal}
    \( E_c^\pm \) and \( E_1^\pm \) are isometrically isomorphic under the following scaling transformation
    $$  T_c(u)(x)=  c^{-1/2}u(c^{-1}x),\quad \forall u\in E_c.$$
    \begin{proof}

For $u\in E_c$, we have
\begin{equation}
    \begin{split}
        \widehat{P_1^\pm T_c(u)}(\xi)&= 
        \frac{c^{-1/2}}{2}\left(I_4\pm\frac{m}{\lambda_1}\beta\pm\frac{1 }{\lambda_1}\alpha\cdot \xi\right)\widehat{T_c(u)}(\xi)\\
        &= \frac{c^{5/2}}{2}\left(I_4\pm\frac{mc^2}{\lambda_c(c\xi)}\beta\pm\frac{c }{\lambda_c(c\xi)}\alpha\cdot (c\xi)\right)\hat{u}(c\xi)\\
        &={c^{5/2}}\widehat{P_c^\pm u}(c\xi)
    \end{split}
\end{equation}
Hence, 
$$
(P_1^\pm T_c(u))(x) = c^{-1/2}(P_c^\pm u) (c^{-1}x)= T_c(P_c^{\pm}(u))(x),
$$
which implies $T_c(u)\in E_1^\pm$ if and only if $u\in E_c^\pm$. In addition, we have
\begin{equation}
    \begin{split}
        \|T_c(u)\|_1^2 &=\int_{\mathbb{R}^3}\sqrt{|\xi|^2+ m^2}|\widehat{T_c(u)}|^2(\xi) d \xi
        \\
&=\int_{\mathbb{R}^3}\sqrt{c^2|\xi|^2+ m^2c^4}|\widehat{u}|^2(\xi) d \xi\\
        &=\|u\|_c^2.
    \end{split}
\end{equation}
Therefore, we have that \( T_c \) is an isometric isomorphism from \( E_c^\pm \) to \( E_1^\pm \).    
    \end{proof}
\end{lemma}

In the standard model, a state $u\in E_c$ is a superposition of particles and its antiparticles, $u^+:=P_c^+u $ describes the state of Dirac fermions with positive energy, and $u^-:=P_c^-u$ describes its antiparticles with positive energy, which can cancel part of the energy of $u^+$. 
The projection can be extended to $L^q$ continuously as in the following Lemma; the proof can be found in the \cite{dongdingguo}. Inspired by Lemma \ref{scal}, we adopt a different approach from \cite{dongdingguo} to prove this lemma.

\begin{lemma}\label{com}
For each $q>1$, there holds $P_c^\pm\in \mathcal{L}(L^q; L^q)$, and 
$$
   \| P_c^\pm\|_{ \mathcal{L}(L^q; L^q)} = \| P_1^\pm\|_{ \mathcal{L}(L^q; L^q)}.
   $$
\end{lemma}
\begin{proof}
    A straightforward computation reveals that for $u\in E_c\cap L^q$, there holds
    $$
    \|T_c(u)\|_{L^q} = c^{\frac{6-q}{2q}} \|u\|_{L^q}.
    $$
    Then
    \begin{equation}
        \begin{split}
            \|P_c^\pm u\|_{L^q}&=c^{\frac{q-6}{2q}}\|T_cP_c^\pm u\|_{L^q}=
            c^{\frac{q-6}{2q}}\|P_1^\pm T_cu\|_{L^q}\\
            &\leq \| P_1^\pm\|_{ \mathcal{L}(L^q; L^q)}  c^{\frac{q-6}{2q}}
            \| T_cu\|_{L^q}\\
            &= \| P_1^\pm\|_{ \mathcal{L}(L^q; L^q)}\|u\|_{L^q},
        \end{split}
    \end{equation}
    which yields $\| P_c^\pm\|_{ \mathcal{L}(L^q; L^q)} \leq \| P_1^\pm\|_{ \mathcal{L}(L^q; L^q)}$.
    Similarly, we have $\| P_c^\pm\|_{ \mathcal{L}(L^q; L^q)} \geq \| P_1^\pm\|_{ \mathcal{L}(L^q; L^q)}$.
\end{proof}
In \cite[Throrem 1.3]{MR4091059}, the authors prove the following  $L^p$-estimates for  pseudo-relativistic operator $|\mathscr{D}_c|- mc^2 + 1$ by the Mikhlin-Hörmander multiplier theorem:  
 \begin{equation}\label{new_lp0}
  \| (|\mathscr{D}_c|- mc^2 + 1)^{-1} u\|_{W^{1,p}}\leq  C   \|u\|_{L^p}.
 \end{equation}
With minor modifications, we can obtain the $L^p$-estimates for the Dirac operator, 
which are useful for studying the nonrelativistic limit.
\begin{lemma}[$L^p$-estimates for Dirac operator  ]\label{new_lpp}
  There exists $c_0>0$, such that for $c>c_0,\,1<p<\infty$, there holds
  \begin{equation}\label{new_lp1}
    \| (\mathscr{D}_c- mc^2 + 1)^{-1} u\|_{W^{1,p}}\leq  C   \|u\|_{L^p},
  \end{equation}
 \begin{equation}\label{new_lp2}
  \|P_c^- (\mathscr{D}_c- mc^2 + 1)^{-1} u\|_{W^{1,p}}\leq \frac{ C  }{c}\|u\|_{L^p},
 \end{equation}
  and 
  \begin{equation}\label{new_lp3}
    \|P_c^- (\mathscr{D}_c- mc^2 + 1)^{-1} u\|_{L^p}\leq \frac{ C  }{c^2}\|u\|_{L^p}.
  \end{equation}
\begin{proof}
  The operators \({\sqrt{-\Delta+1}}\left(|\mathscr{D}_c|+ mc^2 - 1\right)^{-1}\) and
   \(\left({|\mathscr{D}_c|+ mc^2 - 1}\right)^{-1}\) are associated with the symbols  
\[
f(\xi)=\frac{\sqrt{|\xi|^2+1}}{\sqrt{c^2|\xi|^2+ m^2c^4}+mc^2-1}, \quad g(\xi)=\frac{1}{\sqrt{c^2|\xi|^2+ m^2c^4}+mc^2-1},
\]
respectively. A direct computation shows that for all multi-indices \(\alpha\) with \(0 \leq |\alpha| \leq 2\), the following estimates hold for all \(\xi \in \mathbb{R}^{n} \setminus \{0\}\):
\[
|\nabla^{\alpha}f(\xi)| \leq \frac{C}{c|\xi|^{|\alpha|}}, \quad |\nabla^{\alpha}g(\xi)| \leq \frac{C}{c^2|\xi|^{|\alpha|}},
\]
where \(C\) is a constant independent of \(\xi\) and \(c\). By Lemma \ref{com} and
Mikhlin–Hörmander multiplier theorem (see Theorem 5.2.7 in \cite{MR3243734}),
there holds
\begin{equation*}
  \begin{split}
    \|P_c^- (\mathscr{D}_c- mc^2 + 1)^{-1} u\|_{W^{1,p}}
    &=\| (|\mathscr{D}_c|+ mc^2 - 1)^{-1} u^-\|_{W^{1,p}}\\
    &\leq \frac{ C  }{c}\|u^-\|_{L^p}\\
    & \leq \frac{ C  }{c}\|u\|_{L^p},
  \end{split}
\end{equation*}
and 
\begin{equation*}
  \begin{split}
    \|P_c^- (\mathscr{D}_c- mc^2 + 1)^{-1} u\|_{L^p}
    &=\| (|\mathscr{D}_c|+ mc^2 - 1)^{-1} u^-\|_{L^p}\\
    &\leq \frac{ C  }{c^2}\|u^-\|_{L^p}\\
    & \leq \frac{ C  }{c^2}\|u\|_{L^p},
  \end{split}
\end{equation*}
Combining \eqref{new_lp0} with \eqref{new_lp2}, we obtain 
\begin{equation*}
  \begin{split}
    \| (\mathscr{D}_c- mc^2 + 1)^{-1} u\|_{W^{1,p}}&\leq \| P_c^+(\mathscr{D}_c- mc^2 + 1)^{-1} u\|_{W^{1,p}}+ \| P_c^-(\mathscr{D}_c- mc^2 + 1)^{-1} u\|_{W^{1,p}}\\
    & = \| (|\mathscr{D}_c|- mc^2 + 1)^{-1} u^+\|_{W^{1,p}} + \| P_c^-(\mathscr{D}_c- mc^2 + 1)^{-1} u\|_{W^{1,p}}\\
    &\leq  C   \|u\|_{L^p}.
  \end{split}
\end{equation*}
This ends the proof.
\end{proof}
\end{lemma}
We recall that the projection operators onto the first two components and the last two components are given by 
    $$
    P_\infty^+= \frac{1}{2}(I_4 + \beta)=\operatorname*{diag}\{ 1,1,0,0\}, \quad P_\infty^- =\frac{1}{2}(I_4 - \beta) =I_4 - P_\infty^+.
    $$
The following lemma implies that, for a given function $f\in H^s(\mathbb{R}^3, \mathbb{C}^4)$, its positive (negative) part will converge to its first (last) two components in a lower Sobolev norm  as the speed of light converges to infinity.  It is worth mentioning that $H^2$ boundedness cannot imply $H^2$ convergence, but $H^{2+\varepsilon}$ can.

\begin{lemma}\label{nar}
For each $s\geq 0$, $\varepsilon\geqslant0$, there holds 
$$
      \|P_c^\pm - P_\infty^\pm\|_{\mathcal{L}(H^{s+\varepsilon}; H^{s})} \lesssim \frac{1}{c^{\min\{1, \varepsilon\}}}.
      $$
\end{lemma}
\begin{proof}
     For $f,g\in H^{s+\varepsilon}(\mathbb{R}^3, \mathbb{C}^2)$,
     $u=(f, g)^T\in H^{s+\varepsilon}(\mathbb{R}^3, \mathbb{C}^4)$, we have
    \begin{equation}
            \left\|P_c^+ u - P_\infty^+u \right\|_{H^s}^2\lesssim
            \left\|(P_c^+ - I_4)P_\infty^+ u\right\|_{H^s}^2
            +
            \left\|P_c^+ P_\infty^- u\right\|_{H^s}^2,
    \end{equation}
    and in view of \eqref{proj}, we have
    \begin{equation}
        \begin{split}
             \left\|(P_c^+ - I_4) P_\infty^+ u\right\|_{H^s}^2
             =&\frac{1}{4}\left\| (1+|\xi|^2)^{\frac{s}{2}}\left(1-\frac{mc^{2}}{\lambda_c(\xi)}\right)\hat{f}(\xi)\right\|_{L^{2}}^{2}\\
             &+\frac{1}{4}\left\|(1+|\xi|^2)^{\frac{s}{2}}\frac{c}{\lambda_c(\xi)}\xi\cdot\sigma\hat{f}(\xi)\right\|_{L^{2}}^{2}.
        \end{split}
    \end{equation}
  For $0\leq \varepsilon\leq 1$,  From the elementary inequality $a^2 + b^2\geqslant
    C_\varepsilon|a|^\varepsilon|b|^{2-\varepsilon}$, we are led to
$$ |1-\frac{mc^2}{\lambda_c(\xi)}|\leq \frac{c^2|\xi|^2}{c^2|\xi|^2+ m^2c^4}\leq   C  _{m, \varepsilon} \frac{|\xi|^\varepsilon}{c^\varepsilon},\quad\frac{c}{\lambda_c(\xi)}\leq   C  _{m,\varepsilon} \frac{1}{{c^\varepsilon|\xi|^{1-\varepsilon}}},$$ then
\begin{equation}
    \begin{split}
         \left\|(P_c^+ - I_4) P_\infty^+ u\right\|_{H^s}^2\lesssim
         \frac{1}{c^{2\varepsilon}}\| (1+|\xi|^2)^{\frac{s}{2}+\frac{\varepsilon}{2}} \hat{f}(\xi\|_{L^2}^2\lesssim
         \frac{1}{c^{2\varepsilon}} \|f\|_{H^{s+\varepsilon}}^2.
    \end{split}
\end{equation}
Similarly, 
$$
\left\|P_c^+ P_\infty^- u\right\|_{H^s}^2\lesssim\frac{1}{c^{2\varepsilon}} \|g\|_{H^{s+\varepsilon}}^2,
$$
which yields the conclusion.
    \end{proof}
 \begin{remark}
Lemma \ref{nar} indicates that if \( u \in H^2 \) satisfies \( P_{\infty}^{-}u = 0 \), then
\begin{equation}\label{rem_stb}
    \|u\|_{H^1}^2 = \|u^+\|_{H^1}^2 + \|u^-\|_{H^1}^2 = \|u^+\|_{H^1}^2 + \mathcal{O}\left(\frac{1}{c^2}\right).
\end{equation}
 \end{remark}
%
%
%参考文献
We recall the following Sobolev inequality involving fractional derivatives; for further details, we refer the reader to \cite{MR1947703}.

\begin{lemma}[\textbf{Sobolev Inequality}]\label{SB}
Let $2 \leq p < \infty$. Then there exists a constant $C > 0$ such that for
$u\in H^{\frac{3}{2}-\frac{3}{p}}$, 
\[
\|u\|_{L^p} \leq C \left\|(-\Delta)^{\frac{3}{4} - \frac{3}{2p}} u \right\|_{L^2}.
\]
\end{lemma}

Using Lemma \ref{SB} and interpolation inequality, we obtain the following Gagliardo--Nirenberg type inequality.

\begin{lemma}[\textbf{Gagliardo--Nirenberg Inequality}]\label{GN}
Let $2 \leq p < \infty$. Then there exists a constant $C > 0$ such that
\[
\|u\|_{L^p} \leq C \left\|(-\Delta)^{\frac{m}{2}} u \right\|_{L^2}^\theta \|u\|_{L^2}^{1-\theta}
\]
holds for $m \in \mathbb{R}^+$ and $\theta \in [0,1]$ satisfying
\[
\frac{1}{p} = \frac{1}{2} - \frac{m\theta}{3}.
\]
\end{lemma}

For $s= \frac{3p-7}{6p-16}\in \left(1,\frac{3p-6}{6p-16}\right)$, we introduce the following notations to overcome the difficulties from the $L^2$-supercritical nonlinearity.
$$\mathscr{O}_c := \{
u\in E_c: \|u\|_{L^2}\leq 1, \|u\|_c < c^s 
\}, 
\mathcal{O}_c := \{
u\in \mathscr{O}_c: \|u\|_{L^2}=1 
\}, 
\mathcal{O}_c^+:= \mathcal{O}_c \cap E_c^+.
$$

\begin{lemma}\label{loc_ineq}
    For $u, v\in E_c$, $ \|u\|_c < c^s$, then
    $$
    \int_{\mathbb{R}^3}|u|^{p-2}|v|^2\leq C c^{-{1}/{2}}\|u\|_{L^2}^{6-2p} \|v\|_c^2.
    $$
\end{lemma}
\begin{proof}
    By H\"older and Gagliardo-Nirenberg inequalities, we have
    \begin{equation*}
        \begin{split}
            \int_{\mathbb{R}^3}|u|^{p-2}|v|^2 dx
                  &\leq \|u\|_{L^{3p-6 }}^{p-2}\|v\|_{L^3}^2\\
                  &\leq C
                  \|(-\Delta)^{\frac14}u\|_{L^2}^{3p-8} \|u\|_{L^2}^{6-2p}    \|(-\Delta)^{\frac14}v\|_{L^2}^2\\
                  &\leq C c^{-(3p-6)/2}
                  \|u\|_c^{3p-8}\cdot  \|u\|_{L^2}^{6-2p}   \cdot \|v\|_c^2\\
                  & \leq C c^{-1/2}\|u\|_{L^2}^{6-2p} \|v\|_c^2.
        \end{split}
    \end{equation*}
\end{proof}

The following modified Gagliardo-Nirenberg inequality
was essentially proved in \cite[Proposition 2]{MR4430585}, see also previous work by Bellazzini, Georgiev and Visciglia \cite{Bella18}.
\begin{lemma}\label{modified}
  For $p\in (8/3, 3)$, there exists a constant $C>0$, such that
  $$
  \| u\|_{L^p}^p \leq C \left( \left\|u\right\|_{L^2}^{\frac{6-p}{2}}
  (\|u\|_c^2- mc^2\|u\|_{L^2}^2)^{\frac{3p-6}{4}}
   + c^{-\frac{3p-6}{2}}\|u\|_{L^2}^{6-2p} (\|u\|_c^2- mc^2\|u\|_{L^2}^2)^{\frac{3p-6}{2}}\right).
  $$
  In addition, for $u\in \mathcal{O}_c$, there holds
  $$
  \| u\|_{L^p}^p \leq C \left( (\|u\|_c^2- mc^2\|u\|_{L^2}^2)^{\frac{3p-6}{4}} + c^{-1/2}(\|u\|_c^2- mc^2\|u\|_{L^2}^2)\right).
  $$
\end{lemma}

Now, we consider the following nonlinear Dirac equations with a fixed frequency:
 \begin{equation}\label{Dirac_fix}
  \mathscr{D}_cu  - \omega_cu-|u|^{p-2}u=0,
  \end{equation}
  which is useful for proving the existence and nonrelativistic limit of energy ground state of \eqref{Dirac_e}.
  For the fixed constant $C_i>0$ $(i=1,2,3)$, set
 \begin{equation*}
     \begin{split}
         \mathscr{U}_c(C_1, C_2, C_3):=\{&u\in E_c:
         \text{ there exist} \,\,\omega_c\in (mc^2- C_1, mc^2- C_2), \\&\text{ such that}\, u\,\, \text{is a weak solution of }\, \,\eqref{Dirac_fix}\,\, \text{and}\,\,\|u\|_{H^{1/2}}\leq C_3\}.
      \end{split}
 \end{equation*}

\begin{proposition}\label{bound}  For the fixed constant $C_i>0$ $(i=1,2,3)$,
  there exists $c_0>0$, such that for $c>c_0$ and $u\in \mathscr{U}_c(C_1, C_2, C_3)$,
  the following holds.
  \begin{enumerate}
    \item [(1)] $\|u^+\|_{L^2}^2>\|u^-\|_{L^2}^2$.
    \item [(2)] For each $q>1$, there exists a constant $C$ which is  dependent on the $p, q, C_1, C_2, C_3$, such that
    $$
   \| u\|_{W^{2,q}} \leq C.$$
   \item[(3)] For each $q>1$, there exists a constant $C$ which is  dependent on the $p, q, C_1, C_2, C_3$, such that
   $$
   \| u^-\|_{W^{1,q}}\leq\frac{C}{c^2}, \quad \| u^-\|_{W^{2,q}}\leq\frac{C}{c}.
   $$
   \item[(4)] $u/\|u\|_{L^2} \in \mathcal{O}_c$ and 
 $u^+/\|u^+\|_{L^2} \in \mathcal{O}_c^+.$
 \item [(5)] Let $\mathcal{I}^c_{\omega_c}$ be the action functional corresponding to \eqref{Dirac_fix}, then, there exist a constant $C>0$, such that
 $\mathcal{I}^c_{\omega_c}(u)\geq C.$
  \end{enumerate}
\end{proposition}
\begin{proof}
  \begin{enumerate}
    \item [(1)] By multiplying both sides of \eqref{Dirac_fix} by $u^+ - u^-$
    and integrate, and applying Lemma \ref{loc_ineq}, we obtain
      \begin{equation*}
        \begin{split}
          &\omega_c(\|u^+\|_{L^2}^2-\|u ^-\|_{L^2}^2 ) = \|u \|_{c}^2 - \int_{\mathbb{R}^3}|u |^{p-2}\Re(u , u ^+- u ^-)\\
          &\geq \|u \|_{c}^2 - C\|u \|_{L^p}^p \geq (1-Cc^{-1/2})\|u\|_c ^2 >0.
        \end{split}
      \end{equation*}
      which yields $\|u^+\|_{L^2}^2>\|u^-\|_{L^2}^2$ for large $c>0$.
    \item [(2)] By Hölder interpolation inequality, we have
    $$
    \|u\|_{L^{2p-2}}\leq \|u\|_{L^3}^{\frac{4-p}{p-1}}\|u\|_{L^6}^{\frac{2p-5}{p-1}},
    $$
   and according to Lemma \ref{GN}, we obtain
   \begin{equation*}
       \begin{split}
           \|\mathscr{D}_c u\|_{L^2}^2= m^2 c^4\|u\|_{L^2}^2 &+
           c^2\|\nabla u\|_{L^2}^2=\omega_c^2\|u\|_{L^2}^2 +2\omega_c\|u\|_{L^p}^p + \|u\|_{L^{2p-2}}^{2p-2}\\
           &\leq m^2 c^4\|u\|_{L^2}^2 + C mc^2\| \nabla  u\|_{L^2}^{\frac{3p-6}{2}} + C\|\nabla u\|_{L^2}^{4p-10},
       \end{split}
   \end{equation*}
   Hence $\|\nabla u\|_{L^2}\leq C$.  
    By the Sobolev embedding \(H^1(\mathbb{R}^3) \hookrightarrow L^{3(p-1)}(\mathbb{R}^3)\) and Lemma \ref{new_lpp},
     we obtain
    \[
    \|u\|_{W^{1,3}} \leq C \left\| |u|^{p-2}u \right\|_{L^3} \leq C.
    \]
    Furthermore, using the Sobolev embedding \(W^{1,3}(\mathbb{R}^3) \hookrightarrow L^q(\mathbb{R}^3)\) for any \(3 < q < \infty\), we conclude
    \[
    \|u\|_{L^q} \leq C.
    \]
    Thus, by Lemma \ref{new_lpp} and \eqref{Dirac_fix}, we get
    \[
    \|u\|_{W^{1,q}} \leq C \left\| |u|^{p-2}u \right\|_{L^q} \leq C,
    \]
    \[
    \|u\|_{W^{2,q}} \leq C \left\| |u|^{p-2}u \right\|_{W^{1,q}} \leq C.
    \]
    Using Sobolev embedding once more, we also have
    \[
    \|u\|_{L^\infty} + \|\nabla u\|_{L^\infty} \leq C.
    \]
    \item[(3)] By taking the operator $P_c^-$ to both sides of \eqref{Dirac_fix}, we obtain
    $$
    -(\mathscr{D}_c+\omega_c)u^-= P_c^-\left(|u|^{p-2}u \right),
    $$
    then, it follows from Lemma \ref{new_lpp}, we get
    $$
    \|u^-\|_{W^{1,q}}\leq \frac{C}{c^2}\left\| |u|^{p-2}u\right\|_{W^{1,q}} \leq \frac{C}{c^2},
    $$
    and
    $$
    \|u^-\|_{W^{2,q}}\leq \frac{C}{c}\left\| |u|^{p-2}u\right\|_{W^{1,q}} \leq \frac{C}{c}.
    $$
   \item[(4)] By Gagliardo-Nirenberg inequality, it follows that
   \begin{equation*}
       \begin{split}
           \|u \|_{c}^2&=\omega_c(\|u^+\|_{L^2}^2-\|u ^-\|_{L^2}^2 )+ \int_{\mathbb{R}^3}|u |^{p-2}\Re(u , u ^+- u ^-)\\
           &\leq mc^2\|u\|_{L^2}^2 + C \|u\|_{L^p}^p\\
           &\leq mc^2\|u\|_{L^2}^2 + C\|\Delta u\|_{L^2}^{\frac{3p-6}{4}}\|u\|_{L^2}^{\frac{6+p}{4}}\\
           &\leq (mc^2 + C)\|u\|_{L^2}^2 \\
           & < c^{2s}\|u\|_{L^2}^2,
       \end{split}
   \end{equation*}
   which yields $u/\|u\|_{L^2}\in \mathcal{O}_c$, similarly, 
   $u^+/\|u^+\|_{L^2}\in \mathcal{O}_c^+.$
   \item[(5)]
We first prove that $\|u\|_{H^{1/2}}$ is  bounded from below. Actually,
   $\omega_c\in (0,mc^2)$ shows that there exists $C>0$ which depends on $\omega$ and $c$, such that
   $$
   \|u\|_c^2-\omega\|u\|_{L^2}^2\geqslant C\|u\|_c^2.
   $$
   Since $d\mathcal{I}^{c}_{\omega_c}({u})[u^+- u^-]=0$, then there holds
   \begin{equation}\label{jams4}
       \begin{split}
        C\|u\|_c^2&\leq   \|u\|_c^2-\omega\|u\|_{L^2}^2\leq   \|u\|_c^2- \omega\left( \|u^+\|_{L^2}^2- 
          \|u^-\|_{L^2}^2\right)\\
          &=\frac{1}{2}d\mathcal{I}^c_{\omega_c}[u][u^+- u^-]\\
          &\leq  C_4\|u\|_c^{p},
       \end{split}
   \end{equation}
   which yields there exists $C>0$, such that
   $$
   \|u\|_c\geqslant C.
   $$
   If there exists a sequence of critical points for
   $\mathcal{I}^{c}_{\omega_c}$, which is denoted by $\{u_n\}$,  such that $\mathcal{I}^{c}_{\omega_c}(u_n)\to 0$, then
   \begin{equation*}
       \begin{split}
           \mathcal{I}^{c}_{\omega_c}(u_n)&=\mathcal{I}^{c}_{\omega_c}(u_n)-\frac{1}{2}d\mathcal{I}^{c}_{\omega_c}(u_n)[u_n]\\
           &=\frac{p-2}{p}\int_{\mathbb{R}^3}  |u_n|^p dx\to 0.
       \end{split}
   \end{equation*}
  By \eqref{jams4} and Lemma \ref{nar}, we have
   $$
   \|u_n\|_c^2\leq C d\mathcal{I}^{c}_{\omega_c}[u_n][u_n^+- u_n^-]\leq  C\|u\|_{L^p}^p\to 0,
   $$
   which is contradictory with the fact that $\|u\|_c$ is   bounded from below. 
  \end{enumerate}

\end{proof}
\begin{remark}
  \normalfont
  We remark that the decay rate of $1/c^2$ obtained for
   $\|u^-\|_{W^{1,q}}$ in Proposition \ref{bound} is optimal.
    However, for the decay rates of $\|u^-\|_{W^{2,q}}$,
     we are unable to establish that the optimal rate remains $1/c^2$. 
     This difficulty arises from the insufficient regularity of 
     the power-type nonlinearity $|u|^{p-2}u$. 
     If instead a Hartree-type nonlinearity of the form $(|x|^{-1} * |u|^2)u$ 
     is considered, we can prove that the optimal decay rate for Sobolev norms 
     of arbitrary order $s$, namely $\|u^-\|_{H^s}$, is indeed $1/c^2$.
\end{remark}
%要加上Gamma的条件，这样的问题很多

\medskip

\section{Existence of energy ground state}\label{section3}
In this section, we focus on proving the existence of energy ground state to \eqref{Dirac_e}.
We denote $\mathcal{S}^{c,\pm}:= \mathcal{S}\cap E_c^\pm.$
For $u\in H^{1/2}(\mathbb{R}^3,\mathbb{C}^4)$,
set
$$
A[u]:=\frac{2}{p}\int_{\mathbb{R}^3} |u|^p.
$$
Then, we introduce the family of
  functionals $\mathcal{I}^{c, \tau} : E_c\rightarrow \mathbb{R}$  related to \eqref{Dirac_e} :
$$
\begin{aligned}
  \mathcal{I}^{c, \tau}({u})=&\|u^+\|_c^2-\|u^-\|_c^2
-\tau^{\zeta} A[u],
\end{aligned}
$$
where $\zeta= 5-\frac{3p }{2}$.
For fixed $c>0$ and $u\in E_c$, we have $\mathcal{I}^{c,\tau}(u)$ is non-increasing with respect to $\tau\in (0,\infty)$. 
For the mass critical or subcritical case, i.e. $2<p\leq 8/3$, finding the critical points of the functional \( \mathcal{I}^{c,\tau} \) on \( \mathcal{S} \) is equivalent to finding the critical points of the reduced functional
$$
\mathcal{I}_{red}^{c,\tau}(u)= \sup_{\substack{w\in \operatorname{span}\{u\}\oplus E_c^-\\\|w\|_{L^2}=1}}\mathcal{I}^{c,\tau}(w), \quad
u\in \mathcal{S}^{c,+}.
$$
 Compared to \( \mathcal{I}^{c,\tau} \), the reduced functional is bounded from below on \(  \mathcal{S}^{c,+} \). Using the concentration compactness principle, we can find a minimizer of the functional \( \mathcal{I}_{red}^{c,\tau} \) on \( \mathcal{S}^{c,+} \). 
In the case of $L^2$-supercritical, it is not appropriate to directly reduce the functional onto \(  \mathcal{S}^{c,+} \), as the reduced functional is no longer bounded from below. Therefore, we need to find a suitable open set on \(  \mathcal{S}^{c,+} \) for the reduction.

%Next, we will use the reduction method to study the normalized solutions of  \eqref{Dirac}.

\subsection[]{Maximization problem}

In order to find the energy ground state of \eqref{Dirac_e}, we first reduce the functional onto the open set $\mathcal{O}_c^+$ by considering the maximization problem.
For any $w\in \mathcal{O}_c^+$ with $\|w\|_{L^2}^2=1 $, denote 
$$
S(w) =\{ u\in \mathcal{S}: u^+\in \operatorname{span}\{w\} \}.
$$
Our  first step is to maximize the functional $\mathcal{I}^{c, \tau }$ in the space $S(w)$. The tangent space of $S(w)$ at $u\in S(w)$ is given by $$
T_u(S(w)) =\{h\in \operatorname{span}\{w\}\oplus E_c^-:\Re(u,h)_{L^2}=0 \}.
$$
The projection of the gradient $d \mathcal{I}^{c, \tau }(u)$ on $T_u(S(w))$  is given by
$$
d \mathcal{I}^{c, \tau }|_{S(w)}(u)[ h] = d \mathcal{I}^{c, \tau}(u)[h]- 2\omega(u)\Re(u,h)_{L^2},\quad \forall h\in \operatorname{span}\{w\}\oplus E_c^-,
$$
where $\omega(u)\in \mathbb{R}$ is such that $d \mathcal{I}^{c, \tau }|_{S(w)}(u)\in T_u(S(w)).$
  Our first results are about the compactness of the Palais-Smale sequence of $\mathcal{I}^{c, \tau }$ on $  S(w)$.

  \begin{proposition}\label{prop:3.1}
    Let \( w \in \mathcal{O}_c^+ \), and suppose \( \{u_n\} \subset S(w) \) is a Palais–Smale sequence for \( \mathcal{I}^{c,\tau} \) on \( S(w) \) at level \( l \). Then:
    \begin{enumerate}
      \item[(1)] \( \{u_n\} \) is bounded in \( E_c \).
      \item[(2)] If \( l > 0 \) and \( c \) is sufficiently large, then
            \[
            \liminf_{n \to +\infty} \omega(u_n) > 0.
            \]
      \item[(3)] Under the assumptions in (2), the sequence \( \{u_n\} \) is precompact in \( H^{1/2}(\mathbb{R}^3, \mathbb{C}^4) \).
    \end{enumerate}
    \end{proposition}
    
    \begin{proof}
    \begin{enumerate}
      \item[(1)] The proof of (1) can be found in \cite[Proposition 3.2]{CDGW24}.
    
      \item[(2)] Note that
            \[
            \begin{split}
              \omega(u_n) &= \frac{1}{2} d\mathcal{I}^{c,\tau}(u_n)[u_n] + o_n(1) \\
                          &= \mathcal{I}^{c,\tau}(u_n) - \frac{\tau^{\zeta}(p - 2)}{p} \int_{\mathbb{R}^3} |u_n|^p + o_n(1).
            \end{split}
            \]
            Hence,
            \[
            \mathcal{I}^{c,\tau}(u_n) - \frac{\tau^{\zeta}(p - 2)}{p} \int_{\mathbb{R}^3} |u_n|^p \leq \omega(u_n) \leq \mathcal{I}^{c,\tau}(u_n).
            \]
            The boundedness of \( \{u_n\} \) in \( E_c \) implies that \( \omega(u_n) \) is bounded. If \( l > 0 \), then for large \( n \), we have \( \|u_n^+\|_c > \|u_n^-\|_c \). By Lemma~\ref{com} and Lemma~\ref{loc_ineq},
            \[
            \begin{split}
              \omega(u_n) \|u_n^+\|_{L^2}^2 
              &= \frac{1}{2} d\mathcal{I}^{c,\tau}(u_n)[u_n^+] + o_n(1) \\
              &= \|u_n^+\|_c^2 - \tau^{\zeta} \int_{\mathbb{R}^3} |u_n|^{p-2} \Re(u_n, u_n^+) \, dx + o_n(1) \\
              &\geq \|u_n^+\|_c^2 - \left( \int_{\mathbb{R}^3} |u_n|^p \right)^{\frac{p-1}{p}} \cdot \left( \int_{\mathbb{R}^3} |u_n^+|^p \right)^{\frac{1}{p}} \\
              &\geq \|u_n^+\|_c^2 - C \|u_n\|_{L^p}^p \\
              &\geq \left( \frac{1}{2} - C c^{-1/2} \right) \|u_n\|_c^2.
            \end{split}
            \]
            Therefore, \( \liminf\limits_{n \to +\infty} \omega(u_n) > 0 \).
    
      \item[(3)] By (1), after passing to a subsequence, \( u_n \rightharpoonup u \) in \( E_c \). Since \( \dim(\operatorname{span}\{u_n^+\}) = 1 \), we have \( u_n^+ \to u^+ \) in \( E_c \). Moreover, by (2),
            \[
            \begin{aligned}
              o_n(1) &= -\frac{1}{2} d\mathcal{I}^c(u_n)[u_n^- - u^-] + \omega(u_n) \|u_n^- - u^-\|_{L^2}^2 \\
              &\geq \|u_n^- - u^-\|_c^2 + \tau^{\zeta} \int_{\mathbb{R}^3} |u_n|^{p-2} \Re(u_n, u_n^- - u^-) \, dx \\
              &\geq \|u_n^- - u^-\|_c^2 + \tau^{\zeta} \int_{\mathbb{R}^3} |u_n|^{p-2} |u_n^- - u^-|^2 \, dx \\
              &\quad - \tau^{\zeta} \int_{\mathbb{R}^3} |u_n|^{p-2} (|u^-| + |u_n^+|) |u_n^- - u^-| \, dx.
            \end{aligned}
            \]
            For any \( \varepsilon > 0 \), there exists \( R > 0 \) such that
            \[
            \int_{\mathbb{R}^3 \setminus B_R} |u^-|^3 \, dx < \varepsilon.
            \]
            Then for large \( n \),
            \[
            \begin{split}
              \int_{\mathbb{R}^3} |u_n|^{p-2} |u^-| |u_n^- - u^-| \, dx 
              &\leq \|u_n\|_{L^{3p-6}}^{p-2} \|u^-\|_{L^3(\mathbb{R}^3 \setminus B_R)} \|u_n^- - u^-\|_{L^3} \\
              &\quad + \|u_n\|_{L^{3p-6}}^{p-2} \|u^-\|_{L^3} \|u_n^- - u^-\|_{L^3(B_R)} = o_n(1).
            \end{split}
            \]
            Similarly,
            \[
            \int_{\mathbb{R}^3} |u_n|^{p-2} |u_n^+| |u_n^- - u^-| \, dx = o_n(1).
            \]
            Hence, \( \|u_n^- - u^-\|_{H^{1/2}} = o_n(1) \). This completes the proof of (3).\qedhere
    \end{enumerate}
    \end{proof}

Employing an analogous argument presented in \cite[Proposition 3.2]{Nolasco}, one can get the following results which imply  
the critical point of $ \mathcal{I}^{c, \tau }$ on $S(w)$ at positive levels is at a strict local maximum.
\begin{lemma}\label{lemm:3.1}
  Let \( w \in \mathcal{O}_c^+ \), and suppose \( u \in E_c \) is a critical point of \( \mathcal{I}^{c,\tau} \) on \( S(w) \) at a positive level, i.e.,
  \[
  d\mathcal{I}^{c,\tau}(u)[h] - 2\omega(u) \Re(u, h)_{L^2} = 0 \quad \forall h \in \operatorname{span}\{w\} \oplus E_c^-, \quad \text{and} \quad \mathcal{I}^{c,\tau}(u) > 0.
  \]
  Then
  \[
  d^2\mathcal{I}^{c,\tau}(u)[h, h] - 2\omega(u) \|h\|_{L^2}^2 < -(2 - C c^{-1/2}) \|h\|_c^2,
  \]
  and hence \( u \) is a strict local maximum of \( \mathcal{I}^{c,\tau} \) on \( S(w) \).
  \end{lemma}
  
  \begin{proof}
  Set \( u = a w + \eta \), where \( a = \sqrt{1 - \|\eta\|_{L^2}^2} \). Since \( \mathcal{I}^{c,\tau}(u) > 0 \), Proposition~\ref{prop:3.1} (2) implies \( \omega(u) > 0 \) and
  \[
  c^{2s} > \|w\|_c^2 \geq \|u^+\|_c^2 \geq \|\eta\|_c^2,
  \]
  so \( \eta \in \mathscr{O}_c \). For any \( h \in T_u S(w) \), write \( h = b w + \xi \) with \( b = a^{-1} \Re(\eta, \xi)_{L^2} \). Then
  \begin{equation}\label{prop3.3.1}
  \begin{aligned}
  d^2\mathcal{I}^{c,\tau}(u)[h, h]
  &= a^{-1} b \, d^2\mathcal{I}^{c,\tau}(u)[u, (b w - a^{-1} b \eta)] \\
  &\quad + 2 d^2\mathcal{I}^{c,\tau}(u)[b w, \xi] + a^{-2} b^2 d^2\mathcal{I}^{c,\tau}(u)[\eta, \eta] + d^2\mathcal{I}^{c,\tau}(u)[\xi, \xi],
  \end{aligned}
  \end{equation}
  where the second derivative is given by
  \begin{equation}\label{second_de}
  \begin{split}
  d^2\mathcal{I}^{c,\tau}(u)[h, k]
  &= 2 (h^+, k^+)_c - 2 (h^-, k^-)_c - 2\tau^{\zeta} \int_{\mathbb{R}^3} |u|^{p-2} \Re\langle h, k \rangle \, dx \\
  &\quad - 2(p - 2)\tau^{\zeta} \int_{\mathbb{R}^3} |u|^{p-4} \Re\langle u, h \rangle \cdot \Re\langle u, k \rangle \, dx.
  \end{split}
  \end{equation}
  From this, we obtain
  \begin{equation}
  \begin{split}
  d^2\mathcal{I}^{c,\tau}(u)[u, h]
  &= 2 (u^+, h^+)_c - 2 (u^-, h^-)_c - 2(p - 1)\tau^{\zeta} \int_{\mathbb{R}^3} |u|^{p-2} \Re\langle u, h \rangle \, dx \\
  &= 2 d\mathcal{I}^{c,\tau}(u)[h] - 2 (u^+, h^+)_c + 2 (u^-, h^-)_c \\
  &\quad - 2(p - 3)\tau^{\zeta} \int_{\mathbb{R}^3} |u|^{p-2} \Re\langle u, h \rangle \, dx.
  \end{split}
  \end{equation}
  
  By Lemma~\ref{loc_ineq}, there exists \( C > 0 \) such that
  \begin{equation}\label{prop3_3.3}
  \begin{split}
  &a^{-1} b \, d^2\mathcal{I}^{c,\tau}(u)[u, (b w - a^{-1} b \eta)] \\
  &= 4 b^2 \omega(u) \|w\|_{L^2}^2 - 4 \omega(u) b^2 a^{-2} \|\eta\|_{L^2}^2 - 2 b^2 \|w\|_c^2 - 2 b^2 a^{-2} \|\eta\|_c^2 \\
  &\quad - 2 b^2 (p - 3)\tau^{\zeta} \int_{\mathbb{R}^3} |u|^{p-2} |w|^2 \, dx + 2 a^{-2} b^2 (p - 3)\tau^{\zeta} \int_{\mathbb{R}^3} |u|^{p-2} |\eta|^2 \, dx \\
  &\leq 2 \omega(u) \|h\|_{L^2}^2 - 2 b^2 \|w\|_c^2 - 2 b^2 a^{-2} \|\eta\|_c^2 \\
  &\quad - 2 b^2 (p - 3)\tau^{\zeta} \int_{\mathbb{R}^3} |u|^{p-2} |w|^2 \, dx + 2 a^{-2} b^2 (p - 3)\tau^{\zeta} \int_{\mathbb{R}^3} |u|^{p-2} |\eta|^2 \, dx \\
  &\leq 2 \omega(u) \|h\|_{L^2}^2 - 2 \|h^+\|_c^2 + C c^{-1/2} \|h\|_c^2.
  \end{split}
  \end{equation}
  
  Again by Lemma~\ref{loc_ineq},
  \begin{equation}\label{prop3_3.4}
  \begin{split}
  2 d^2\mathcal{I}^{c,\tau}(u)[b w, \xi]
  &= - 4 \tau^{\zeta} \int_{\mathbb{R}^3} |u|^{p-2} \Re\langle b w, \xi \rangle \, dx \\
  &\quad - 4 (p - 2)\tau^{\zeta} \int_{\mathbb{R}^3} |u|^{p-4} \Re\langle u, b w \rangle \cdot \Re\langle u, \xi \rangle \, dx \\
  &\leq C c^{-1/2} \|h\|_c^2.
  \end{split}
  \end{equation}
  
  From \eqref{second_de},
  \begin{equation}\label{prop3_3.5}
  \begin{split}
  a^{-2} b^2 d^2\mathcal{I}^{c,\tau}(u)[\eta, \eta]
  &= - 2 a^{-2} b^2 \|\eta\|_c^2 - 2 \tau^{\zeta} a^{-2} b^2 \int_{\mathbb{R}^3} |u|^{p-2} |\eta|^2 \, dx \\
  &\quad - 2 (p - 2)\tau^{\zeta} a^{-2} b^2 \int_{\mathbb{R}^3} |u|^{p-4} |\langle u, \eta \rangle|^2 \, dx < 0,
  \end{split}
  \end{equation}
  and
  \begin{equation}\label{prop3_3.6}
  \begin{split}
  d^2\mathcal{I}^{c,\tau}(u)[\xi, \xi]
  &= - 2 \|\xi\|_c^2 - 2 \tau^{\zeta} \int_{\mathbb{R}^3} |u|^{p-2} |\xi|^2 \, dx \\
  &\quad - 2 (p - 2)\tau^{\zeta} \int_{\mathbb{R}^3} |u|^{p-4} |\langle u, \xi \rangle|^2 \, dx \\
  &\leq - 2 \|h^-\|_c^2.
  \end{split}
  \end{equation}
  
  Combining \eqref{prop3.3.1}--\eqref{prop3_3.6}, we conclude that
  \[
  d^2\mathcal{I}^{c,\tau}(u)[h, h] - 2 \omega(u) \|h\|_{L^2}^2 \leq - (2 - C c^{-1/2}) \|h\|_c^2.
  \]
  \end{proof}

For any $w \in \mathcal{O}_c^+$, we consider the following maximization problem
\begin{align}\label{111}
  \rho_\tau(w)=\sup _{{u} \in   S(w)} \mathcal{I}^{c, \tau }({u}) .
\end{align}
For $u\in \mathcal{O}_c$,
we define
$$ \mathcal{I}^{Pseudo, c, \tau}(u) = \|u\|_c^2- \tau^{\zeta} A[u] ,\quad
e^{Pseudo,c}(\tau) : = \inf_{u\in \mathcal{O}_c} \mathcal{I}^{Pseudo, c, \tau}(u).
$$
Then we have the following estimates on $ \rho_\tau(w)$ which imply $\rho_\tau(w) $
is bounded from below uniformly with respect to $w\in \mathcal{O}_c^+$.
\begin{lemma}\label{lemm:3.2}
  For any $w \in \mathcal{O}_c^{+}$, there holds
  $$
  e^{Pseudo,c}(\tau) \leq \rho_\tau(w) \leq \|w\|_c^2.
  $$
  In addition, there exists a constant $  C  >0$, such that $$ (1 -   Cc^{-1/2}  )\|w\|_{c}^2\leq \rho_\tau(w). $$
  %$  C  =  C  (m)>0$, such that 
%  $$
%c-  C   \leq (c-  C   )\| w \|_{H^{1/2}}^2\leq  \lambda_W  \leq \|w\|_{ c}^2 .
 % $$

\end{lemma}
\begin{proof}
%Since $w\in S^+$, there exists $  C  >0$, such that $\|w\|_{H^{1/2}}^{3pq-6}\leq   C  \|w\|_{H^{1/2}}^2$. Then by Corollary \ref{lemm:2.3} and Lemma \ref{lemm:2.4}, we have
%\[ B[w]\leq    C    \|(-\Delta)^{1/4}w\|_{L^2}^{3pq-6}\|w\|_{L^2}^{6-2pq}\leq   C  \|w\|_{H^{1/2}}^{3pq-6}\leq   C  \|w\|_{H^{1/2}}^{2},\]
%and
%\[A[w]\leq   C    \| (-\Delta)^{1/4}w\|_{L^2}^2\|{w}\|_{L^2}^{2s-2}\leq   C  \|w\|_{H^{1/2}}^{2}.\]
%Therefore, we get

  %\begin{align*}
%\lambda_W \geq  \mathcal{I}^{c, \tau }(w)&=\|w\|_c^2- 
%\tau A[w] - \tau^\kappa B[w]
%\\
  %& \geq  \|w\|_c^2 -   C   \| w \|_{H^{1/2}}^2\\
 % &\geq  (c-  C   )\| w \|_{H^{1/2}}^2\\
 % &\geq     c-  C   .
 % \end{align*}
It is clear that
\begin{equation*}
    \begin{split}
        \rho_\tau(w) \geqslant\mathcal{I}^{c,\tau}(w)&= \|w\|_c^2 - \tau^{\zeta}A[w] \\
        &\geq 
        \inf_{u\in \mathcal{O}_c} \left(\|u\|_c^2- \tau^{\zeta} A[u] \right)\\
        &=
e^{Pseudo,c}(\tau ).
    \end{split}
\end{equation*}
On the other hand, for any ${u} \in   S(w)$,we have $\mathcal{I}^{c, \tau }({u}) \leq \| u^+ \| _{ c}^2 \leq \|w\|_{   c}^2$, that means $ \rho_\tau(w)  \leq \|w\|_{  c}^2$.
Finally, it follows from Lemma \ref{loc_ineq}
that 
$$
\rho_\tau(w) \geq \|w\|_c^2 - \tau^{\zeta}A[w] 
\geq (1 -   Cc^{-1/2}  )\|w\|_{c}^2.
$$
\end{proof}

We finish the maximization problem in view of the following proposition.
\begin{proposition}\label{prop:3.2}
  For any \( w \in \mathcal{O}_c^+ \), there exists \( \varphi^{c,\tau}(w) \in S(w) \), such that
  \[
  \mathcal{I}^{c,\tau}\left(\varphi^{c,\tau}(w)\right) = \sup_{u \in S(w)} \mathcal{I}^{c,\tau}(u) = \rho_\tau(w).
  \]
  Moreover, the following hold:
  \[
  d\mathcal{I}^{c,\tau}(\varphi^{c,\tau}(w))[h] - 2\omega(\varphi^{c,\tau}(w)) \Re\left(\varphi^{c,\tau}(w), h\right)_{L^2} = 0, \quad \forall h \in \operatorname{span}\{w\} \oplus E_c^-,
  \]
  and
  \begin{enumerate}
    \item[(1)] \(\|\varphi^{c,\tau}(w)^+\|_{L^2}^2 - \|\varphi^{c,\tau}(w)^-\|_{L^2}^2 > 0\),
    \item[(2)] \(\tau^\zeta A[\varphi^{c,\tau}(w)] + \|\varphi^{c,\tau}(w)^-\|_c^2 \leq \tau^{\zeta} A[w]\),
    \item[(3)] Up to a phase factor, \( \varphi^{c,\tau}(w) \) is unique,
    \item[(4)] The map \( \varphi^{c,\tau} : v \in \mathcal{O}_c^+ \to \varphi^{c,\tau}(v) \) is smooth.
  \end{enumerate}
  \end{proposition}
  
  \begin{proof}
  By Lemma~\ref{lemm:3.2}, \( \rho_\tau(w) > 0 \). Ekeland's variational principle implies the 
  existence of a Palais–Smale maximizing sequence \( \{u_n\} \subset S(w) \) for \( \mathcal{I}^{c,\tau} \) 
  at a positive level. Hence, by Proposition \ref{prop:3.1}, there exists \( \varphi^{c,\tau}(w) \in S(w) \) such that \( \|u_n - \varphi^{c,\tau}(w)\|_c \to 0 \), and \( \omega(u_n) \to \omega(\varphi^{c,\tau}(w)) > 0 \). Therefore,
  \[
  \mathcal{I}^{c,\tau}(\varphi^{c,\tau}(w)) = \sup_{u \in S(w)} \mathcal{I}^{c,\tau}(u),
  \]
  and \( \| d\mathcal{I}|_{S(w)}(\varphi^{c,\tau}(w)) \| = 0 \).
  
  \begin{enumerate}
    \item[(1)] By Lemma~\ref{com} and Lemma~\ref{loc_ineq}, we have
    \[
    \begin{aligned}
      &2\omega(\varphi^{c,\tau}(w)) \left( \|\varphi^{c,\tau}(w)^+\|_{L^2}^2 - \|\varphi^{c,\tau}(w)^-\|_{L^2}^2 \right) \\
      &= d\mathcal{I}^{c,\tau}(\varphi^{c,\tau}(w))(\varphi^{c,\tau}(w)^+ - \varphi^{c,\tau}(w)^-) \\
      &= 2\|\varphi^{c,\tau}(w)\|_c^2 - 2\tau \int_{\mathbb{R}^3} |\varphi^{c,\tau}(w)|^{p-2} \Re\left( \varphi^{c,\tau}(w), \varphi^{c,\tau}(w)^+ - \varphi^{c,\tau}(w)^- \right) \\
      &\geq 2\|\varphi^{c,\tau}(w)\|_c^2 - C \|\varphi^{c,\tau}(w)\|_{L^p}^p \\
      &\geq (2 - C c^{-1/2}) \|\varphi^{c,\tau}(w)\|_c^2 > 0.
    \end{aligned}
    \]
    Hence, for large \( c > 0 \), \( \|\varphi^{c,\tau}(w)^+\|_{L^2}^2 - \|\varphi^{c,\tau}(w)^-\|_{L^2}^2 > 0 \).
  
    \item[(2)] Note that
    \[
    \begin{aligned}
      \|w\|_c^2 - \tau^{\zeta} A[w] 
      &= \mathcal{I}^{c,\tau}(w) \leq \mathcal{I}^{c,\tau}(\varphi^{c,\tau}(w)) \\
      &\leq \|w\|_c^2 - \|\varphi^{c,\tau}(w)^-\|_c^2 - \tau^\zeta A[\varphi^{c,\tau}(w)],
    \end{aligned}
    \]
    which implies
    \[
    \tau^\zeta A[\varphi^{c,\tau}(w)] + \|\varphi^{c,\tau}(w)^-\|_c^2 \leq \tau^{\zeta} A[w].
    \]
  
    \item[(3)] Suppose there exist two distinct maximizers \( \varphi_1^{c,\tau}, \varphi_2^{c,\tau} \in S(w) \). Define the set
    \[
    \Gamma= \left\{ \gamma : [0,1] \to S(w) \mid \gamma(0) = \varphi_1^{c,\tau},\ \gamma(1) = \varphi_2^{c,\tau} \right\},
    \]
    and the min-max level
    \[
    l = \sup_{\gamma \in \Gamma} \min_{t \in [0,1]} \mathcal{I}^{c,\tau}(\gamma(t)).
    \]
    Write
    \[
    \varphi_1^{c,\tau} = a(\eta_1) w + \eta_1, \quad \varphi_2^{c,\tau} = a(\eta_2) w + \eta_2,
    \]
    where \( \eta_1, \eta_2 \in \mathscr{O}_c \cap E_c^- \), and \( a(\eta_i) = \sqrt{1 - \|\eta_i\|_{L^2}^2} \). By (1), \( |a(\eta_i)|^2 > \frac{1}{2} \). For \( t \in (0,1) \), define
    \[
    \eta(t) = t \eta_2 + (1 - t) \eta_1, \quad \gamma(t) = a(\eta(t)) w + \eta(t) \in \Gamma.
    \]
    Then \( \gamma(t)^\pm \in \mathscr{O}_c \cap E_c^\pm \). By Lemma~\ref{loc_ineq}, \( \tau^{\zeta} A[\gamma(t)] \leq C c^{-1/2} \|\gamma(t)\|_c^2 \). Hence, for \( t \in (0,1) \),
    \[
    \begin{split}
      \mathcal{I}^{c,\tau}(\gamma(t))
      &= \|\gamma(t)^+\|_c^2 - \|\gamma(t)^-\|_c^2 - \tau^{\zeta} A[\gamma(t)] \\
      &\geq (1 - C c^{-1/2}) \|\gamma(t)^+\|_c^2 - (1 + C c^{-1/2}) \|\gamma(t)^-\|_c^2 \\
      &\geq \frac{1}{2}(1 - C c^{-1/2}) \|w\|_c^2 - (1 + C c^{-1/2}) \|\eta_1\|_c^2 \\
      &\quad + \frac{1}{2}(1 - C c^{-1/2}) \|w\|_c^2 - (1 + C c^{-1/2}) \|\eta_2\|_c^2.
    \end{split}
    \]
    From (2) and Lemma~\ref{loc_ineq}, we have
    \[
    \|\eta_i\|_c^2 \leq \tau^{\zeta} A[w] \leq C c^{-1/2} \|w\|_c^2,
    \]
    so
    \[
    \mathcal{I}^{c,\tau}(\gamma(t)) \geq (1 - C c^{-1/2}) \|w\|_c^2 > 0.
    \]
    In particular, \( l \geq \min\limits_{t \in [0,1]} \mathcal{I}^{c,\tau}(\gamma(t)) > 0 \). By the mountain pass theorem, there exists a mountain pass critical point \( u \in S(w) \) for \( \mathcal{I}^{c,\tau} \) with \( \mathcal{I}^{c,\tau}(u) > 0 \), contradicting Lemma~\ref{lemm:3.1}.
  
    \item[(4)] The smoothness of \( \varphi^{c,\tau} \) follows from a similar argument as in \cite[Proposition 3.6]{Nolasco}.
  \end{enumerate}
  \end{proof}

\subsection{Upper bound estimation for ground state energy}\label{appendix_1}
Based on the previous discussion, the reduced functional $\mathcal{I }_{red}^{c,\tau}  :\mathcal{O}_c^+ \rightarrow \mathbb{R}$  is well defined and it is expressed as 
$$
\mathcal{I }_{red}^{c,\tau}  (w)=\mathcal{I}^{c, \tau }\left(\varphi^{c, \tau}(w)\right)=\sup _{{u} \in   S(w)} \mathcal{I}^{c, \tau }({u})$$
The minimax problem has now been converted into a minimization problem. That is
$$
e^c(\tau):=\inf _{\substack{w \in \mathcal{O}_c^+  }} \sup _{{u} \in   S(w)} \mathcal{I}^{c, \tau }({u})=\inf _{w \in \mathcal{O}_c^{+}} \mathcal{I }_{red}^{c,\tau}  (w).
$$
 
The following proposition is critical in proving the
 existence of the energy ground state of \eqref{Dirac_e}. Additionally, this proposition reveals the convergence rate of the ground state energy associated with \eqref{Dirac_e}.
 
\begin{proposition}\label{mulp}
   There exist constants $c_0, C>0$, such that for $c>c_0$ and every  $\tau\in (0,1]$, there holds   $$e^c(\tau)\in\left[mc^2 - C,\,mc^2 + e^{\infty}(\tau)+ \frac{C \tau^2}{c^2}\right],$$
   where $e^{\infty}(\tau)$ is defined as the ground state energy of the functional 
   $$
\mathcal{I}^{\infty, \tau}(f) := \frac{1}{2m}\|\nabla f\|_{L^2}^2 - \tau^{\zeta} A[f], \quad f\in H^1(\mathbb{R}^3, \mathbb{C}^2) \,\,\text{or}\,\,f\in H^1(\mathbb{R}^3, \mathbb{C}^4),
$$
more precisely
   $$
   e^{\infty}(\tau):=\inf_{\substack{\|f\|_{L^2}=1}}\mathcal{I}^{\infty, \tau}(f).
   $$
\end{proposition}

\begin{lemma}\label{AA_1}
  $e^\infty(\tau)=\tau^2 e^\infty(1)<0$.
\end{lemma}
\begin{proof}
 For $\tau \in (0,1)$, we define $f_\tau(x) := \tau^{3/2} f(\tau x)$. A direct computation shows that
 \[
 \mathcal{I}^{\infty, \tau}(f_\tau) = \tau^2 \mathcal{I}^{\infty, 1}(f).
 \]
 Consequently, one has $e^\infty(\tau) = \tau^2 e^\infty(1)$. Furthermore, for any $f\in H^1(\mathbb{R}^3, \mathbb{C}^2)$ with $\|f\|_{L^2} = 1$ and for sufficiently small $\tau > 0$, it follows that
 \[
 e^\infty(1) \le \mathcal{I}^{\infty, 1}(f_\tau) < 0.
 \]
\end{proof}
Let \( f \in H^1(\mathbb{R}^3, \mathbb{C}^2) \) be a minimizer of \( e^{\infty}(1) \), and let \( \omega_{\infty,\tau} \) be the corresponding Lagrange multiplier. Define  
\begin{equation}\label{llp_d}
f_\tau(x) := \left( \tau^{3/2} f(\tau x),\ 0 \right)^T \in H^1(\mathbb{R}^3, \mathbb{C}^4), \quad f_{c,\tau} := \mathbf{U}_{\mathrm{FW}}^{-1} f_\tau \in \mathcal{S}^{c,+}. 
\end{equation}
 Since \(\mathbf{U}_{\mathrm{FW}}^{-1}\) is a unitary operator on \(H^s(\mathbb{R}^3; \mathbb{C}^4)\) for any \(s > 0\), it follows that for large \(c\), we have
  \(
  f_{c,\tau} \in \mathcal{O}_c^+.
  \)

\begin{lemma}\label{unig}
  For every $2 \leq r \leq 3$, the following estimate holds:
  \[
  \left\|\varphi^{c,\tau}(f_{c,\tau})\right\|_{L^r} \leq C \tau^{\frac{3r - 6}{2r}}.
  \]
  \end{lemma}
  
  \begin{proof}
  By the Sobolev inequality and the unitarity of $\mathbf{U}_{\mathrm{FW}}^{-1}$, we obtain
  \begin{equation}\label{stb_16}
  \begin{aligned}
  \left\|\left(\varphi^{c,\tau}(f_{c,\tau})\right)^+\right\|_{L^r} 
  &\leq  \left\|f_{c,\tau}\right\|_{L^r} \\
  &\leq C \left\|(-\Delta)^{\frac{3}{4} - \frac{3}{2r}} f_{c,\tau}\right\|_{L^2} \\
  &= C \left\|(-\Delta)^{\frac{3}{4} - \frac{3}{2r}} f_{\tau}\right\|_{L^2} \\
  &\leq C \tau^{\frac{3r - 6}{2r}}.
  \end{aligned}
  \end{equation}
  
  Moreover, Proposition~\ref{prop:3.2} (2) implies
  \begin{equation}\label{youyong}
  \begin{aligned}
  \left\|\left(\varphi^{c,\tau}(f_{c,\tau})\right)^-\right\|_{L^2}^2 &= \tau^2 \cdot \mathcal{O}\left(\frac{1}{c^2}\right), \\
  \left\|(-\Delta)^{1/4} \left(\varphi^{c,\tau}(f_{c,\tau})\right)^-\right\|_{L^2}^2 &= \tau^2 \cdot \mathcal{O}\left(\frac{1}{c}\right).
  \end{aligned}
  \end{equation}
  
  Using~\eqref{youyong} together with the Sobolev and interpolation inequalities, we find
  \begin{equation}\label{stb_17}
  \begin{aligned}
  \left\|\left(\varphi^{c,\tau}(f_{c,\tau})\right)^-\right\|_{L^r} 
  &\leq C \left\|(-\Delta)^{\frac{3}{4} - \frac{3}{2r}} \varphi^{c,\tau}(f_{c,\tau})^-\right\|_{L^2} \\
  &\leq C \left\|(-\Delta)^{1/4} \left(\varphi^{c,\tau}(f_{c,\tau})\right)^-\right\|_{L^2}^{3 - 6/r} 
  \cdot \left\|\left(\varphi^{c,\tau}(f_{c,\tau})\right)^-\right\|_{L^2}^{6/r - 2} \\
  &\leq \tau \cdot \mathcal{O}\left(\frac{1}{c^{\frac{6 - r}{2r}}}\right).
  \end{aligned}
  \end{equation}
  
  The desired result now follows by combining estimates~\eqref{stb_16} with \eqref{stb_17}:
  \[
  \left\|\varphi^{c,\tau}(f_{c,\tau})\right\|_{L^r} \leq C \tau^{\frac{3r - 6}{2r}}.
  \]
  \end{proof}

    For convenience, the Lagrange multiplier corresponding to $\varphi^{c,\tau}(f_{c, \tau})$ is denoted by  
    $\omega := \omega\left(\varphi^{c,\tau}(f_{c, \tau})\right)$.  
    \begin{lemma}\label{stb_1}
    The family $\left\{\left(\varphi^{c,\tau}(f_{c, \tau})\right)^-\right\}$ converges to 0 in $H^{1/2}(\mathbb{R}^3,\mathbb{C}^4)$ as $c \to \infty$. Moreover, one of the following estimates holds:
    \[
    \left\|\left(\varphi^{c,\tau}(f_{c, \tau})\right)^-\right\|_{L^2}^2 \leq \tau^4 \cdot \mathcal{O}\left(\frac{1}{c^4}\right), \quad
    \left\|(-\Delta)^{1/4}\left(\varphi^{c,\tau}(f_{c, \tau})\right)^-\right\|_{L^2}^2 \leq \tau^3 \cdot \mathcal{O}\left(\frac{1}{c^2}\right).
    \]
    \end{lemma}
    
    \begin{proof}
    From the identity
    \[
    d\mathcal{I}^{c,\tau}\left(\varphi^{c,\tau}(f_{c, \tau})\right)\left[\left(\varphi^{c,\tau}(f_{c, \tau})\right)^-\right] = \omega \left\|\left(\varphi^{c,\tau}(f_{c, \tau})\right)^-\right\|_{L^2}^2,
    \]
    we deduce
    \begin{equation*}
    \begin{aligned}
    \left\|\left(\varphi^{c,\tau}(f_{c, \tau})\right)^-\right\|_c^2 
    &= -\tau^{\zeta} \int_{\mathbb{R}^3} \left|\varphi^{c,\tau}(f_{c, \tau})\right|^{p-2} \Re\left( \varphi^{c,\tau}(f_{c, \tau}), \left(\varphi^{c,\tau}(f_{c, \tau})\right)^- \right) 
    - \omega \left\|\left(\varphi^{c,\tau}(f_{c, \tau})\right)^-\right\|_{L^2}^2 \\
    &\leq \tau^{\zeta} \int_{\mathbb{R}^3} \left|\varphi^{c,\tau}(f_{c, \tau})\right|^{p-1} \left|\left(\varphi^{c,\tau}(f_{c, \tau})\right)^-\right|.
    \end{aligned}
    \end{equation*}
    Now, by Lemma~\ref{unig}, for $p \leq \frac{5}{2}$, we obtain
    \begin{equation*}
    \begin{aligned}
    m c^2 \left\|\left(\varphi^{c,\tau}(f_{c, \tau})\right)^-\right\|_{L^2}^2 
    &\leq \left\|\left(\varphi^{c,\tau}(f_{c, \tau})\right)^-\right\|_c^2 \\
    &\leq \tau^{\zeta} \left\|\varphi^{c,\tau}(f_{c, \tau})\right\|_{L^{2p-2}}^{p-1} \left\|\left(\varphi^{c,\tau}(f_{c, \tau})\right)^-\right\|_{L^2} \\
    &\leq C \tau^2 \left\|\left(\varphi^{c,\tau}(f_{c, \tau})\right)^-\right\|_{L^2},
    \end{aligned}
    \end{equation*}
    which implies
    \[
    \left\|\left(\varphi^{c,\tau}(f_{c, \tau})\right)^-\right\|_{L^2}^2 \leq \tau^4 \cdot \mathcal{O}\left(\frac{1}{c^4}\right).
    \] 
    If instead $p \geq \frac{7}{3}$, then
    \begin{equation*}
    \begin{aligned}
    c \left\|(-\Delta)^{1/4}\left(\varphi^{c,\tau}(f_{c, \tau})\right)^-\right\|_{L^2}^2 
    &\leq \left\|\left(\varphi^{c,\tau}(f_{c, \tau})\right)^-\right\|_c^2 \\
    &\leq \tau^{\zeta} \left\|\varphi^{c,\tau}(f_{c, \tau})\right\|_{L^{3(p-1)/2}}^{p-1} \left\|\left(\varphi^{c,\tau}(f_{c, \tau})\right)^-\right\|_{L^3} \\
    &\leq \tau^{3/2} \left\|(-\Delta)^{1/4}\left(\varphi^{c,\tau}(f_{c, \tau})\right)^-\right\|_{L^2},
    \end{aligned}
    \end{equation*}
    and therefore
    \[
    \left\|(-\Delta)^{1/4}\left(\varphi^{c,\tau}(f_{c, \tau})\right)^-\right\|_{L^2}^2 \leq \tau^3 \cdot \mathcal{O}\left(\frac{1}{c^2}\right).
    \]
    \end{proof}

    \begin{lemma}\label{stb_7}
      The following estimates hold:
      \begin{align}
      \int_{\mathbb{R}^3} \left(|f_{c,\tau}|^{p-2} + |f_{\tau}|^{p-2} \right) \left| f_{c,\tau} - f_\tau \right|^2 &\leq \frac{C \tau^{\frac{3p-2}{2}}}{c^2}, \label{stb_5} \\
      \int_{\mathbb{R}^3} \left(|f_{c,\tau}|^{p-2} + |\varphi^{c,\tau}(f_{c, \tau})|^{p-2} \right) \left| f_{c,\tau} - \varphi^{c,\tau}(f_{c, \tau}) \right|^2 &\leq \frac{C \tau^{2}}{c^2}. \label{stb_6}
      \end{align}
      \end{lemma}
      
      \begin{proof}
      By the Hölder and Sobolev inequalities, we have
      \begin{equation*}
      \begin{aligned}
      I_1 &:= \int_{\mathbb{R}^3} |f_{c,\tau}|^{p-2} \left| f_{c,\tau} - f_\tau \right|^2 \\
      &\leq \|f_{c, \tau}\|_{L^p}^{p-2} \|f_{c,\tau} - f_\tau\|_{L^p}^2 \\
      &\leq C \left\|(-\Delta)^{\frac{3}{4} - \frac{3}{2p}} f_{c,\tau}\right\|_{L^2}^{p-2} \left\|(-\Delta)^{\frac{3}{4} - \frac{3}{2p}} (f_{c,\tau} - f_\tau)\right\|_{L^2}^2 \\
      &= C \left\|(-\Delta)^{\frac{3}{4} - \frac{3}{2p}} f_{\tau}\right\|_{L^2}^{p-2} \left\||\xi|^{\frac{3}{2} - \frac{3}{p}} (\mathbf{U}^{-1} - I_4) \widehat{f_\tau} \right\|_{L^2}^2 \\
      &= C \tau^{\frac{3(p-2)}{2p}} \left\|(-\Delta)^{\frac{3}{4} - \frac{3}{2p}} f\right\|_{L^2}^{p-2} \left\||\xi|^{\frac{3}{2} - \frac{3}{p}} (\mathbf{U}^{-1} - I_4) \widehat{f_\tau} \right\|_{L^2}^2.
      \end{aligned}
      \end{equation*}
      Using the estimate
      \[
      |\mathbf{U}^{-1}(\xi) - I_4|^2 \leq 2|1 - \Upsilon_+|^2 + 2|\Upsilon_-|^2 \leq \frac{C|\xi|^2}{c^2},
      \]
      we obtain
      \[
      I_1 \leq \frac{C}{c^2} \tau^{\frac{3(p-2)}{2p}} \left\||\xi|^{\frac{5}{2} - \frac{3}{p}} \widehat{f_\tau} \right\|_{L^2}^2 \leq \frac{C \tau^{\frac{3p-2}{2}}}{c^2},
      \]
      which implies that inequality~\eqref{stb_5} holds.
      
      Similarly, by Proposition~\ref{prop:3.2} (2), we have
      \begin{equation}\label{stb_2}
      \begin{aligned}
      &\int_{\mathbb{R}^3} |\varphi^{c,\tau}(f_{c, \tau})|^{p-2} \left| f_{c,\tau} - \varphi^{c,\tau}(f_{c, \tau}) \right|^2 \\
      &\quad\leq \|\varphi^{c,\tau}(f_{c, \tau})\|_{L^p}^{p-2} \| f_{c,\tau} - \varphi^{c,\tau}(f_{c, \tau}) \|_{L^p}^2 \\
      &\quad\leq \|f_{c,\tau}\|_{L^p}^{p-2} \left\|(-\Delta)^{\frac{3}{4} - \frac{3}{2p}} (f_{c,\tau} - \varphi^{c,\tau}(f_{c, \tau}))\right\|_{L^2}^2 \\
      &\quad\leq C \tau^{\frac{3(p-2)}{2p}} \left\|(-\Delta)^{\frac{3}{4} - \frac{3}{2p}} (f_{c,\tau} - \varphi^{c,\tau}(f_{c, \tau}))\right\|_{L^2}^2.
      \end{aligned}
      \end{equation}
      Combining~\eqref{youyong}, Lemma~\ref{stb_1}, and the interpolation inequality, we get
      \begin{equation}\label{stb_3}
      \begin{aligned}
      &\left\|(-\Delta)^{\frac{3}{4} - \frac{3}{2p}} \left(\varphi^{c,\tau}(f_{c, \tau})\right)^-\right\|_{L^2}^2 \\
      &\quad\leq \left\|(-\Delta)^{1/4} \left(\varphi^{c,\tau}(f_{c, \tau})\right)^-\right\|_{L^2}^{6 - 12/p} \left\|\left(\varphi^{c,\tau}(f_{c, \tau})\right)^-\right\|_{L^2}^{12/p - 4} \\
      &\quad\leq \frac{C \tau^2}{c^2},
      \end{aligned}
      \end{equation}
      and
      \begin{equation}\label{stb_4}
      \begin{aligned}
      &\left\|(-\Delta)^{\frac{3}{4} - \frac{3}{2p}} \left(f_{c,\tau} - \left(\varphi^{c,\tau}(f_{c, \tau})\right)^+\right)\right\|_{L^2}^2 \\
      &\quad= \left(1 - \sqrt{1 - \left\| \left(\varphi^{c,\tau}(f_{c, \tau})\right)^- \right\|_{L^2}^2} \right)^2 \left\|(-\Delta)^{\frac{3}{4} - \frac{3}{2p}} f_{\tau}\right\|_{L^2}^2 \\
      &\quad\leq C \tau^{\frac{3p-6}{p}} \left\| \left(\varphi^{c,\tau}(f_{c, \tau})\right)^- \right\|_{L^2}^4 \\
      &\quad\leq \frac{C \tau^{2}}{c^2}.
      \end{aligned}
      \end{equation}
      Putting together~\eqref{stb_2}, \eqref{stb_3}, and \eqref{stb_4}, we conclude that
      \[
      \int_{\mathbb{R}^3} |\varphi^{c,\tau}(f_{c, \tau})|^{p-2} \left| f_{c,\tau} - \varphi^{c,\tau}(f_{c, \tau}) \right|^2 \leq \frac{C \tau^{2}}{c^2},
      \]
      which proves inequality~\eqref{stb_6}.
      \end{proof}

      \begin{lemma}\label{new_b}
        For every $\tau \in (0,1]$, the following inequality holds:
        \[
        \mathcal{I}^{\infty,\tau}(f_\tau) \leq \mathcal{I}^{\infty,\tau}(f_{c,\tau}) \leq \mathcal{I}^{\infty,\tau}(f_\tau) + \frac{C\tau^2}{c^2}.
        \]
        \end{lemma}
        
        \begin{proof}
        It is sufficient to show that
        \[
        \mathcal{I}^{\infty,\tau}(f_{c,\tau}) \leq \mathcal{I}^{\infty,\tau}(f_\tau) + \frac{C\tau^2}{c^2}.
        \]
        By the triangle inequality, we decompose it as
        \begin{equation}\label{llp_0}
        \begin{aligned}
        \mathcal{I}^{\infty,\tau}(f_{c,\tau}) - \mathcal{I}^{\infty,\tau}(f_\tau) 
        &\leq \left| \mathcal{I}^{\infty,\tau}(f_{c,\tau}) - \mathcal{I}^{\infty,\tau}(f_\tau) - d\mathcal{I}^{\infty,\tau}(f_\tau)[f_{c,\tau} - f_\tau] \right| \\
        &\quad + \left| d\mathcal{I}^{\infty,\tau}(f_\tau)[f_{c,\tau} - f_\tau] \right|.
        \end{aligned}
        \end{equation}
        
        Note that the first-order term satisfies
        \begin{equation*}
        \begin{aligned}
        d\mathcal{I}^{\infty,\tau}(f_\tau)[f_{c,\tau} - f_\tau] 
        &= 2\omega_{\infty,\tau} \left( f_\tau,\ f_{c,\tau} - f_\tau \right)_{L^2} \\
        &= 2\omega_{\infty,\tau} \left( \widehat{f_\tau},\ (\mathbf{U}^{-1} - I_4) \widehat{f_\tau} \right)_{L^2} \\
        &= 2\omega_{\infty,\tau} \left( \widehat{f_\tau},\ (\Upsilon_+ - 1) \widehat{f_\tau} \right)_{L^2} - 2\omega_{\infty,\tau} \left( \widehat{f_\tau},\ \Upsilon_- \beta \frac{\alpha \cdot \xi}{|\xi|} \widehat{f_\tau} \right)_{L^2} \\
        &= 2\omega_{\infty,\tau} \left( \widehat{f_\tau},\ (\Upsilon_+ - 1) \widehat{f_\tau} \right)_{L^2},
        \end{aligned}
        \end{equation*}
        where the last equality due to the condition $P_\infty^- f_\tau = 0$. A direct computation shows
        \[
        1 - \Upsilon_+ = \frac{1 - \Upsilon_+^2}{1 + \Upsilon_+} \leq \frac{\sqrt{c^2|\xi|^2 + m^2c^4} - mc^2}{2\sqrt{c^2|\xi|^2 + m^2c^4}} \leq \frac{|\xi|^2}{4m^2c^2},
        \]
        which implies
        \begin{equation}\label{llp_1}
        \left| d\mathcal{I}^{\infty,\tau}(f_\tau)[f_{c,\tau} - f_\tau] \right| \leq \frac{C \| \nabla f_\tau \|_{L^2}^2}{c^2} \leq \frac{C\tau^2}{c^2}.
        \end{equation}
        On the other hand, using inequality \eqref{stb_5} and the estimate
        \[
        \left| |a|^p - |b|^p - p|b|^{p-2} \Re \langle b,\ a - b \rangle \right| \leq C_p \max\left\{ |a|^{p-2},\ |b|^{p-2} \right\} |a - b|^2,
        \]
        we obtain
        \begin{equation}\label{new_b1}
        \begin{aligned}
        &\mathcal{I}^{\infty,\tau}(f_{c,\tau}) - \mathcal{I}^{\infty,\tau}(f_\tau) - d\mathcal{I}^{\infty,\tau}(f_\tau)[f_{c,\tau} - f_\tau] \\
        &= \frac{1}{2m} \left( \int_{\mathbb{R}^3} |\nabla f_{c,\tau}|^2 - \int_{\mathbb{R}^3} |\nabla f_\tau|^2 - 2\Re \int_{\mathbb{R}^3} \langle \nabla f_\tau,\ \nabla f_{c,\tau} - \nabla f_\tau \rangle \right) \\
        &\quad - \frac{2\tau^\zeta}{p} \int_{\mathbb{R}^3} \left( |f_{c,\tau}|^p - |f_\tau|^p - p|f_\tau|^{p-2} \Re \langle f_\tau,\ f_{c,\tau} - f_\tau \rangle \right) \\
        &\leq C \| \nabla(f_{c,\tau} - f_\tau) \|_{L^2}^2 + \frac{C\tau^2}{c^2}.
        \end{aligned}
        \end{equation}
        From the estimate
        \[
        |\mathbf{U}^{-1}(\xi) - I_4|^2 \leq 2|1 - \Upsilon_+|^2 + 2|\Upsilon_-|^2 \leq \frac{C|\xi|^2}{c^2},
        \]
        it follows that
        \begin{equation*}
        \begin{aligned}
        \| \nabla(f_{c,\tau} - f_\tau) \|_{L^2}^2 &= \left\| \xi (\mathbf{U}^{-1} - I_4) \widehat{f_\tau} \right\|_{L^2}^2 \leq \frac{C}{c^2} \| \Delta f_\tau \|_{L^2}^2 \leq \frac{C\tau^2}{c^2}.
        \end{aligned}
        \end{equation*}
        Therefore,
        \begin{equation}\label{llp_2}
        \left| \mathcal{I}^{\infty,\tau}(f_{c,\tau}) - \mathcal{I}^{\infty,\tau}(f_\tau) - d\mathcal{I}^{\infty,\tau}(f_\tau)[f_{c,\tau} - f_\tau] \right| \leq \frac{C\tau^2}{c^2}.
        \end{equation}
        Combining inequalities \eqref{llp_0}, \eqref{llp_1}, and \eqref{llp_2}, we complete the proof of the lemma.
        \end{proof}

        \begin{lemma}\label{stb_8}
          The Lagrange multiplier $\omega$ satisfies the inequality
          \[
          \omega < mc^2.
          \]
          \end{lemma}
          
          \begin{proof}
            It easy to see that
          \begin{equation*}
          \begin{aligned}
          \omega &= \frac{1}{2} d\mathcal{I}^{c,\tau}\left(\varphi^{c,\tau}(f_{c,\tau})\right)\left[\varphi^{c,\tau}(f_{c,\tau})\right] \\
          &= \left\|\varphi^{c,\tau}(f_{c,\tau})\right\|_c^2 - \tau^{\zeta} \int_{\mathbb{R}^3} \left|\varphi^{c,\tau}(f_{c,\tau})\right|^p.
          \end{aligned}
          \end{equation*}          
          Using \eqref{youyong}, Lemma~\ref{AA_1}, and the asymptotic expansion
          \[
          \sqrt{-c^2\Delta + m^2c^4} - mc^2 = -\frac{\Delta}{2m} + \Delta^2 \cdot \mathcal{O}\left(\frac{1}{c^2}\right),
          \]
          we get
          \begin{equation}
          \begin{aligned}
          \omega - mc^2 &\leq \left\|\varphi^{c,\tau}(f_{c,\tau})\right\|_c^2 - mc^2 \left\|\varphi^{c,\tau}(f_{c,\tau})\right\|_{L^2}^2 - \frac{2\tau^\zeta}{p} \int_{\mathbb{R}^3} \left|\varphi^{c,\tau}(f_{c,\tau})\right|^p \\
          &= e^\infty(\tau) + o_c(1)\tau^2 \\
          &= \tau^2 \left(e^{\infty}(1) + o_c(1)\right) \\
          &< 0.
          \end{aligned}
          \end{equation}
          This completes the proof.
          \end{proof}

          \begin{lemma}\label{stb_12}
            For every $\tau \in (0,1]$, the following inequality holds:
            \[
            \mathcal{I}^{c,\tau}(f_{c,\tau}) \leq \mathcal{I}^{c,\tau}(\varphi^{c,\tau}(f_{c, \tau})) \leq \mathcal{I}^{c,\tau}(f_{c,\tau}) + \frac{C\tau^2}{c^2}.
            \]
            \end{lemma}
            
            \begin{proof}
            The argument follows a structure similar to the proof of Lemma~\ref{new_b}. 
            It suffices to estimate
            \[
            \mathcal{I}^{c,\tau}(f_{c,\tau}) - \mathcal{I}^{c,\tau}(\varphi^{c,\tau}(f_{c, \tau})) - d\mathcal{I}^{c,\tau}(\varphi^{c,\tau}(f_{c, \tau}))[f_{c,\tau} - \varphi^{c,\tau}(f_{c, \tau})]
            \]
            and
            \[
            d\mathcal{I}^{c,\tau}(\varphi^{c,\tau}(f_{c, \tau}))[f_{c,\tau} - \varphi^{c,\tau}(f_{c, \tau})].
            \]            
            From Proposition~\ref{prop:3.2} (2) and Lemma~\ref{stb_8}, we obtain
            \begin{equation}\label{stb_11}
            \begin{aligned}
            &d\mathcal{I}^{c,\tau}(\varphi^{c,\tau}(f_{c, \tau}))[\varphi^{c,\tau}(f_{c, \tau}) - f_{c,\tau}] \\
            &= 2\omega \left( \varphi^{c,\tau}(f_{c, \tau}),\ \varphi^{c,\tau}(f_{c, \tau}) - f_{c,\tau} \right)_{L^2} \\
            &= 2\omega \left( \left[\varphi^{c,\tau}(f_{c, \tau})\right]^+,\ \left[\varphi^{c,\tau}(f_{c, \tau})\right]^+ - f_{c,\tau} \right)_{L^2} + 2\omega \left\| \left[\varphi^{c,\tau}(f_{c, \tau})\right]^- \right\|_{L^2}^2 \\
            &= -2\omega \left( \left[\varphi^{c,\tau}(f_{c, \tau})\right]^+,\ \left(1 - \sqrt{1 - \left\| \left[\varphi^{c,\tau}(f_{c, \tau})\right]^- \right\|_{L^2}^2} \right) f_{c,\tau} \right)_{L^2} + 2\omega \left\| \left[\varphi^{c,\tau}(f_{c, \tau})\right]^- \right\|_{L^2}^2 \\
            &\leq -{\omega \left\| \left[\varphi^{c,\tau}(f_{c, \tau})\right]^- \right\|_{L^2}^2 \left\| \left[\varphi^{c,\tau}(f_{c, \tau})\right]^+ \right\|_{L^2}^2} + 2\omega \left\| \left[\varphi^{c,\tau}(f_{c, \tau})\right]^- \right\|_{L^2}^2 
            \\
            &=\omega\left\| \left[\varphi^{c,\tau}(f_{c, \tau})\right]^- \right\|_{L^2}^2+
            \omega\left\| \left[\varphi^{c,\tau}(f_{c, \tau})\right]^- \right\|_{L^2}^4
            \\
            &\leq m c^2 \left\| \left[\varphi^{c,\tau}(f_{c, \tau})\right]^- \right\|_{L^2}^2 + \frac{C\tau^2}{c^2}.
            \end{aligned}
            \end{equation}            
            Following the same reasoning as in the proof of~\eqref{new_b1}, and using inequality~\eqref{stb_6}, we deduce
            \begin{equation}\label{stb_10}
            \begin{aligned}
            &\mathcal{I}^{c,\tau}(\varphi^{c,\tau}(f_{c, \tau})) - \mathcal{I}^{c,\tau}(f_{c,\tau}) - d\mathcal{I}^{c,\tau}(\varphi^{c,\tau}(f_{c, \tau}))[\varphi^{c,\tau}(f_{c, \tau}) - f_{c,\tau}] \\
            &\leq \left\| \left[\varphi^{c,\tau}(f_{c, \tau})\right]^+ - f_{c,\tau} \right\|_c^2 - \left\| \left[\varphi^{c,\tau}(f_{c, \tau})\right]^- \right\|_c^2 + \frac{C\tau^2}{c^2}.
            \end{aligned}
            \end{equation}            
            The desired result follows by combining inequalities~\eqref{stb_10} and~\eqref{stb_11}.
            \end{proof}

\begin{lemma}\label{pseudo}
    There exist constants $c_0, C>0$, such that for $c>c_0$, there holds $$e^{Pseudo,c}(\tau)\geq mc^2-C.$$
\end{lemma}
\begin{proof}
    For $u\in \mathcal{O}_c$, by Lemma \ref{modified}, we have
\begin{equation*}
    \begin{split}
        \mathcal{I}^{Pseudo,c, \tau}(u)= & \| u\|_{c}^2 - \tau^{\zeta} A[u]\\
         \geqslant& (1- Cc^{-1/2}) (\| u\|_{c}^2- mc^2\|u\|_{L^2}^2 )
        \\
        &-C \left(\|u\|_c^2 - mc^2\|u\|_{L^2}^2 \right)^{\frac{3p-6}{4}}
        + mc^2\\
        \geqslant& mc^2 - C.
    \end{split}
\end{equation*}
\end{proof}

\begin{proof}[\textbf{Proof of Proposition  \ref{mulp}}]
  By combining Lemma \ref{lemm:3.2} and Lemma \ref{pseudo}, we obtain
  \[
  e^c(\tau) \geq e^{{Pseudo},c}(\tau) \geq m c^2 - C.
  \]  
  Then, applying Lemma \ref{stb_12}, Lemma \ref{new_b}, and the operator inequality
  \[
  \sqrt{-c^2\Delta + m^2 c^4} \leq -\frac{\Delta}{2m} + m c^2,
  \]
  we derive
  \begin{equation}\label{0q}
  \begin{split}
  e^c(\tau) 
  &\leq \mathcal{I}^{c,\tau}(\varphi^{c,\tau}(f_{c,\tau})) 
  \leq \mathcal{I}^{c,\tau}(f_{c,\tau}) + \frac{C \tau^2}{c^2} \\
  &\leq \mathcal{I}^{\infty,\tau}(f_{c,\tau}) + m c^2 + \frac{C \tau^2}{c^2} \\
  &\leq \mathcal{I}^{\infty,\tau}(f_{\tau}) + m c^2 + \frac{C \tau^2}{c^2} \\
  &= e^\infty(\tau) + m c^2 + \frac{C \tau^2}{c^2}.
  \end{split}
  \end{equation}
  This ends the proof.
  \end{proof}

\subsection{Minimization problem}\label{4.2}
In this section, we  employ the upper bound estimate for $e^c(\tau)$ in Proposition \ref{mulp} to prove the existence of a minimizer to the variational problem $$
e^c(\tau)=\inf _{w \in \mathcal{O}_c^{+}} \mathcal{I }_{red}^{c,\tau}  (w).
$$
It follows from Lemma \ref{AA_1} that \( e^{\infty}(\tau) = \tau^2 e^{\infty}(1) < 0 \). Combining this with Proposition \ref{mulp}, we obtain the following result.

\begin{lemma}\label{as_stb}
There exist constants \( c_0, C > 0 \) such that for all \( c > c_0 \) and every \( \tau \in (0,1] \), 
\[
e^c(\tau) \in (0, m c^2).
\]
\end{lemma}

\begin{lemma}\label{lemm_3.88}
  For any $0 < \tau_1 < \tau_2$, we have $e^c(\tau_2) < e^c(\tau_1)$.
\end{lemma}

\begin{proof}
It is clear that $e^c(\tau_2) \leq e^c(\tau_1)$. Suppose, by contradiction, that $e^c(\tau_2) = e^c(\tau_1)$. Then for any $w \in \mathcal{O}_c^+$,
\begin{equation*}
  \begin{split}
    e^c(\tau_2) &\leq \mathcal{I}^{c, \tau_2}(\varphi^{c,\tau_2}(w)) \\
    &= \| \varphi^{c,\tau_2}(w)^+\|_c^2 - \| \varphi^{c,\tau_2}(w)^-\|_c^2 - \tau_1^{\zeta} A[\varphi^{c,\tau_2}(w)] - (\tau_2^{\zeta} - \tau_1^{\zeta}) A[\varphi^{c,\tau_2}(w)] \\
    &\leq \mathcal{I}^{c, \tau_1}\left(\varphi^{c, \tau_1}P(\varphi^{c,\tau_2}(w)^+)
\right) - (\tau_2^{\zeta} - \tau_1^{\zeta}) A[\varphi^{c,\tau_2}(w)] \\
    &= \mathcal{I}^{c, \tau_1}(\varphi^{c, \tau_1}(w)) - (\tau_2^{\zeta} - \tau_1^{\zeta}) A[\varphi^{c,\tau_2}(w)].
  \end{split}
\end{equation*}
Let $\{w_n\} \subset \mathcal{O}_{c}^+$ be a minimizing sequence for the reduced functional $\mathcal{I}_{\mathrm{red}}^{c,\tau_1}(w) = \mathcal{I}^{c, \tau_1}(\varphi^{c, \tau_1}(w))$, so that
\[
e^c(\tau_1) = \mathcal{I}^{c, \tau_1}(\varphi^{c, \tau_1}(w_n)) + o_n(1).
\]
Then
\[
e^c(\tau_2) \leq e^c(\tau_1) - (\tau_2^{\zeta} - \tau_1^{\zeta}) A[\varphi^{c,\tau_2}(w_n)] + o_n(1),
\]
which implies
\[
A[\varphi^{c,\tau_2}(w_n)] = o_n(1).
\]
By Lemma \ref{com}, we also have
\begin{equation}\label{nq}
    A[\varphi^{c,\tau_2}(w_n)^+] = o_n(1).
\end{equation}
Combining Proposition~\ref{prop:3.2} (1) and (2) with \eqref{nq}, we obtain
\[
A(w_n) = o_n(1), \quad \|\varphi^{c,\tau_1}(w_n)^-\|_{c}^2 = o_n(1).
\]
From Lemma \ref{as_stb},
\begin{equation*}
  \begin{split}
    m c^2 - C &\geq e^c(\tau_1) = \mathcal{I}^{c, \tau_1}(\varphi^{c, \tau_1}(w_n)) + o_n(1) \\
    &= \|\varphi^{c, \tau_1}(w_n)^+\|_c^2 + o_n(1) \geq m c^2 + o_n(1),
  \end{split}
\end{equation*}
a contradiction.
\end{proof}

For $\tau \in (0,1)$, define
\[
E^c(\tau) = \tau^{\theta} e^c(\tau), \quad \text{where } \theta = \frac{3\zeta}{2 - \zeta}.
\]
Then Lemma~\ref{lemm_3.88} implies the strict subadditivity of $E_c(\tau^{1/\theta})$.

\begin{lemma}\label{lemm:3.8}
 For any $\tau_1, \tau_2 \in (0, 1)$, we have
 \[
 E^c((\tau_1 + \tau_2)^{1/\theta}) < E^c(\tau_1^{1/\theta}) + E^c(\tau_2^{1/\theta}).
 \]
\end{lemma}

\begin{proof}
  \[
  \begin{aligned}
    E^c((\tau_1 + \tau_2)^{1/\theta})
    &= \frac{\tau_1}{\tau_1 + \tau_2} E^c((\tau_1 + \tau_2)^{1/\theta}) + \frac{\tau_2}{\tau_1 + \tau_2} E^c((\tau_1 + \tau_2)^{1/\theta}) \\
    &= \tau_1 e^c((\tau_1 + \tau_2)^{1/\theta}) + \tau_2 e^c((\tau_1 + \tau_2)^{1/\theta}) \\
    &< \tau_1 e^c(\tau_1^{1/\theta}) + \tau_2 e^c(\tau_2^{1/\theta}) \\
    &= E^c(\tau_1^{1/\theta}) + E^c(\tau_2^{1/\theta}).
  \end{aligned}
  \]
\end{proof}

\begin{lemma}\label{lo}
  There exist constants $c_0, C > 0$ such that for all $c > c_0$ and all $u \in \mathcal{O}_c^+$ satisfying
  \[
  \mathcal{I}_{{red}}^{c,\tau}(u) = \mathcal{I}^{c, \tau}(\varphi^{c, \tau}(u)) < m c^2,
  \]
  we have
  \[
  \|u\|_{H^{1/2}} \leq C.
  \]
\end{lemma}

\begin{proof}
  Since $u \in \mathcal{O}_c^+$, by Lemma~\ref{modified}, we obtain
  \[
  \begin{split}
    0 &> \mathcal{I}^{c, \tau}(\varphi^{c, \tau}(u)) - m c^2 \geq \mathcal{I}^{c, \tau}(u) - m c^2 \\
    &\geq (1 - C c^{-1/2})(\|u\|_c^2 - m c^2) - C (\|u\|_c^2 - m c^2)^{\frac{3p - 6}{4}},
  \end{split}
  \]
  which implies
  \[
  C_2 \|u\|_{H^{1/2}}^2 \leq \|u\|_c^2 - m c^2 + 1 \leq C_1.
  \]
\end{proof}

\begin{remark}
   Lemma \ref{lo} shows that there exists a constant  $C>0$, such that
   $$
   e^c(\tau)= \inf \{\mathcal{I }_{red}^{c,\tau}(u): u\in \mathcal{O}_c^+ \}
   =\inf \{\mathcal{I }_{red}^{c,\tau}(u): \|u\|_{L^2}=1, \|u\|_{H^{1/2}} < C \}.
   $$
\end{remark}
The following Lemma is essentially Lemma 4.5  in \cite{Nolasco}.
\begin{lemma}\label{lemm:3.1*}
  Let ${u} \in \mathcal{O}_c$ be such that
 $$
 d \mathcal{I}^{c, \tau}({u})[h]-2 \omega \Re({u}, h)_{L^2}=0, \quad \forall h \in   H^{1/2}(\mathbb{R}^3,\mathbb{C}^4),
 $$
 for some $\tau \in(0,1]$ and $\omega \in(0,mc^2)$. If  $w=u^+ / \|u^+\|_{L^2} \in \mathcal{O}_c^{+}$, then $u=\varphi^{c,\tau}(w)$ is the unique  maximizer of $\mathcal{I}^{c, \tau}$ on $S(w)$, namely
 $$
 \mathcal{I}^{c, \tau}({u})=\sup _{{v} \in S(w)} \mathcal{I}^{c, \tau}({v})=\mathcal{I }_{red}^{c,\tau}(w).
 $$
 \end{lemma}
 \begin{proof}[\textbf{Proof of Theorem \ref{them:1.1}}] 
  \textbf{Existence:}  
  By Ekeland’s variational principle and Lemma \ref{lo}, there exists a minimizing sequence $\{w_n\} \subset \mathcal{O}_c^+$ such that $\|w_n\|_{H^{1/2}} < C$,  
  \[
  \mathcal{I}_{red}^{c,1}(w_n) = \mathcal{I}^{c,1}(\varphi^{c,1}(w_n)) \to e^c(1), \quad \text{and} \quad \|d\mathcal{I}_{red}^{c,1}(w_n)\| \to 0.
  \]  
  Define $u_n := \varphi^{c,1}(w_n)$. Then we have  
  \[
  \sup_{\|h\|_{H^{1/2}} = 1} \left| d\mathcal{I}^{c,1}(u_n)[h] - 2\omega(u_n) \Re(u_n, h)_{L^2} \right| \to 0.
  \]  
  Since $\{u_n\}$ is bounded in $H^{1/2}$, it follows from Proposition \ref{mulp} that (up to a subsequence) $\omega(u_n) \to \omega_c \in (mc^2 - C_1, mc^2 - C_2)$ for some constants $C_1, C_2 > 0$, and $\{u_n\}$ is a bounded Palais–Smale sequence for the functional  
  \[
  \mathcal{I}_\omega^c(u) = \mathcal{I}^{c,1}(u) - \omega_c \|u\|_{L^2}^2.
  \]  
  By the concentration–compactness principle, either $\{u_n\}$ is vanishing or nonvanishing. We first show that $\{u_n\}$ is nonvanishing. Suppose, by contradiction, that it is vanishing. Then $\|u_n\|_{L^t} \to 0$ for all $t \in (2, 3)$, which implies  
  \[
  A[u_n] = o_n(1).
  \]  
  Using an argument from Lemma \ref{lemm_3.88}, we obtain $\|u_n^-\|_c \to 0$ as $n \to \infty$. Hence,  
  \[
  mc^2 - C \geq e^c(1) = \mathcal{I}^{c,1}(u_n) + o_n(1) = \|u_n^+\|_c^2 + o_n(1) \geq mc^2 + o_n(1),
  \]  
  a contradiction. Therefore, $\{u_n\}$ is nonvanishing.  
  
  By Proposition \ref{bound} (5) and the concentration–compactness principle, there exist a finite integer $q \geq 1$, nontrivial critical points $v_1, \dots, v_q \in H^{1/2}(\mathbb{R}^3, \mathbb{C}^4)$ of $\mathcal{I}_\omega^c$ with $\|v_i\|_{L^2}^2 = \tau_i \in (0, 1]$, and sequences $\{x_n^i\} \subset \mathbb{R}^3$ for $i = 1, \dots, q$ such that $|x_n^i - x_n^j| \to +\infty$ for $i \ne j$, and (up to a subsequence)  
  \[
  \left\| u_n - \sum_{i=1}^q v_i(\cdot - x_n^i) \right\|_c \to 0 \quad \text{as} \quad n \to +\infty.
  \]  
  In particular,  
  \[
  1 = \|u_n\|_{L^2}^2 = \sum_{i=1}^q \tau_i.
  \]  
  Moreover, since $u_n^\pm(\cdot + x_n^i) \rightharpoonup v_i^\pm$ weakly in $E_c$, we obtain  
  \[
  \begin{aligned}
  \|u_n^\pm\|_c^2 &= \left( u_n - \sum_{i=1}^q v_i(\cdot - x_n^i), u_n^\pm \right)_c + \sum_{i=1}^q \left( v_i, u_n^\pm(\cdot + x_n^i) \right)_c \\
  &= \sum_{i=1}^q \|v_i^\pm\|_c^2 + o_n(1).
  \end{aligned}
  \]  
  Similarly,  
  \[
  A[u_n] = \sum_{i=1}^q A[v_i] + o_n(1),
  \]  
  which implies  
  \[
  \mathcal{I}^{c,1}(u_n) = \sum_{i=1}^q \mathcal{I}^{c,1}(v_i) + o_n(1), \quad \text{and hence} \quad e^c (1)= E^c(1) = \sum_{i=1}^q \mathcal{I}^{c,1}(v_i).
  \]  
  
  For each $i = 1, \dots, q$, define $g_i = v_i / \|v_i\|_{L^2} = v_i / \sqrt{\tau_i} \in \mathcal{S}$. The boundedness of $\{u_n\}$ in $H^{1/2}$ and Proposition \ref{bound} (4) imply that $g_i \in \mathcal{O}_c$. Moreover,  
  \[
  \mathcal{I}^{c,1}(v_i) = \mathcal{I}^{c,1}(\sqrt{\tau_i} g_i) = \tau_i \mathcal{I}^{c, \tau_i^{1/\theta}}(g_i),
  \]  
  and  
  \[
  0 = d\mathcal{I}_\omega^c(v_i)[h] = \sqrt{\tau_i} \left( d\mathcal{I}^{c, \tau_i^{1/\theta}}(g_i)[h] - 2\omega_c \Re(g_i, h)_{L^2} \right), \quad \forall h \in H^{1/2}(\mathbb{R}^3, \mathbb{C}^4).
  \]  
  Thus, for each $i$, $g_i \in \mathcal{O}_c$ is a critical point of $\mathcal{I}^{c, \tau_i^{1/\theta}}$ on $\mathcal{S}$ with Lagrange multiplier $\omega_c \in (mc^2 - C_1, mc^2 - C_2)$.  
  
  Now define $w_i = g_i^+ / \|g_i^+\|_{L^2} = v_i^+ / \|v_i^+\|_{L^2} \in \mathcal{S}^{c,+}$. By Proposition \ref{bound} (4), $w_i \in \mathcal{O}_c^+$ for all $i$. Then Lemma \ref{lemm:3.1*} implies that $g_i = \varphi^{c, \tau_i^{1/\theta}}(w_i)$ and  
  \[
  \mathcal{I}^{c, \tau_i^{1/\theta}}(g_i) = \sup_{u \in S(w_i)} \mathcal{I}^{c, \tau_i^{1/\theta}}(u) \geq e^c(\tau_i^{1/\theta}).
  \]  
  Therefore,  
  \[
  e^c(1) = E^c(1) = \sum_{i=1}^q \mathcal{I}^{c,1}(v_i) = \sum_{i=1}^q \tau_i \mathcal{I}^{c, \tau_i^{1/\theta}}(g_i) \geq \sum_{i=1}^q \tau_i e^c(\tau_i^{1/\theta}) = \sum_{i=1}^q E^c(\tau_i^{1/\theta}),
  \]  
  which contradicts Lemma \ref{lemm:3.8} unless $q = 1$. Hence, up to translations, $u_n \to u_c = v_1$ strongly in $E_c$, with $\|u_c\|_{L^2}^2 = 1$ and  
  \[
  \mathcal{I}^{c,1}(u_c) = e^c(1) = \inf_{w \in \mathcal{O}_c^+} \sup_{u \in S(w)} \mathcal{I}^{c,\tau}(u).
  \]  
  Moreover,  
  \[
  d\mathcal{I}^{c,1}(u_c)[h] - 2\omega_c \Re(u_c, h)_{L^2} = 0 \quad \forall h \in H^{1/2}(\mathbb{R}^3, \mathbb{C}^4),
  \]  
  so $(u_c, \omega) \in H^{1/2}(\mathbb{R}^3, \mathbb{C}^4) \times (mc^2 - C_1, mc^2 - C_2)$ is a weak solution of \eqref{Dirac_e}.  
  
  Finally, by Proposition \ref{bound}, for each $t > 1$,  
  \[
  \sup_{c > c_0} \|u_c\|_{W^{2,t}} < \infty.
  \]  
  \end{proof}

Although the energy minimizer $u_c$ constructed in Theorem \ref{them:1.1} is a local
minimizer, it can still be referred to as an energy ground state of \eqref{Dirac_e} in the sense that
   $$
    \mathcal{I}^{c}(u_c)= \inf\limits_{u\in\mathcal{S}}\{\mathcal{I}^{c}(u): d\mathcal{I}^{c}(u)|_{\mathcal{S}} =0 \,\,\text{and}\,\,\mathcal{I}^{c}(u)>0 \}.
    $$
    We begin by considering the following Pohozaev identity.
    \begin{lemma}\label{Pohozaev} 
      Let \( u \in H^{1/2}(\mathbb{R}^3, \mathbb{C}^4) \) be a weak solution of
      \begin{equation}\label{eqat}
          \mathscr{D}_{c} u - |u|^{p-2}u = \mu u
      \end{equation}
      for some \( \mu \in (0, mc^2) \). Then \( u \) satisfies the Pohozaev identity:
      \begin{equation}
          \|u^+\|_c^2 - \|u^-\|_c^2 = \int_{\mathbb{R}^3} \langle m c^2 \beta u, u \rangle - \frac{6 - 3p}{p} \int_{\mathbb{R}^3} |u|^p.
      \end{equation}
  \end{lemma}
  
  \begin{proof}
      Multiply both sides of~\eqref{eqat} by \( x \cdot \nabla u \) and integrate over \( \mathbb{R}^3 \). Then
      \begin{equation}\label{123}
          \begin{split}
              \Re(\mathscr{D}_c u, x \cdot \nabla u)_{L^2}
              &= \left. \frac{d}{d\lambda} \right|_{\lambda=1} \Re(\mathscr{D}_c u, u(\lambda x))_{L^2} \\
              &= \left. \frac{d}{d\lambda} \right|_{\lambda=1} \lambda^{-3} \Re \left( \widehat{\mathscr{D}_c}(\xi) \hat{u}(\xi), \hat{u}(\lambda^{-1} \xi) \right)_{L^2} \\
              &= \left. \frac{d}{d\lambda} \right|_{\lambda=1} \lambda^{-1} \Re \left( 
                  \begin{pmatrix}
                      \lambda^{-1/2} m c^2 I_2 & c \sigma \cdot \xi \\
                      c \sigma \cdot \xi & -\lambda^{-1/2} m c^2 I_2
                  \end{pmatrix}
                  \hat{u}(\lambda^{1/2} \xi), \hat{u}(\lambda^{-1/2} \xi) \right)_{L^2} \\
              &= - (\mathscr{D}_c u, u)_{L^2} - \frac{1}{2} (m c^2 \beta u, u)_{L^2}.
          \end{split}
      \end{equation}  
      Since \( u \) decays exponentially, we also have
      \begin{equation}\label{234}
          \begin{split}
              \Re(|u|^{p-2}u + \mu u, x \cdot \nabla u)_{L^2}
              &= \frac{\mu}{2} \int_{\mathbb{R}^3} x \cdot \nabla(|u|^2) + \frac{1}{p} \int_{\mathbb{R}^3} x \cdot \nabla(|u|^p) \\
              &= -\frac{3\mu}{2} \|u\|_{L^2}^2 - \frac{3}{p} \|u\|_{L^p}^p.
          \end{split}
      \end{equation}  
      From~\eqref{123}, it follows that
      \begin{equation}\label{789}
          (\mathscr{D}_c u, u)_{L^2} + \frac{1}{2} (m c^2 \beta u, u)_{L^2} = \frac{3\mu}{2} \|u\|_{L^2}^2 + \frac{3}{p} \|u\|_{L^p}^p.
      \end{equation}  
      Now multiply~\eqref{eqat} by \( u \) and integrate to obtain
      \begin{equation}\label{456}
          (\mathscr{D}_c u, u)_{L^2} - \|u\|_{L^p}^p = \mu \|u\|_{L^2}^2.
      \end{equation}  
      Combining~\eqref{456} and~\eqref{789}, we conclude that
      \begin{equation*}
          \|u^+\|_c^2 - \|u^-\|_c^2 = \int_{\mathbb{R}^3} \langle m c^2 \beta u, u \rangle - \frac{6 - 3p}{p} \int_{\mathbb{R}^3} |u|^p.
      \end{equation*}
  \end{proof}

     \begin{lemma}
     \( u_c \) is an energy ground state of \( \mathcal{I}^c \) on \( \mathcal{S} \), i.e.,
      \[
      \mathcal{I}^{c}(u_c) = e_{\mathrm{ene}}^c = \inf \left\{ \mathcal{I}^{c}(u) : u \in \mathcal{S},\ d\mathcal{I}^{c}(u)|_{\mathcal{S}} = 0,\ \mathcal{I}^{c}(u) > 0 \right\}.
      \]
  \end{lemma}
  
  \begin{proof}
      Suppose there exists a critical point \( v_c \) of \( \mathcal{I}^c \) on \( \mathcal{S} \) such that
      \[
      0 < \mathcal{I}^c(v_c) \leq \mathcal{I}^c(u_c) < m c^2.
      \]
      Then \( v_c \) satisfies
      \[
      \mathscr{D}_{c} v_c - |v_c|^{p-2}v_c = \mu_c v_c
      \]
      for some \( \mu_c < m c^2 \). Since
      \[
      \mu_c = \frac{2}{p} \mathcal{I}^c(v_c) + \frac{p-2}{p} (\mathscr{D}_c v_c, v_c)_{L^2} > 0,
      \]
      it follows from Lemma~\ref{Pohozaev} that \( v_c \) satisfies the Pohozaev identity:
      \begin{equation}
          \|v_c^+\|_c^2 - \|v_c^-\|_c^2 = \int_{\mathbb{R}^3} \langle m c^2 \beta v_c, v_c \rangle - \frac{6 - 3p}{p} \int_{\mathbb{R}^3} |v_c|^p.
      \end{equation}
      Therefore,
      \begin{equation}
          \begin{split}
              m c^2 &> \mathcal{I}^c(u_c) \geq \mathcal{I}^c(v_c) \\
              &= \|v_c^+\|_c^2 - \|v_c^-\|_c^2 - \frac{2}{p} \int_{\mathbb{R}^3} |v_c|^p \\
              &= \frac{3p - 8}{3p - 6} (\|v_c^+\|_c^2 - \|v_c^-\|_c^2) + \frac{2}{3p - 6} \int_{\mathbb{R}^3} \langle m c^2 \beta v_c, v_c \rangle.
          \end{split}
      \end{equation}
      Hence,
      \[
      \|v_c^+\|_c^2 - \|v_c^-\|_c^2 \leq C_p m c^2,
      \]
      and since \( m c^2 > \mathcal{I}^c(v_c) > 0 \), we also have
      \[
      \|v_c\|_{L^p}^p \leq C_p m c^2, \quad \mu_c < \mathcal{I}^c(v_c) < m c^2.
      \]
      Consequently,
      \[
      \|v_c^-\|_c^2 < \|v_c^+\|_c^2 = \int_{\mathbb{R}^3} |v_c|^{p-2} \Re(v_c, v_c^+) + \mu_c \|v_c^+\|_{L^2}^2 < C_p m c^2,
      \]
      which implies \( v_c \in \mathcal{O}_c \). Moreover,
      \begin{equation}\label{az}
          \begin{split}
              \mu_c \|v_c^+\|_{L^2}^2 &= \|v_c^+\|_c^2 - \int_{\mathbb{R}^3} |v_c|^{p-2} \Re(v_c, v_c^+) \\
              &\geq \frac{1}{2} \|v_c\|_c^2 - C_p \|v_c\|_{L^p}^p \\
              &\geq \left( \frac{1}{2} - C_p c^{-1/2} \right) \|v_c\|_c^2 > 0,
          \end{split}
      \end{equation}
      and similarly,
      \begin{equation}\label{sx}
          \begin{split}
              \mu_c \left( \|v_c^+\|_{L^2}^2 - \|v_c^-\|_{L^2}^2 \right) &= \|v_c\|_c^2 - \int_{\mathbb{R}^3} |v_c|^{p-2} \Re(v_c, v_c^+ - v_c^-) > 0.
          \end{split}
      \end{equation}
      Combining \eqref{az} and \eqref{sx}, we obtain \( \|v_c^+\|_{L^2}^2 > \frac{1}{2} \|v_c\|_{L^2}^2 = \frac{1}{2} \), and
      \[
      \left\| {v_c^+}/{\|v_c^+\|_{L^2}} \right\|_c^2 < 2 \|v_c^+\|_c^2 < c^{2s}.
      \]
      Hence, \( v_c^+ / \|v_c^+\|_{L^2} \in \mathcal{O}_c^+ \). By Lemma~\ref{lemm:3.1*}, we conclude that
      \[
      \mathcal{I}^c(v_c) = \mathcal{I}_{\mathrm{red}}^{c,\tau} \left( {v_c^+}/{\|v_c^+\|_{L^2}} \right) \geq \mathcal{I}^c(u_c),
      \]
      which implies
      \[
      \mathcal{I}^{c}(u_c) = \inf \left\{ \mathcal{I}^{c}(u) : u \in \mathcal{S},\ d\mathcal{I}^{c}(u)|_{\mathcal{S}} = 0,\ \mathcal{I}^{c}(u) > 0 \right\}.
      \]
      This ends the proof.
  \end{proof}

{

From the fact that
\begin{equation}\label{new_m}
  -\infty<\lim\inf_{c\to\infty}(\omega_c-mc^2)\leq\lim\sup_{c\to\infty}(\omega_c-mc^2)<0,
\end{equation}
we can assume 
$
\lim\limits_{c\to\infty}mc^2- \omega_c = \lambda>0.
$
\begin{lemma}
  For any $\delta \in (0, \sqrt{2m\lambda})$, there exists a constant $0 < C(\delta) < \infty$ such that for all sufficiently large $c > 0$, the following estimates hold:
  \[
  |P_\infty^+ u_c(x)| \leq C(\delta) e^{-\delta |x|}, \quad
  |P_\infty^- u_c(x)| \leq \frac{C(\delta)}{c} e^{-\delta |x|}.
  \]
\end{lemma}

\begin{proof}
  Note that for large $c > 0$, the operator $(\mathscr{D}_c - \omega_c)^{-1}$ is well defined, and $u_c$ admits the integral representation
  \[
  u_c(x) = \int_{\mathbb{R}^3} Q_c(x - y) M_c(y) u_c(y) \, dy,
  \]
  where $M_c(x) = |u_c|^{p-2}(x)$ and $Q_c$ is the Green's function of $\mathscr{D}_c - \omega_c$, given explicitly by
  \[
  Q_c(x) = \left( ic \frac{\alpha \cdot x}{|x|^2} + i\sqrt{m^2c^4 - \omega_c^2} \frac{\alpha \cdot x}{|x|} + \beta m c^2 + \omega_c \right)
  \frac{1}{4\pi c^2 |x| e^{\sqrt{m^2 c^2 - \omega_c^2/c^2} |x|}}.
  \]  
  Using the limit
  \[
  \lim_{c \to \infty} (m c^2 - \omega_c) = \lambda > 0,
  \]
  we deduce that for any $\delta \in (0, \sqrt{2m\lambda})$,
  \[
  |P_\infty^+ Q_c(x)| \leq C(\delta) \frac{e^{-\delta |x|}}{|x|^2}, \quad
  |P_\infty^- Q_c(x)| \leq C(\delta) \frac{e^{-\delta |x|}}{c |x|^2}, \quad \text{for all } x \in \mathbb{R}^3.
  \]  
  Moreover, the function $M_c(x) = |u_c|^{p-2}(x)$ is continuous on $\mathbb{R}^3$ and satisfies $\lim\limits_{|x| \to \infty} M_c(x) = 0$ uniformly for $c > c_0$, since $u_c$ is uniformly bounded in $H^2(\mathbb{R}^3, \mathbb{C}^4)$. Therefore, by applying Theorem 2.1 in \cite{MR1785381}, we conclude that for any $\delta \in (0, \sqrt{2m\lambda})$, there exists $0 < C(\delta) < \infty$ such that for all large $c$,
  \[
  |P_\infty^+ u_c(x)| \leq C(\delta) e^{-\delta |x|}, \quad
  |P_\infty^- u_c(x)| \leq \frac{C(\delta)}{c} e^{-\delta |x|}.
  \]
  This completes the proof.
\end{proof}

Replicate the proof of Theorem \ref{them:1.1}
, one can conclude that there exists $c_0 >0$,  such that for $\tau\in (0,1]$, the functional
$$\mathcal{I}^{c, \tau}({u}):= \|u^+\|_c^2-\| u^-\|_c^2- \tau^\zeta A[u] $$
possesses a critical point $u_{c,\tau}$ on $\mathcal{S}$
at the lowest positive critical value level. 
That is 
$$
\mathcal{I}^{c,\tau}(u_{c,\tau}) =
\inf_{w\in \mathcal{O}_c^{+}}\mathcal{I}^{c, \tau}(\varphi^{c,\tau}(w))=
\inf _{\substack{w \in \mathcal{O}_c^+  }} \sup _{{u} \in   S(w)} \mathcal{I}^{c, \tau }({u}).
$$
We denote the Lagrange multiplier  associated with $u_{c,\tau}$ by $\omega_{c, \tau}$.
 Lemma \ref{lemm:3.8} implies ${E}_c(\tau^{1/\theta})=
 \tau e^c(\tau^{1/\theta})$ is 
 strictly concave on $(0, 1]$, where $\theta= \frac{3\zeta}{2-\zeta}$.
 This leads to the following properties:
 \begin{enumerate}
 \item ${E}_c(\tau)$ is continuous on $(0, 1]$.
 \item The left and right derivatives of ${E}_c(\tau)$, denoted by $(d{E}_c(\tau)/d\tau)_-$ and $(d {E}_c(\tau)/d\tau)_+$ respectively, exist.
 \item The derivative $d {E}_c(\tau)/d\tau$ exists for all $\tau \in (0, 1]$, except possibly on a countable set $\Sigma_c$.
 \end{enumerate}
For convenience, we denote the set where $d {E}_c(\tau)/ d \tau$
exists by $\Sigma_c^C := (0,1]\backslash \Sigma_c$.
\begin{lemma}\label{lemm:uniq}
  The one-sided derivatives of \( E_c(\tau) \) satisfy
  \[
  \left(\frac{d E_c(\tau)}{d \tau}\right)_+ \leq \theta \tau^{\theta - 1} \omega_{c,\tau} \leq \left(\frac{d E_c(\tau)}{d \tau}\right)_-.
  \]
\end{lemma}

\begin{proof}
Let \( u_{c, \tau} \) be a critical point of the functional \( \mathcal{I}^{c,\tau} \) at the lowest positive critical value, i.e.,
\[
\mathcal{I}^{c,\tau}(u_{c, \tau}) = e^c(\tau).
\]
Then the right derivative satisfies
\begin{align*}
\left(\frac{d E_c(\tau)}{d \tau}\right)_+ 
&= \lim_{\mu \to \tau^+} \frac{E_c(\mu) - E_c(\tau)}{\mu - \tau} \\
&= \theta \tau^{\theta - 1} e^c(\tau) + \tau^\theta \lim_{\mu \to \tau^+} \frac{\mathcal{I}^{c, \mu}(u_{c, \mu}) - \mathcal{I}^{c, \tau}(u_{c, \tau})}{\mu - \tau} \\
&\leq \theta \tau^{\theta - 1} e^c(\tau) + \tau^\theta \lim_{\mu \to \tau^+} \frac{\mathcal{I}^{c, \mu}(\varphi^{c, \mu}(P(u_{c, \tau}^+))) - \mathcal{I}^{c, \tau}(u_{c, \tau})}{\mu - \tau},
\end{align*}
where \( P(u_{c, \tau}^+) = u_{c, \tau}^+ / \|u_{c, \tau}^+\|_{L^2} \).

For \( \mu > \tau \), we have the identity
\[
\mathcal{I}^{c, \mu}(u) = \mathcal{I}^{c, \tau}(u) + (\tau^\zeta - \mu^\zeta) A[u].
\]
Therefore,
\begin{align*}
\left(\frac{d E_c(\tau)}{d \tau}\right)_+ 
&\leq \theta \tau^{\theta - 1} e^c(\tau) + \tau^\theta \lim_{\mu \to \tau^+} \frac{\mathcal{I}^{c, \tau}(\varphi^{c, \mu}(P(u_{c, \tau}^+))) - \mathcal{I}^{c, \tau}(u_{c, \tau})}{\mu - \tau} \\
&\quad - \zeta \tau^{\theta + \zeta - 1} \lim_{\mu \to \tau^+} A[\varphi^{c, \mu}(P(u_{c, \tau}^+))] \\
&= \theta \tau^{\theta - 1} e^c(\tau) + \tau^\theta \lim_{\mu \to \tau^+} \frac{\mathcal{I}^{c, \tau}(\varphi^{c, \mu}(P(u_{c, \tau}^+))) - \mathcal{I}^{c, \tau}(\varphi^{c, \tau}(P(u_{c, \tau}^+)))}{\mu - \tau} \\
&\quad - \zeta \tau^{\theta + \zeta - 1} A[u_{c, \tau}] \\
&= \theta \tau^{\theta - 1} e^c(\tau) + \tau^\theta \, d\mathcal{I}^{c,\tau}(u_{c, \tau}) \left[ \frac{d}{dt} \varphi^{c, t}(P(u_{c, \tau}^+))\big|_{t=\tau} \right] - \zeta \tau^{\theta + \zeta - 1} A[u_{c, \tau}].
\end{align*}

Let \( W = \operatorname{span}\{u_{c, \tau}^+\} \). Since \( \varphi^{c, t}(P(u_{c, \tau}^+)) \in W \oplus E_c^- \), it follows that
\[
\frac{d}{dt} \varphi^{c, t}(P(u_{c, \tau}^+))\big|_{t=\tau} \in T_{u_{c, \tau}}(W \oplus E_c^-),
\]
and therefore
\[
d\mathcal{I}^{c,\tau}(u_{c, \tau}) \left[ \frac{d}{dt} \varphi^{c, t}(P(u_{c, \tau}^+))\big|_{t=\tau} \right] = 0.
\]
Thus,
\[
\left(\frac{d E_c(\tau)}{d \tau}\right)_+ \leq \theta \tau^{\theta - 1} e^c(\tau) - \zeta \tau^{\theta + \zeta - 1} A[u_{c, \tau}] = \theta \tau^{\theta - 1} \omega_{c,\tau}.
\]

A similar argument shows that
\[
\theta \tau^{\theta - 1} \omega_{c,\tau} \leq \left(\frac{d E_c(\tau)}{d \tau}\right)_-,
\]
which completes the proof.
\end{proof}

Lemma \ref{lemm:uniq} implies $d {E}_c(\tau)/ d \tau =  
            \theta\tau^{\theta-1}\omega_{c,\tau}$ on $\Sigma_c^C$,  in particular $ \omega_{c, \tau}$ is unique.

            \begin{lemma}\label{un}
              For every $c > c_0$, the value $\omega_{c,1}$ depends only on $c$, except possibly on a countable set $\Xi$.
          \end{lemma}
          
          \begin{proof}
          For any $u \in E_c$, define the scaling transformation
          \[
          \mathcal{T}_c(u)(x) = c^{-3/2} u(c^{-1} x).
          \]
          A direct computation shows that
          \[
          \mathcal{I}^{c,1}(u) = c^{-2} \mathcal{I}^{1,c^{-1}}(\mathcal{T}_c u).
          \]
          By Lemma \ref{scal}, we have $u \in \mathcal{S}^{c,\pm}$ if and only if $\mathcal{T}_c(u) \in \mathcal{S}^{1,\pm}$. Hence, for a fixed $w \in \mathcal{O}_c^+$,
          \[
          \mathcal{T}_c(\varphi^{c,1}(w)) = \varphi^{1,c^{-1}}(\mathcal{T}_c(w)).
          \]
          This implies that $u_{c,1}$ is an energy ground state of $\mathcal{I}^{c,1}$ if and only if $\mathcal{T}_c u_{c,1}$ is an energy ground state of $\mathcal{I}^{1,c^{-1}}$. Therefore,
          \[
          \omega_{c,1} = c^{-2} \omega_{1,c^{-1}}.
          \]
          From Lemma \ref{lemm:uniq}, it follows that $\omega_{1,c^{-1}}$ depends only on $c$, except when $c^{-1} \in \Sigma_1$. Thus, the lemma holds with the countable exception set
          \[
          \Xi = \{ c > c_0 : c^{-1} \in \Sigma_1 \}.
          \]
          \end{proof}
This ends the proof of Theorem \ref{them:1.1}.

}

\section{Nonrelativistic Limits  of energy ground state}\label{section4}

In this section, we present the proof of Theorem \ref{them:1.2}. We begin by establishing 
the convergence rate of the ground state energy for \eqref{Dirac_e}. Combining this result, 
we then show that the positive part of the energy ground state \(u_c\) forms a minimizing 
sequence for \(\mathcal{I}^\infty\) over \(\mathcal{S}'\). Finally, using the concentration-compactness principle, 
we conclude the convergence of the solutions.

%Since $u_c$ is $L^2$-normalized, then we have
%the functional  $\mathcal{I}^{c, \tau }|_S$.

The following Lemma is a direct consequence of \eqref{new_m} and of Proposition \ref{bound} (3).
\begin{lemma}\label{lemm:4.0}
  For each $q>1$, the family $\{u_c^-\}$ converges to $0$ in $W^{2,q}(\mathbb{R}^3,\mathbb{C}^4)$, as $c\to \infty$. Moreover, 
\begin{equation*}
  \| u_c^-\|_{W^{1,q}}=\mathcal{O}\left(\frac{1}{c^2}\right),\quad \| u_c^-\|_{W^{2,q}}=\mathcal{O}\left(\frac{1}{c}\right), \quad \text{as}\quad c\rightarrow \infty.
\end{equation*}
\end{lemma}
\begin{lemma}\label{stb_13}
  It holds that
  \[
  e_{\mathrm{ene}}^c = e_{\mathrm{ene}}^\infty + m c^2 + \mathcal{O}\left(\frac{1}{c^2}\right).
  \]
\end{lemma}

\begin{proof}
  From Proposition \ref{mulp}, we obtain
  \[
  e_{\mathrm{ene}}^c \leq e_{\mathrm{ene}}^\infty + m c^2 + \mathcal{O}\left(\frac{1}{c^2}\right).
  \]
  On the other hand, by Lemma \ref{lemm:4.0} and the operator inequality
  \[
  \sqrt{-c^2\Delta + m^2 c^4} \geq m c^2 - \frac{\Delta}{2m} - \frac{\Delta^2}{8m^3 c^2},
  \]
  we derive
  \begin{align*}
    e_{\mathrm{ene}}^c &= \mathcal{I}^c(u_c) = \mathcal{I}^c(P(u_c^+)) + \mathcal{O}\left(\frac{1}{c^2}\right) \\
    &\geq \mathcal{I}^\infty(P(u_c^+)) + m c^2 + \mathcal{O}\left(\frac{1}{c^2}\right) \\
    &\geq e_{\mathrm{ene}}^\infty + m c^2 + \mathcal{O}\left(\frac{1}{c^2}\right).
  \end{align*}
\end{proof}

\begin{lemma}\label{stb_14}
  \[
  \mathcal{I}^\infty(P(u_c^+)) = e_{\mathrm{ene}}^\infty + o_c(1).
  \]
\end{lemma}

\begin{proof}
  Using Lemma \ref{lemm:4.0}, Lemma \ref{stb_13}, and the asymptotic expansion
  \begin{equation}\label{rem_stb2}
        -\frac{\Delta}{2m} = \sqrt{-c^2\Delta + m^2 c^4} - m c^2 + \Delta^2 \cdot \mathcal{O}\left(\frac{1}{c^2}\right),
  \end{equation}
  we obtain
  \begin{align*}
    \mathcal{I}^\infty(P(u_c^+)) &= \mathcal{I}^c(P(u_c^+)) - m c^2 + \mathcal{O}\left(\frac{1}{c^2}\right) \\
    &= \mathcal{I}^c(u_c) - m c^2 + \mathcal{O}\left(\frac{1}{c^2}\right) \\
    &= e_{\mathrm{ene}}^\infty + \mathcal{O}\left(\frac{1}{c^2}\right).
  \end{align*}
\end{proof}

Therefore, by Lemma \ref{stb_14}, \( P(u_c^+) \) is a minimizing sequence for the functional \( \mathcal{I}^{\infty} \) on \( \mathcal{S}' \). Then, via the concentration-compactness principle, \( u_c^+ \) converges to an energy ground state of \eqref{laplace_e}. For completeness, we outline the argument below.

Define
\[
\mathcal{I}^{\infty,\tau}(f) = \frac{1}{2m} \int_{\mathbb{R}^3} |\nabla f|^2  dx - \tau^\zeta A[f], \quad \text{for } f \in H^1(\mathbb{R}^3),
\]
where
\[
A[f] = \frac{2}{p} \int_{\mathbb{R}^3} |f|^p  dx.
\]
The ground state energy is given by
\[
e^{\infty}(\tau) := \inf_{f \in \mathcal{S}'} \mathcal{I}^{\infty,\tau}(f), \quad E^{\infty}(\tau) := \tau^{\theta} e^{\infty}(\tau).
\]

\begin{remark}
  It is noted that the ground state energy $e^{\infty}(\tau)$ remains 
the same regardless of whether the domain of $\mathcal{I}^{\infty,\tau}$ is defined as 
$H^1(\mathbb{R}^3, \mathbb{C}^4)$ or $H^1(\mathbb{R}^3, \mathbb{C}^2)$. 
Therefore, this paper makes no distinction and treats its domain as $H^1$.
\end{remark}
Similarly to Lemma \ref{lemm:3.8}, we can show the subadditivity of $E^{\infty}(\tau^{1/\theta})$.
\begin{lemma}\label{subadd}
  For $\tau_1, \tau_2\in (0, 1]$, then $E^{\infty}((\tau_1 + \tau_2)^{1/\theta})< E^{\infty}(\tau_1^{1/\theta}) + E^{\infty}(\tau_2^{1/\theta}).$
\end{lemma}

%and it is  well-known that for any $s_1, s_2>0$,  $e^\prime(s_1 + s_2) < e^\prime(s_1) + e^\prime(s_2)<0.$
%It follows from the subadditivity of $e^{\infty}(s)$ that $f_c$ converge to some minimizer of $\mathcal{E}(f)$ on sphere $\|f\|_{L^2}=1.$
\begin{lemma}
  Up to translations, the sequence \(\{u_c^+\}\) is relatively compact in \(H^1\).
\end{lemma}

\begin{proof}
  Since \(\{P(u_c^+)\}\) is bounded in \(H^1\), up to a subsequence, we may assume that \(\{P(u_c^+)\}\) converges weakly in \(H^1(\mathbb{R}^3; \mathbb{C}^4)\) to some \(f_\infty \in H^1\). We now apply the concentration-compactness principle \cite{MR778970, MR778974}. Vanishing can be ruled out by observing that if it occurs, then \(\{P(u_c^+)\} \to 0\) in \(L^t\) for all \(t \in (2,6)\), which would imply
  \[
  e^{\infty}(1) = \lim_{c \to \infty} \mathcal{I}^{\infty,1}(f_c) \geq 0,
  \]
  contradicting the fact that \(e^{\infty}(1) < 0\).

  The boundedness of \(\{P(u_c^+)\}\) in \(H^1\) implies that
  \[
  \left\{ \omega_c := \frac{1}{2} d\mathcal{I}^{\infty,1}(P(u_c^+))[P(u_c^+)] = \mathcal{I}^{\infty,1}(P(u_c^+)) - \frac{p-2}{p} \|P(u_c^+)\|_{L^p}^p \right\}
  \]
  is bounded. We may assume \(\omega_c \to -\lambda < 0\) as \(n \to \infty\). Hence, \(\{P(u_c^+)\}\) is a bounded Palais–Smale sequence for the functional
  \[
  J_\lambda^\infty(f) = \mathcal{I}^{\infty,1}(f) + \lambda \|f\|_{L^2}^2, \quad \text{defined on } H^1(\mathbb{R}^3; \mathbb{C}^4).
  \]

  Then there exist a finite integer \(q \geq 1\), non-zero critical points \(\varphi_1, \dots, \varphi_q\) of \(J_\lambda^\infty\) in \(H^1(\mathbb{R}^3; \mathbb{C}^4)\) with \(\|\varphi_i\|_{L^2}^2 = \mu_i \in (0,1)\), and sequences \(\{x_c^i\} \subset \mathbb{R}^3\) for \(i = 1, \dots, q\), such that for \(i \neq j\), \(|x_c^i - x_c^j| \to \infty\) as \(c \to \infty\), and
  \[
  \left\| P(u_c^+) - \sum_{i=1}^q \varphi_i(\cdot + x_c^i) \right\|_{H^1} \to 0 \quad \text{as } c\to \infty.
  \]
  Moreover,
  \[
  \|P(u_c^+)\|_{L^2}^2 = \sum_{i=1}^q \mu_i = 1, \quad \text{and} \quad \mathcal{I}^{\infty,1}(P(u_c^+)) = \sum_{i=1}^q \mathcal{I}^{\infty,1}(\varphi_i) + o_c(1).
  \]

  Let \(\Phi_i = \varphi_i / \sqrt{\mu_i} \in \mathcal{S}'\). Then
  \begin{align*}
  E^\infty(1) &= \sum_{i=1}^q \mathcal{I}^{\infty,1}(\sqrt{\mu_i} \Phi_i) + o_c(1) 
               = \sum_{i=1}^q \mu_i \mathcal{I}^{\infty, \mu_i^{1/\theta}}(\Phi_i) \\
              &\geq \sum_{i=1}^q \mu_i \inf_{f \in \mathcal{S}'} \mathcal{I}^{\infty, \mu_i^{1/\theta}}(f)
               = \sum_{i=1}^q E^{\infty}(\mu_i^{1/\theta}),
  \end{align*}
  which contradicts the strict subadditivity condition of \(E^{\infty}(\mu^{1/\theta})\) in Lemma \ref{subadd}. Therefore, \(q = 1\), and up to translation, \(u_c^+\) converges strongly in \(H^1\) to some \(f_\infty\).
\end{proof}

\begin{remark}
     Lemma \ref{stb_14} implies $f_\infty $ is an energy ground state of  \eqref{laplace_e}.  This shows the first part of Theorem \ref{them:1.2}.
\end{remark}

  According to Lemma \ref{nar}, the last two components of $f_\infty$ are zero,  i.e., $f_\infty= (f_\infty,0)^T$, and the first 
  two components of $u_c$ will converge to $f_\infty$, while the last two components will converge to zero.

  \begin{lemma}\label{lemm:4.2}
    The families \(\{g_c\}\) and \(\{f_c\}\) converge to \(0\) and \(f_\infty\) in \(H^1(\mathbb{R}^3, \mathbb{C}^2)\), respectively. Moreover, as \(c \to \infty\),
    \[
    \|g_c\|_{H^1} = \mathcal{O}\left(\frac{1}{c}\right).
    \]
  \end{lemma}
  
  \begin{proof}
  By Lemma \ref{nar} and the boundedness of \(\{u_c\}\) in \(H^2\), we have
  \[
  \|P_\infty^- u_c - P_c^- u_c\|_{H^q} \lesssim \frac{1}{c}.
  \]
  Combining this with Lemma \ref{lemm:4.0}, we obtain
  \begin{align*}
  \|g_c\|_{H^1} 
  &\leq \|P_\infty^- u_c - P_c^- u_c\|_{H^1} + \|P_c^- u_c\|_{H^1} \\
  &\lesssim \frac{1}{c}.
  \end{align*}
  Similarly,
  \[
  \|f_c - f_\infty\|_{H^1} \leq \|u_c^+ - f_\infty\|_{H^1} + \|P_\infty^+ u_c - P_c^+ u_c\|_{H^1} \to 0.
  \]
  This completes the proof.
  \end{proof}

% by estimating the energy of $u_c$ and $h$. 

\section{Nonrelativistic Limits  of action ground state}\label{section5}
\medskip
In this section, we prove that the asymptotic behavior of the action ground state of \eqref{Dirac} 
 as \(c\to\infty\).  
To begin, we first recall that the nonrelativistic action ground state \( f_\infty\in H^1(\mathbb{R}^3, \mathbb{C}^2) \) of \eqref{laplace} can be characterized as a minimizer of the nonrelativistic action functional \( \mathcal{J}^\infty_\lambda(u) \) over the Nehari manifold.  
\[
\mathcal{N}_\lambda^\infty := \left\{ u \in H^{1}(\mathbb{R}^3, \mathbb{C}^2) \setminus \{0\} : d\mathcal{J}_{\lambda}^\infty(u)[u] = 0 \right\}.
\]
However, in the relativistic case, the classical Nehari manifold
 \[ \mathcal{N}_c:=\left\{u\in H^{1/2}(\mathbb{R}^3, \mathbb{C}^4):d\mathcal{J}_{\omega_c}^c(u)[u]=0 \right\} \]  
is not closed and may not even form a manifold, 
since the set \( \mathcal{N}_c \) fails to capture the behavior of \( \mathcal{J}_{\omega_c}^c \) 
along the negative directions. A more refined constraint is therefore required 
to characterize the minimax structure of the functional over all indefinite directions. 
This motivates the introduction of the reduced Nehari manifold,
we refer to \cite{MR2216902,SzulkinWeth} for more details. We provide a brief introduction here.

For any \( u \in E_c^+ \), set
\[
\mathcal{J}_u(v)=\mathcal{J}_{\omega_c}^c(u+v),\quad v\in E_c^-.
\]
 It is straightforward to verify that \( \mathcal{J}_u(v) \) is strictly concave, i.e.,  
\[
d^2J_u(v)[w,w] < 0, \quad w\in E_c^-.
\]  
This implies that for each \( u \), there exists a unique \( \psi_{\omega_c}^c(u)\in E_c^- \)
 that attains the maximum of the following variational problem
 \[
 \mathcal{J}_{\omega_c}^c(u+ \psi_{\omega_c}^c(u))= \sup_{v\in E_c^-}\mathcal{J}_{\omega_c}^c(u+v).
 \]
 Using this property, we define a reduced functional 
 \[
 \mathcal{J}_{\omega_c, red}^c(u)=  \mathcal{J}_{\omega_c}^c(u+ \psi_{\omega_c}^c(u)), \quad u\in E_c^+.
 \]
 and introduce the reduced Nehari manifold as:  
\[
\mathcal{N}_{\omega_c}^c: = \{ u \in E_c^+ :d\mathcal{J}_{\omega_c, red}^c(u)[u]  = 0 \}.
\]  
Moreover, it can be shown that there exists a one-to-one correspondence between the critical points of the original functional
 \( \mathcal{J}_{\omega_c}^c \) and those of \( \mathcal{J}_{\omega_c, red}^c \) via a mapping \( u\to u+ \psi_{\omega_c}^c(u) \). 
 Hence, the action ground state \( u_c = u_c^+ + \psi_{\omega_c}^c(u_c^+) \), associated with \eqref{Dirac},
 is characterized as a minimizer of the corresponding reduced action functional:  
\[
e_{\omega_c, act}^c=\mathcal{J}_{\omega_c}^c(u_c) = \mathcal{J}_{\omega_c, \text{red}}^c(u_c^+) = \inf_{u \in \mathcal{N}_{\omega_c}^c} \mathcal{J}_{\omega_c, \text{red}}^c(u).
\]  

It is noteworthy that \eqref{laplace} has the
 same ground state energy regardless of whether it is defined on 
 $H^{1}(\mathbb{R}^3, \mathbb{C}^2)$ or $H^{1}(\mathbb{R}^3, \mathbb{C}^4)$, see Appendix \ref{a_B}. 
 Therefore, we treat $f_\infty\in H^1(\mathbb{R}^3, \mathbb{C}^2)$ and $(f_\infty,0)^T\in H^{1}(\mathbb{R}^3, \mathbb{C}^4)$ as the same function 
 and regard
 it as the ground state solution when the phase space of \eqref{laplace} is $\mathbb{C}^4$.
 %\subsection{\texorpdfstring{Case $\mathrm{I}$: $\omega_c=mc^2-\lambda$}{Case I: omega-c = mc 2 - lambda}}
\begin{lemma}\label{Lem:A2}
	For each $c>1$, there exists $t_c>0$, such that
	$t_cf_\infty^+\in\mathcal{N}^c_{\omega_c}$, and there holds
    $$ t_c=1+ \mathcal{O}(\frac{1}{c^2}).$$
\end{lemma}
\begin{proof}
  For the existence of \( t_c \), we refer the reader to \cite[Lemma 3.7]{WJC}. Note that \( t_c f_\infty^+ \in \mathcal{N}^c_{\omega_c} \) if and only if
  \begin{equation}\label{as_2}
      \begin{split}
          t_c^2 &\left( \|f_\infty^+\|_c^2 - mc^2 \|f_\infty^+\|_{L^2}^2 + \lambda \|f_\infty^+\|_{L^2}^2 \right) \\
          &= \int_{\mathbb{R}^3} \left| t_c f_\infty^+ + \psi_{\omega_c}^c(t_c f_\infty^+) \right|^{p-2}
              \left( t_c f_\infty^+ + \psi_{\omega_c}^c(t_c f_\infty^+) \right) \cdot t_c f_\infty^+.
      \end{split}
  \end{equation}
  and
  \begin{equation}\label{new_q1}
      \begin{split}
          \mathcal{J}^c_{\omega_c,\text{red}}(t_c f_\infty^+)
          &= \mathcal{J}^c_{\omega_c} \left( t_c f_\infty^+ + \psi_{\omega_c}^c(t_c f_\infty^+) \right) \\
          &= t_c^2 \left( \|f_\infty^+\|_c^2 - mc^2 \|f_\infty^+\|_{L^2}^2 + \lambda \|f_\infty^+\|_{L^2}^2 \right) \\
          &\quad - \| \psi_{\omega_c}^c(t_c f_\infty^+) \|_c^2 - \omega_c \| \psi_{\omega_c}^c(t_c f_\infty^+) \|_{L^2}^2 \\
          &\quad - \frac{2}{p} \int_{\mathbb{R}^3} \left| t_c f_\infty^+ + \psi_{\omega_c}^c(t_c f_\infty^+) \right|^p > 0.
      \end{split}
  \end{equation}
  It follows \eqref{rem_stb2} and \eqref{rem_stb} that
  \begin{equation}\label{new_q3}
  \begin{split}
         \| f_\infty^+ \|_c^2 - mc^2 \| f_\infty^+ \|_{L^2}^2+\lambda \|f_\infty^+\|_{L^2}^2& = \frac{1}{2m} \| \nabla f_\infty^+ \|_{L^2}^2+ \lambda \|f_\infty^+\|_{L^2}^2 + \mathcal{O}\left(\frac{1}{c^2}\right)\\
         &=\frac{1}{2m} \| \nabla f_\infty \|_{L^2}^2 +\lambda \|f_\infty\|_{L^2}^2+ \mathcal{O}\left(\frac{1}{c^2}\right).
  \end{split}
  \end{equation}
  Combining Lemma~\ref{com} with \eqref{new_q1} and \eqref{new_q3}, we obtain
  \begin{equation}\label{new_q4}
      t_c^2 \left( \frac{1}{2m} \| \nabla f_\infty \|_{L^2}^2 + \lambda \| f_\infty \|_{L^2}^2 + o_c(1) \right)
      \geq C_p t_c^p \left( \int_{\mathbb{R}^3} |f_\infty|^p + o_c(1) \right),
  \end{equation}
  which implies that \( \limsup\limits_{c \to \infty} t_c < \infty \). Since
  \[
      \mathcal{J}^c_{\omega_c} \left( t_c f_\infty^+ + \psi_{\omega_c}^c(t_c f_\infty^+) \right) \geq \mathcal{J}^c_{\omega_c} (t_c f_\infty^+),
  \]
  it follows that
  \[
      \| \psi_{\omega_c}^c(t_c f_\infty^+) \|_c^2 \leq \frac{2}{p} t_c^p \int_{\mathbb{R}^3} |f_\infty^+|^p < \infty,
  \]
  and hence
  \begin{equation}\label{as_1}
        \begin{split}
          &t_c^2 \left( \frac{1}{2m} \| \nabla f_\infty \|_{L^2}^2 + \lambda \| f_\infty \|_{L^2}^2 + \mathcal{O}\left(\frac{1}{c^2}\right) \right) \\
          &= t_c^2 \left( \| f_\infty^+ \|_c^2 - mc^2 \| f_\infty^+ \|_{L^2}^2 + \lambda \| f_\infty^+ \|_{L^2}^2 \right) \\
          &= \int_{\mathbb{R}^3} \left| t_c f_\infty^+ + \psi_{\omega_c}^c(t_c f_\infty^+) \right|^{p-2}
              \left( t_c f_\infty^+ + \psi_{\omega_c}^c(t_c f_\infty^+) \right) \cdot t_c f_\infty^+ \\
          &= t_c^p \left( \int_{\mathbb{R}^3} |f_\infty|^p + o_c(1) \right) \\
          &= t_c^p \left( \frac{1}{2m} \| \nabla f_\infty \|_{L^2}^2 + \lambda \| f_\infty \|_{L^2}^2 + o_c(1) \right),
      \end{split}
  \end{equation}
  which yields 
  $$
  t_c= \left(1+\mathcal{O}\left(\frac{1}{c^2}\right) \right)^{1/(p-2)}= 1+ \mathcal{O}\left(\frac{1}{c^2}\right).
  $$
\end{proof}

\begin{lemma}\label{rem_stb3}
    $\limsup\limits_{c\to \infty}e_{\omega_c, act}^c\leq e_{\lambda,act}^\infty$.
\end{lemma}
\begin{proof}
    It follows from Lemma \ref{Lem:A2} and \eqref{new_q3} that we get
    \begin{equation}
        \begin{split}
         e_{\omega_c, act}^c&\leq \mathcal{J}_{\omega_c, red}^c(t_c f_\infty^+)  =\mathcal{J}^c_{\omega_c}(t_c f_\infty^++ \psi_{\omega_c}^c(t_c f_\infty^+))\\
          &=
          t_c^2\left(\|f_\infty^+\|_c^2  -mc^2\|f_\infty^+\|_{L^2}^2+ \lambda\|f_\infty^+\|_{L^2}^2 \right)- \|\psi_{\omega_c}^c(t_c f_\infty^+)\|_c^2\\
          &\quad -(mc^2-\lambda)\|\psi_{\omega_c}^c(t_c f_\infty^+)\|_{L^2}^2 - \frac{2}{p}\int_{\mathbb{R}^3}|t_c f_\infty^++ \psi_{\omega_c}^c(t_c f_\infty^+)|^p\\
          &\leq\frac{1}{2m}\|\nabla f_\infty\|_{L^2}^2 + \lambda\|f_\infty\|_{L^2}^2 -\frac{2}{p}\int_{\mathbb{R}^3}| f_\infty|^p +o_c(1)\\
          &= e_{\lambda,act}^{\infty}+ o_c(1).
        \end{split}
    \end{equation}
\end{proof}
\begin{rem*}
    In the proof of Lemma \ref{rem_stb3}, we have omitted the convergence rate of the term \( o_c(1) \), which will be substantiated at the end of this section; see Lemma \ref{s_stb5}.

\end{rem*}
\begin{lemma}\label{A_51}
  For each $q>1$,  $\sup\limits_{c>1}\|u_c\|_{W^{2,q}}<\infty$.

\end{lemma}
\begin{proof}
  Since 
  \[
  e_{\omega_c, \text{act}}^c = \mathcal{J}^c_{\omega_c}(u_c) = \frac{p-2}{p} \int_{\mathbb{R}^3} |u_c|^p < \infty,
  \]
  and from the fact that
  \[
  \mathcal{J}^c_{\omega_c}(u_c) = \sup_{v \in E_c^-} \mathcal{J}^c_{\omega_c}(u_c^+ + v) \geq \mathcal{J}^c_{\omega_c}(u_c^+),
  \]
  we conclude that
  \[
 \|u_c^-\|_c^2= \|\psi_{\omega_c}^c(u_c^+)\|_c^2 \leq \frac{2}{p} \int_{\mathbb{R}^3} |u_c^+|^p < \infty.
  \]
  Furthermore, from the identity
  \begin{equation}
      \begin{split}
          e_{\omega_c, \text{act}}^c = \mathcal{J}^c_{\omega_c}(u_c) = \frac{p-2}{p} \Big( 
          &\|u_c^+\|_c^2 - mc^2 \|u_c^+\|_{L^2}^2 + \lambda \|u_c^+\|_{L^2}^2 \\
          &- \|\psi_{\omega_c}^c(u_c^+)\|_c^2 - \omega_c \|\psi_{\omega_c}^c(u_c^+)\|_{L^2}^2 \Big),
      \end{split}
  \end{equation}
  we obtain
  \[
  \sup_{c > 1} \left( \|u_c^+\|_c^2 - mc^2 \|u_c^+\|_{L^2}^2 + \lambda \|u_c^+\|_{L^2}^2 \right) < \infty,
  \]
  which implies
  \[
  \sup_{c > 1} \|u_c^+\|_{H^{1/2}}^2 < \infty,
  \]
  and hence
  \[
  \sup_{c > 1} \|u_c\|_{H^{1/2}}^2 < \infty.
  \]
  Therefore, by Proposition~\ref{bound} (2), we conclude that
  \[
  \sup_{c > 1} \|u_c\|_{W^{2,q}} < \infty.
  \]
\end{proof}
Analogously to Lemma \ref{lemm:4.0} and Lemma \ref{lemm:4.2}, the following decay properties hold for 
the negative part of the action ground state of \eqref{Dirac}.
\begin{lemma}\label{A_21}
  For each $q>1$, there holds
    $$\|\psi_{\omega_c}^c(u_c^+)\|_{W^{1,q}}= \|u_c^-\|_{W^{1,q}}= \mathcal{O}\left(\frac{1}{c^2}\right),$$
     and $$\| u_c^-\|_{W^{2,q}}= \mathcal{O}\left(\frac{1}{c}\right), \quad \|g_{c}\|_{H^1}=\mathcal{O}\left(\frac{1}{c}\right).$$
\end{lemma}

\begin{lemma}\label{A_3}
    There exists $r_c >0$, such that $r_c u_c^+ \in \mathcal{N}^\infty_\lambda$, and $$r_c = 1 + \mathcal{O}\left(\frac{1}{c^2}\right).$$
\end{lemma}
\begin{proof}
  By definition, we have
  \begin{equation}
      r_c^{2-p} = \frac{\int_{\mathbb{R}^3} |u_c^+|^p \, dx}{\dfrac{1}{2m} \|\nabla u_c^+\|_{L^2}^2 + \lambda \|u_c^+\|_{L^2}^2}.
  \end{equation}
  Applying Lemma \ref{A_21} and the asymptotic expansion \eqref{rem_stb2},
  we obtain
  \begin{equation*}
      \begin{split}
          \frac{1}{2m} \|\nabla u_c^+\|_{L^2}^2 + \lambda \|u_c^+\|_{L^2}^2 
          &= \|u_c^+\|_c^2 - mc^2 \|u_c^+\|_{L^2}^2 + \lambda \|u_c^+\|_{L^2}^2 + \mathcal{O}\left(\frac{1}{c^2}\right) \\
          &= \int_{\mathbb{R}^3} |u_c|^p \, dx + \omega_c \|u_c^-\|_{L^2}^2 + \|u_c^-\|_c^2 + \mathcal{O}\left(\frac{1}{c^2}\right) \\
          &= \int_{\mathbb{R}^3} |u_c^+|^p \, dx + \mathcal{O}\left(\frac{1}{c^2}\right).
      \end{split}
  \end{equation*}
  Therefore,
  \[
      r_c = \left(1 + \mathcal{O}\left(\frac{1}{c^2}\right)\right)^{\frac{1}{2-p}} = 1 + \mathcal{O}\left(\frac{1}{c^2}\right).
  \]
\end{proof}

\begin{lemma}\label{rem_stb5}
    $\lim\limits_{c\to \infty}\mathcal{J}^\infty_\lambda(u_c^+)=e_{\lambda, act}^\infty$, and $e_{\lambda, act}^\infty\leq e_{\omega_c, act}^c+ \mathcal{O}\left(\frac{1}{c^2}\right)$.
\end{lemma}
\begin{proof}
  By Lemma \ref{A_3}, we have
  \begin{equation}\label{as_4}
      \begin{split}
          e_{\lambda, \text{act}}^\infty \leq \liminf_{c \to \infty} \mathcal{J}^\infty_\lambda(r_c u_c^+)
          = \liminf_{c \to \infty} \mathcal{J}^\infty_\lambda(u_c^+).
      \end{split}
  \end{equation}
  Since
  \begin{equation}\label{as_7}
      -\frac{1}{2m} \Delta \leq \sqrt{-c^2 \Delta + m^2 c^4} - m c^2 + \frac{\Delta^2}{8m^3 c^2},
  \end{equation}
  it follows from Lemma \ref{A_3} and Lemma \ref{A_21} that
  \begin{equation}\label{as_5}
      \begin{split}
          \mathcal{J}^\infty_\lambda(u_c^+) 
          &\leq \|u_c^+\|_c^2 - m c^2 \|u_c^+\|_{L^2}^2 + \lambda \|u_c^+\|_{L^2}^2 - \frac{2}{p} \int_{\mathbb{R}^3} |u_c^+|^p \\
          &= \mathcal{J}^c_{\omega_c}(u_c) + \mathcal{O}\left(\frac{1}{c^2}\right) \\
          &= e_{\omega_c, \text{act}}^c + \mathcal{O}\left(\frac{1}{c^2}\right).
      \end{split}
  \end{equation}
  Therefore,
  \[
      e_{\lambda, \text{act}}^\infty \leq e_{\omega_c, \text{act}}^c + \mathcal{O}\left(\frac{1}{c^2}\right).
  \]
  On the other hand, by Lemma \ref{Lem:A2}, along with \eqref{as_4} and \eqref{as_5}, we obtain
  \[
      e_{\lambda, \text{act}}^\infty \leq \liminf_{c \to \infty} \mathcal{J}^\infty_\lambda(u_c^+)
      \leq e_{\omega_c, \text{act}}^c + \mathcal{O}\left(\frac{1}{c^2}\right)
      \leq e_{\lambda, \text{act}}^\infty + \mathcal{O}\left(\frac{1}{c^2}\right),
  \]
  which implies
  \[
      \mathcal{J}^\infty_\lambda(u_c^+) = e_{\lambda, \text{act}}^\infty + o_c(1).
  \]
\end{proof}

\begin{lemma}\label{Lem:A8}
	$\{u_c^+\}$ is a bounded Palais-Smale sequence for the functional
	$\mathcal{J}^\infty_\lambda(u)$ at level $e_{\lambda, act}^\infty.$
\end{lemma}

\begin{proof}
    It follows from \eqref{as_7}, Lemma \ref{A_21} and Lemma \ref{A_51} that
    \begin{equation*}
        \begin{split}
            	\sup_{\|h\|_{H^1} \leq 1} |d\mathcal{J}^\infty_\lambda(u_c^+)[h]| &\leq
		\sup_{\|h\|_{H^1} \leq 1} |d\mathcal{J}^\infty_\lambda(u_c^+)[h] -
		d\mathcal{J}^c_{\omega_c}(u_c^+)[h]| \\
&\quad +\sup_{\|h\|_{H^1} \leq 1} |d\mathcal{J}^c_{\omega_c}(u_c)[h] -
		d\mathcal{J}^c_{\omega_c}(u_c^+)[h]|
        \\
		&\leq 2\sup_{\|h\|_{H^1} \leq 1} 
        \int_{\mathbb{R}^3}\left|\left(\sqrt{-c^2\Delta+m^2c^4}-mc^2+ \frac{1}{2m}\Delta\right)u_c^+ \cdot h\right| +o_c(1)
        \\
		&= o_c(1).
        \end{split}
    \end{equation*}
\end{proof}
\begin{lemma}\label{A_22}
    There exists a nonrelativistic action ground state $f_\infty$, such that up to translation and subsequence, there holds
    $$
    \|u_c^+- f_\infty\|_{H^1}\to 0.
    $$
\end{lemma}
\begin{proof}
  If $\{u_c^+\}$ is vanishing, i.e.,
  \[
      \lim_{c \to \infty} \sup_{y \in \mathbb{R}^3} \int_{|x - y| < R} |u_c^+|^2  dx = 0 \quad \text{for all } R > 0,
  \]
  then $\|u_c\|_{L^t} \to 0$ for every $t \in (2,6)$. Consequently,
  \[
      e_{\lambda, \text{act}}^\infty = e_{\lambda,act}^{c} + o_c(1) = \frac{p - 2}{p} \|u_c\|_{L^p}^p \to 0,
  \]
  which leads to a contradiction.

  Therefore, by the concentration-compactness principle and Lemma \ref{Lem:A8},
   there exist a finite integer $q > 1$, 
  non-zero critical points $v_1, \dots, v_q$ of $\mathcal{J}_{\lambda}^\infty$ 
  in $H^1(\mathbb{R}^3, \mathbb{C}^4)$, and sequences $\{x_c^i\} \subset \mathbb{R}^3$ 
  for $i = 1, \dots, q$, such that for $i \neq j$, $|x_c^i - x_c^j| \to \infty$ as $c \to \infty$, and
  \[
      \left\| u_c^+ - \sum_{i=1}^q v_i(\cdot - x_c^i) \right\|_{H^1} \to 0 \quad \text{as } c \to \infty.
  \]
  It follows that
  \[
      e_{\lambda,act}^{\infty} + o_c(1) = \mathcal{J}^\infty_\lambda(u_c^+) = \sum_{i=1}^q \mathcal{J}^\infty_\lambda(v_i) \geq q  e_{\lambda, \text{act}}^\infty,
  \]
  which implies $q = 1$.
\end{proof}
We now proceed to the proof of (3) of Theorem \ref{them:1.3}. We adopt a proof strategy similar to that of Proposition \ref{mulp}. Due to the methodological similarity, we provide only an outline of the proof here; the detailed argument may be found in the proof of Proposition \ref{mulp}.  
  
First, analogous to Lemma \ref{new_b} and Lemma \ref{stb_12}, by considering the linearization of the functional $\mathcal{J}^c_{\omega_c}$ at \( t_c f_\infty^+ + \psi_{\omega_c}^c(t_c f_\infty^+) \) and the linearization of \(\mathcal{J}_\lambda^\infty\) at \( f_\infty \), we obtain the following lemma.
\begin{lemma}\label{s_stb3}
The following estimates hold:
\begin{equation}\label{a_stb}
    \mathcal{J}^c_{\omega_c}\bigl(t_c f_\infty^+ + \psi_{\omega_c}^c(t_c f_\infty^+)\bigr) \leq \mathcal{J}^c_{\omega_c}(t_c f_\infty^+) + \mathcal{O}\biggl(\frac{1}{c^2}\biggr),
\end{equation}
and
\begin{equation}\label{a_stb1}
    \mathcal{J}_\lambda^\infty(f_\infty^+) \leq \mathcal{J}_\lambda^\infty(f_\infty) + \mathcal{O}\biggl(\frac{1}{c^2}\biggr).
\end{equation}
\end{lemma}

\begin{lemma}\label{s_stb5}
    It holds that 
    \[
    e_{\omega_c, \mathrm{act}}^c = e_{\lambda}^\infty + \mathcal{O}\biggl(\frac{1}{c^2}\biggr).
    \]
\end{lemma}

\begin{proof}
By Lemma \ref{rem_stb5}, it suffices to prove
\begin{equation}\label{rem_stb6}
    e_{\omega_c, \mathrm{act}}^c \leq e_{\lambda}^\infty + \mathcal{O}\biggl(\frac{1}{c^2}\biggr).
\end{equation}
From \eqref{a_stb}, we deduce
\begin{equation}
\begin{split}
    e_{\omega_c, \mathrm{act}}^c 
    &= \mathcal{J}^c_{\omega_c}(u_c) 
    = \mathcal{J}^c_{\omega_c, \mathrm{red}}(u_c^+) \\
    &\leq \mathcal{J}^c_{\omega_c, \mathrm{red}}(t_c f_\infty^+) \\
    &= \mathcal{J}^c_{\omega_c}\bigl(t_c f_\infty^+ + \psi_{\omega_c}^c(t_c f_\infty^+)\bigr) \\
    &\leq \mathcal{J}^c_{\omega_c}(t_c f_\infty^+) + \mathcal{O}\biggl(\frac{1}{c^2}\biggr).
\end{split}
\end{equation}
Combining Lemma \ref{Lem:A2} with \eqref{rem_stb2} and \eqref{a_stb1}, we obtain
\begin{equation*}
\begin{split}
    e_{\omega_c, \mathrm{act}}^c 
    &\leq \mathcal{J}^c_{\omega_c}(t_c f_\infty^+) + \mathcal{O}\biggl(\frac{1}{c^2}\biggr) \\
    &\leq \mathcal{J}_\lambda^\infty(f_\infty^+) + \mathcal{O}\biggl(\frac{1}{c^2}\biggr) \\
    &\leq \mathcal{J}_\lambda^\infty(f_\infty) + \mathcal{O}\biggl(\frac{1}{c^2}\biggr) \\
    &= e_\lambda^\infty + \mathcal{O}\biggl(\frac{1}{c^2}\biggr),
\end{split}
\end{equation*}
which completes the proof.
\end{proof}

\begin{proof}[\textbf{Proof of Theorem \ref{them:1.4} } ]
  By repeating the argument of Lemma \ref{lemm:4.2}, the convergence of $f_c$ is obtained. The proof is then completed by Lemma \ref{A_21}, Lemma \ref{A_22} and Lemma \ref{s_stb5}.
\end{proof}

\begin{comment}
\subsection{\texorpdfstring{Case $\mathrm{II}$: $\omega_c=-(mc^2-\lambda)$}{Case II: omega-c = mc 2 - lambda}}
\begin{proof}[\textbf{Proof of Theorem \ref{them:1.4} (2)} ]
  It follows from 
  \[
  \mathcal{J}^c_{\omega_c}(u_c) = \sup_{v \in E_c^-} \mathcal{J}^c_{\omega_c}(u_c^+ + v) \geq \mathcal{J}^c_{\omega_c}(u_c^+),
  \]
  that  
  \[
  \frac{2}{p} \int_{\mathbb{R}^3} |u_c|^p \leq \frac{2}{p} \int_{\mathbb{R}^3} |u_c^+|^p.
  \]
  Moreover, since \( d\mathcal{J}_{\omega_c}^c(u_c)[u_c^+] = 0 \), we obtain
  \begin{equation*}
    \begin{split}
      c\|u_c^+\|_{H^{1/2}}^2 
      &\leq \int_{\mathbb{R}^3} \left( \sqrt{-c^2\Delta + m^2c^4} + mc^2 - \lambda \right) u_c^+ \cdot u_c^+ \\
      &= \int_{\mathbb{R}^3} |u_c|^{p-2} \, \Re\langle u_c, u_c^+ \rangle \\
      &\leq C_p \|u_c^+\|_{L^p}^p \\
      &\leq C_p \|u_c^+\|_{H^{1/2}}^p,
    \end{split}
  \end{equation*}
  which implies  
  \[
  \lim_{c \to \infty} \|u_c\|_{H^{1/2}} \geq \lim_{c \to \infty} \|u_c^+\|_{H^{1/2}} = \infty.
  \]
\end{proof}
\end{comment}
%\noindent{\bf Acknowledgments.}
%The first author is supported by NSFC No. 12201625, No. 12031015.  
\section{Equivalence of action and energy ground state}\label{section6}
Inspired by \cite[Theorem 1.3]{DST23}, this section proves the consistency 
between the action and  energy ground 
state of the Dirac equation. The notation in this section 
follows that of Theorem \ref{them:1.4}, where \( c \in (c_0,+\infty)\backslash \Xi \), \( AG_{\omega_c} \) 
denotes the set of action ground state of \eqref{Dirac},
 and \( EG_c \) represents the set of energy ground state of \eqref{Dirac_e}.
 \begin{lemma}\label{new_p}
  Suppose the sequence $\omega_c$ satisfies
  $$
  -\infty < \liminf_{c \to \infty} (\omega_c - m c^2) \leq \limsup_{c \to \infty} (\omega_c - m c^2) < 0.
  $$
  Then for all sufficiently large $c$, every action ground state $u$ of \eqref{Dirac} satisfies $P(u^+) \in \mathcal{O}_c^+$.
\end{lemma}

\begin{proof}
  By Theorem \ref{them:1.3}, the action ground state $u = u^+ + \psi_{\omega_c}^c(u)$ satisfies
   all conditions of Proposition \ref{bound}. Hence, it follows that \(P(u^+) \in \mathcal{O}_c^+\).
\end{proof}
\begin{lemma}\label{ine}
  For each $\lambda>0, \omega_c= mc^2-\lambda$, $u\in  \mathcal{N}_{\omega_c}^c$, $P(u)\in \mathcal{O}_c^+$ there  holds
  \begin{equation}\label{eq_11}
       \mathcal{J}^c_{\omega_c,red}(u)\geq e_{ene}^c-\omega_c
  \end{equation}
  and equality in \eqref{eq_11} holds if and only if
  $$
 \varphi^c(P(u))=u + \psi_{\omega_c}^c(u)
  $$
  is both an energy ground for $\mathcal{I}^c$
on $\mathcal{S}$ and an action ground state
for $\mathcal{J}^c_{\omega_c}$.
\end{lemma}
\begin{proof}
For each $t>0$ and $v \in E_c^-$, we have
\begin{equation}
  \begin{split}
      \mathcal{J}^c_{\omega_c,\mathrm{red}}(u) \geq \mathcal{J}^c_{\omega_c,\mathrm{red}}(tP(u)) \geq \mathcal{J}^c_{\omega_c}(tP(u) + v).
  \end{split}
\end{equation}
Setting $t = \sqrt{1 - \|v\|_{L^2}^2}$, it follows from the identity
\[
\mathcal{J}^c_{\omega_c}(u) = \mathcal{I}^c(u) - \omega_c\|u\|_{L^2}^2
\]
that
\begin{equation}\label{eq_12}
  \begin{split}
      \mathcal{J}^c_{\omega_c,\mathrm{red}}(u) &\geq \sup_{v \in E_c^-} \left[ \mathcal{I}^c\left( \sqrt{1 - \|v\|_{L^2}^2} P(u) + v \right) -\omega_c \right] \\
      &= \sup_{w \in S(u)} \mathcal{I}^c(w) -\omega_c \\
      &\geq \inf_{g \in \mathcal{O}_c^+} \sup_{w \in S(g)} \mathcal{I}^c(w) -\omega_c \\
      &= e_{\mathrm{ene}}^c -\omega_c.
  \end{split}
\end{equation}

If $u + \psi_{\omega_c}^c(u)$ is an action ground state for $\mathcal{J}^c_{\omega_c}$ satisfying $\|u + \psi_{\omega_c}^c(u)\|_{L^2} = 1$, then the first inequality in \eqref{eq_12} becomes an equality.
If $ \varphi^c(P(u))$ is an energy ground state for $\mathcal{I}^c$ on $\mathcal{S}$, then the second inequality in \eqref{eq_12} becomes an equality. Consequently,
$
  \mathcal{J}^c_{\omega_c,\mathrm{red}}(u) = e_{\mathrm{ene}}^c -\omega_c.
$

Conversely, suppose that for some $u \in \mathcal{N}_{\omega_c}^c,\,P(u)\in \mathcal{O}_c^+$, the equality
\begin{equation*}
  \mathcal{J}^c_{\omega_c,\mathrm{red}}(u) = e_{\mathrm{ene}}^c -\omega_c
\end{equation*}
holds. Then the first inequality in \eqref{eq_12} is an equality, implying $\sqrt{1 - \|v\|_{L^2}^2}P(u) = u$ and $v = \psi_{\omega_c}^c(u)$, and hence $\|u + \psi_{\omega_c}^c(u)\|_{L^2} = 1$.
Since the second inequality in \eqref{eq_12} is also an equality, the uniqueness of $\varphi^c(P(u))$ implies that
\[
u + \psi_{\omega_c}^c(u) = \varphi^c(P(u)),
\]
and $\varphi^c(P(u))$ is indeed an energy ground state for $\mathcal{I}^c$ on $\mathcal{S}$.
Furthermore, $u + \psi_{\omega_c}^c(u)$ must be an action ground state of $\mathcal{J}_{\omega_c}^c$. 
Otherwise, there would exist $v \in \mathcal{N}_{\omega_c}^c$, $v \neq u$, and $v+\psi_{\omega_c}^c(v)$ is an action ground state of $\mathcal{J}_{\omega_c}^c$.
By Lemma \ref{new_p}, \(P(v) \in \mathcal{O}_c^+\). Thus, we have
\[
\mathcal{J}_{\omega_c,\mathrm{red}}^c(v) < \mathcal{J}^c_{\omega_c,\mathrm{red}}(u) = e_{\mathrm{ene}}^c -\omega_c,
\]
contradicting \eqref{eq_11}.
\end{proof}

\begin{proof}[\textbf{Proof of Theorem \ref{them:1.4}}]
For $u\in EG_c$ with the unique Lagrange multiplier $\omega_c $,  
$v\in \mathcal{N}_{\omega_c}^c$, and $v+\psi_{\omega_c}^c(v)$ is an action ground state of \eqref{Dirac}. By Lemma \ref{new_p} and Lemma \ref{ine},
we get
  \begin{equation}\label{eq_111}
  \begin{split}
       \mathcal{J}^c_{\omega_c,red}(v)\geq e_{ene}^c-\omega_c= \mathcal{J}_{\omega_c}^c(u)
  \end{split}
  \end{equation}
  which yield $u$ is an action ground state for $\mathcal{J}_{\omega_c}^c$. Hence 
  $
  EG_c\subset AG_{\omega_c}.
  $
On the other hand, if $w$ is an action state for $\mathcal{J}_{\omega_c}^c$, then $P(w^+)\in \mathcal{O}_c^+$ and
$$
\mathcal{J}_{\omega_c}^c(w)=\mathcal{J}^c_{\omega_c,red}(w^+)=\mathcal{J}^c_{\omega_c,red}(u^+)=\mathcal{J}_{\omega_c}^c(u)=e_{ene}^c-\omega_c.
$$
By Lemma \ref{ine},
$$
w=\varphi^c(P(w^+))=w^+ + \psi_{\omega_c}^c(w^+)
$$
which implies $w\in \mathcal{S}$ and
$$
\mathcal{I}^c(w)=\mathcal{J}_{\omega_c}^c(w)+ \omega_c=e_{ene}^c,
$$
hence $w\in EG_c$. Consequently, $EG_c= AG_{\omega_c}$.
\end{proof}

  For the nonrelativistic limit of the energy ground state \eqref{Dirac_e}, 
  an alternative proof approach can
   be provided by combining Theorem \ref{them:1.3} and Theorem \ref{them:1.4}.
   Here, we outline the proof as follows, see also Figure \ref{fig:diagram_2}. First, according to Theorem \ref{them:1.4}, 
   the energy ground state of \eqref{Dirac_e} is also the action ground state 
   of \eqref{Dirac}, where \( \omega_c \) is the Lagrange multiplier corresponding
    to the energy ground state. Then, by \eqref{new_m} and Theorem \ref{them:1.3}, this solution
     converges to the action ground state \( f_\infty \) of \eqref{laplace} with \(  \lambda=\lim\limits_{c\to\infty}mc^2- \omega_c >0 \). 
     Due to the uniqueness of the ground state of \eqref{laplace} (see Remark \ref{rem_1.7} (3)), \( f _\infty\) is also an energy ground state of
       \eqref{laplace_e}. Therefore, the energy ground state of \eqref{Dirac_e}  converge to that of \eqref{laplace_e}.

\begin{figure}[htbp]
  
\centering
\begin{tikzpicture}[
  every node/.style={text width=3cm, align=center},
  arrow/.style={-Stealth, thick}
]
\node (de) at (0,0) {energy ground state of \eqref{Dirac}};
\node (dc) at (3.5,0) {action ground state of \eqref{Dirac}};
\node (lc) at (7,0) {action ground state of \eqref{laplace}};
\node (le) at (10.5,0) {energy ground state of \eqref{laplace_e}};

  \draw[->,  thick] (de) -- node[above, black] {} (dc);
  \draw[->,  thick] (dc) -- node[below, black] {} (lc);
  \draw[->,  thick] (lc) -- node[above, black] {} (le);
  \draw[->, red, thick] (de) to[bend left=15] (le);
\end{tikzpicture}
\caption{ Proof Roadmap for the Convergence of Energy Ground State}
\label{fig:diagram_2}
\end{figure}
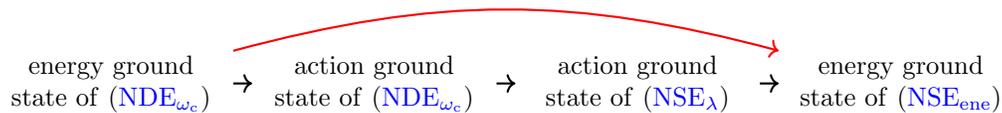

\appendix

\section{Ground states of Nonlinear Schr\"odinger equation}\label{a_B}
In this appendix, we prove that the ground state energy remains the same whether the wave function in equation \eqref{laplace} belongs to \( H^1(\mathbb{R}^3, \mathbb{R}) \), \( H^1(\mathbb{R}^3, \mathbb{C}^2) \), or \( H^1(\mathbb{R}^3, \mathbb{C}^4) \). Furthermore, for the two complex cases (i.e., \( \mathbb{C}^2 \) and \( \mathbb{C}^4 \)), the ground state solution can be generated from the unique radially symmetric real-valued solution of \eqref{laplace} when \( f \in H^1(\mathbb{R}^3, \mathbb{R}) \). 

For convenience, we set \( \lambda = 1 \) and \( m= \frac{1}{2}\) and define the functional
\[
\mathcal{J}^\infty_{\lambda, i}(f) = \int_{\mathbb{R}^3} |\nabla f|^2  + \int_{\mathbb{R}^3}|f|^2- \frac{2}{p} \int_{\mathbb{R}^3} |f|^p,
\]
for \( i = 1,2,3 \), where \( f \in H^1(\mathbb{R}^3, \mathbb{R}) \) when \( i=1 \), \( f \in H^1(\mathbb{R}^3, \mathbb{C}^2) \) when \( i=2 \), and \( f \in H^1(\mathbb{R}^3, \mathbb{C}^4) \) when \( i=3 \).
Let \( \mathcal{N}^\infty_{\lambda, i} \) denote the Nehari manifold associated with \( \mathcal{J}^\infty_{\lambda, i} \) for \( i = 1, 2, 3 \), and define the corresponding ground state energy levels as:

\[
e^\infty_{\lambda,i} = \inf_{f \in \mathcal{N}^\infty_{\lambda, i}} \mathcal{J}^\infty_{\lambda, i}(f).
\]

We then have the following lemma:

\begin{lemma}\label{B_5}
The ground state energies satisfy  
\[
e^\infty_{\lambda,1} = e^\infty_{\lambda,2} = e^\infty_{\lambda,3}.
\]
Moreover, the action ground state of $\mathcal{J}^\infty_{\lambda, 2}$ and  $\mathcal{J}^\infty_{\lambda, 3}$ can be generated by the action ground state of $\mathcal{J}^\infty_{\lambda, 1}$ through the action of $SU(2)$ and $SU(4)$.
\end{lemma}

\begin{proof}
Let  $f \in H^1(\mathbb{R}^3, \mathbb{R})$ and $(f_1, f_2) \in H^1(\mathbb{R}^3, \mathbb{C}^2)$ be action ground states of $\mathcal{J}^\infty_{\lambda, 1}$ and $\mathcal{J}^\infty_{\lambda, 2}$. Then the ground state energy is given by:
\begin{align*}
e^\infty_{\lambda,2} &= \left(1 - \frac{2}{p}\right) \int_{\mathbb{R}^3} \big(|\nabla f_1|^2 + |\nabla f_2|^2 + |f_1|^2 + |f_2|^2\big) \, dx.
\end{align*}

By \cite[Theorem 6.17]{MR1817225}, we have the inequality
\[
\int_{\mathbb{R}^3} \left| \nabla \sqrt{|f_1|^2 + |f_2|^2} \,\right|^2 dx \leq \int_{\mathbb{R}^3} \left( |\nabla f_1|^2 + |\nabla f_2|^2 \right) dx.
\]

This leads to
\begin{equation}\label{B_1}
\int_{\mathbb{R}^3} \left[ \left| \nabla \sqrt{|f_1|^2 + |f_2|^2} \,\right|^2 + \left( \sqrt{|f_1|^2 + |f_2|^2} \,\right)^2 \right] dx - \int_{\mathbb{R}^3} \left( \sqrt{|f_1|^2 + |f_2|^2} \,\right)^p dx \leq 0.
\end{equation}

Now choose $t > 0$ such that $t \sqrt{|f_1|^2 + |f_2|^2} \in \mathcal{N}^\infty_{\lambda,1}$, i.e.,
\[
t^2 \int_{\mathbb{R}^3} \left( \left| \nabla \sqrt{|f_1|^2 + |f_2|^2} \,\right|^2 + |f_1|^2 + |f_2|^2 \right) dx = t^p \int_{\mathbb{R}^3} \big( |f_1|^2 + |f_2|^2 \big)^{\frac{p}{2}} dx.
\]

From \eqref{B_1}, it follows that $t \leq 1$. Consequently,
\begin{equation}\label{B_2}
    \begin{split}
        e^\infty_{\lambda,1} 
&\leq \mathcal{J}^\infty_{\lambda,1} \left( t \sqrt{|f_1|^2 + |f_2|^2} \right) \\
&= t^2 \int_{\mathbb{R}^3} \left( \left| \nabla \sqrt{|f_1|^2 + |f_2|^2} \,\right|^2 + |f_1|^2 + |f_2|^2 \right) dx - \frac{2t^p}{p} \int_{\mathbb{R}^3} \big( |f_1|^2 + |f_2|^2 \big)^{\frac{p}{2}} dx \\
&= \left( 1 - \frac{2}{p} \right) t^2 \int_{\mathbb{R}^3} \left( \left| \nabla \sqrt{|f_1|^2 + |f_2|^2} \,\right|^2 + |f_1|^2 + |f_2|^2 \right) dx \\
&\leq \left( 1 - \frac{2}{p} \right) t^2 \int_{\mathbb{R}^3} \left( |\nabla f_1|^2 + |\nabla f_2|^2 + |f_1|^2 + |f_2|^2 \right) dx \\
&= t^2 \, e^\infty_{\lambda,2} \leq e^\infty_{\lambda,2}.
    \end{split}
\end{equation}
On the other hand, since $H^1(\mathbb{R}^3, \mathbb{R})$ can be embedded into $H^1(\mathbb{R}^3, \mathbb{C}^2)$, we have $e^\infty_{\lambda,1} \geq e^\infty_{\lambda,2}$. Hence, $e^\infty_{\lambda,1} = e^\infty_{\lambda,2}$. A similar argument shows that $e^\infty_{\lambda,1} = e^\infty_{\lambda,3}$.

It is easy to see that equality holds in \eqref{B_2} if and only if 
\begin{equation}\label{B_3}
    \int_{\mathbb{R}^3} \left| \nabla \sqrt{|f_1|^2 + |f_2|^2} \,\right|^2 dx = \int_{\mathbb{R}^3} \left( |\nabla f_1|^2 + |\nabla f_2|^2 \right) dx,
\end{equation}
 and up to translation,
 \begin{equation}\label{B_4}
     f= \sqrt{|f_1|^2+ |f_2|^2}.
 \end{equation}
By \cite[Theorem 6.17]{MR1817225}, \eqref{B_3} is satisfied if and only if the real and imaginary parts of \( f_1 \) and \( f_2 \) are proportional, i.e., \(\Re f_1, \Im f_1, \Re f_2, \Im f_2\) are linearly dependent. Moreover, from \eqref{B_4}, we conclude that there exist constants \( a, b \in \mathbb{C} \) with \( |a|^2 + |b|^2 = 1 \) such that \( (f_1, f_2) = (a f, b f) \). This implies that \( (f_1, f_2)^T = \gamma \cdot (f, 0)^T \) for some \( \gamma \in \mathrm{SU}(2) \). Similarly, the action ground state of \( \mathcal{J}^\infty_{\lambda, 3} \) can be generated from the action ground state of \( \mathcal{J}^\infty_{\lambda, 1} \) through the action of \( \mathrm{SU}(4) \).
\end{proof}
\begin{rem*}
    By repeating the proof of Lemma \ref{B_5}, we can show that the ground state energies of the functionals \( \mathcal{J}^\infty_{\lambda, i}(f) - \|f\|_{L^2}^2 \) on the \( L^2 \)-unit sphere are also identical.
\end{rem*}

\noindent\small {\bf \small  Declarations of interest}: none.

\noindent{\bf \small Data availability statement:} There are no new data associated with this article.

\bibliographystyle{plain}
\bibliography{Dirac}

\begin{thebibliography}{10}

\bibitem{MR2216902}
N.~Ackermann.
\newblock A nonlinear superposition principle and multibump solutions of periodic {S}chr\"odinger equations.
\newblock {\em J. Funct. Anal.}, 234(2):277--320, 2006.

\bibitem{Bella18}
J.~Bellazzini, V.~Georgiev, and N.~Visciglia.
\newblock Long time dynamics for semi-relativistic {NLS} and half wave in arbitrary dimension.
\newblock {\em Mathematische Annalen}, 371:707--740, 2018.

\bibitem{arXiv:2505.05917}
P.~Chen, V.~Coti~Zelati, and Y.~Wei.
\newblock Asymptotic properties of non-relativistic limit for pseudo-relativistic {Hartree} equations.
\newblock Preprint, {arXiv}:2505.05917, 2025.

\bibitem{CDGW24}
P.~Chen, Y.~Ding, Q.~Guo, and H.~Wang.
\newblock Nonrelativistic limit of normalized solutions to a class of nonlinear {D}irac equations.
\newblock {\em Calc. Var. Partial Differential Equations}, 63(4):Paper No. 90, 2024.

\bibitem{MR4091059}
W.~Choi, Y.~Hong, and J.~Seok.
\newblock On critical and supercritical pseudo-relativistic nonlinear {S}chr\"odinger equations.
\newblock {\em Proc. Roy. Soc. Edinburgh Sect. A}, 150(3):1241--1263, 2020.

\bibitem{CN19}
V.~Coti~Zelati and M.~Nolasco.
\newblock Ground state for the relativistic one electron atom in a self-generated electromagnetic field.
\newblock {\em SIAM J. Math. Anal.}, 51(3):2206--2230, 2019.

\bibitem{CN25}
V.~Coti~Zelati and M.~Nolasco.
\newblock Normalized solutions for a nonlinear {D}irac equation.
\newblock {\em J. Differential Equations}, 414:746--772, 2025.

\bibitem{MR1947703}
A.~Cotsiolis and N.~Con. Tavoularis.
\newblock Sharp {S}obolev type inequalities for higher fractional derivatives.
\newblock {\em C. R. Math. Acad. Sci. Paris}, 335(10):801--804, 2002.

\bibitem{dyh}
Y.~Ding.
\newblock Semi-classical ground states concentrating on the nonlinear potential for a {D}irac equation.
\newblock {\em J. Differential Equations}, 249(5):1015--1034, 2010.

\bibitem{WJC}
Y.~Ding and J.~Wei.
\newblock Stationary states of nonlinear {D}irac equations with general potentials.
\newblock {\em Rev. Math. Phys.}, 20(8):1007--1032, 2008.

\bibitem{Dolbeault24}
J.~Dolbeault, D.~Gontier, F.~Pizzichillo, and H.~Van Den~Bosch.
\newblock Keller and {L}ieb–{T}hirring estimates of the eigenvalues in the gap of {D}irac operators.
\newblock {\em Rev. Mat. Iberoam.}, 40(2):649–692, 2024.

\bibitem{dongdingguo}
X.~Dong, Y.~Ding, and Q.~Guo.
\newblock Nonrelativistic limit and nonexistence of stationary solutions of nonlinear {D}irac equations.
\newblock {\em J. Differential Equations}, 372:161--193, 2023.

\bibitem{DST23}
S.~Dovetta, E.~Serra, and P.~Tilli.
\newblock Action versus energy ground states in nonlinear {S}chr\"odinger equations.
\newblock {\em Math. Ann.}, 385:1545–1576, 2023.

\bibitem{ED11}
E.~Engel and R.M. Dreizler.
\newblock {\em Density {F}unctional {T}heory: {A}n {A}dvanced {C}ourse}.
\newblock Theoretical and Mathematical Physics. Spinger-Verlag, Berlin Heidelberg, 2011.

\bibitem{ES01}
M.J. Esteban and E.~S\'{e}r\'{e}.
\newblock Nonrelativistic limit of the {D}irac-{F}ock equations.
\newblock {\em Ann. Henri Poincar\'{e}}, 2(5):941--961, 2001.

\bibitem{MR3243734}
L.~Grafakos.
\newblock {\em Classical {F}ourier analysis}, volume 249 of {\em Graduate Texts in Mathematics}.
\newblock Springer, New York, third edition, 2014.

\bibitem{GuoZeng20}
Y.~Guo and X.~Zeng.
\newblock The {L}ieb-{Y}au conjecture for ground states of pseudo-relativistic {B}oson stars.
\newblock {\em J. Funct. Anal.}, 278(12):108510, 24, 2020.

\bibitem{MR1785381}
P.D. Hislop.
\newblock Exponential decay of two-body eigenfunctions: a review.
\newblock In {\em Proceedings of the {S}ymposium on {M}athematical {P}hysics and {Q}uantum {F}ield {T}heory ({B}erkeley, {CA}, 1999)}, volume~4 of {\em Electron. J. Differ. Equ. Conf.}, pages 265--288. Southwest Texas State Univ., San Marcos, TX, 2000.

\bibitem{MR4430585}
Y.~Hong and S.~Jin.
\newblock Orbital stability for the mass-critical and supercritical pseudo-relativistic nonlinear {S}chr\"odinger equation.
\newblock {\em Discrete Contin. Dyn. Syst.}, 42(7):3103--3118, 2022.

\bibitem{MR2561169}
E.~Lenzmann.
\newblock Uniqueness of ground states for pseudorelativistic {H}artree equations.
\newblock {\em Anal. PDE}, 2(1):1--27, 2009.

\bibitem{LLS22}
M.~Lewin, E.H. Lieb, and R.~Seiringer.
\newblock Improved {L}ieb-{O}xford bound on the indirect and exchange energies.
\newblock {\em Lett. Math. Phys.}, 112(5):Paper No. 92, 36, 2022.

\bibitem{MR1817225}
E.H. Lieb and M.~Loss.
\newblock {\em Analysis}, volume~14 of {\em Graduate Studies in Mathematics}.
\newblock American Mathematical Society, Providence, RI, second edition, 2001.

\bibitem{LiebOxford}
E.H. Lieb and S.~Oxford.
\newblock Improved lower bound on the indirect coulomb energy.
\newblock {\em International Journal of Quantum Chemistry}, 19(3):427--439, 1981.

\bibitem{MR778970}
P.L. Lions.
\newblock The concentration-compactness principle in the calculus of variations. {T}he locally compact case. {I}.
\newblock {\em Ann. Inst. H. Poincar\'e{} Anal. Non Lin\'eaire}, 1(2):109--145, 1984.

\bibitem{MR778974}
P.L. Lions.
\newblock The concentration-compactness principle in the calculus of variations. {T}he locally compact case. {II}.
\newblock {\em Ann. Inst. H. Poincar\'e{} Anal. Non Lin\'eaire}, 1(4):223--283, 1984.

\bibitem{arXiv:2503.21405}
L.~Meng.
\newblock On the relativistic effect in the {Dirac}--{Fock} theory.
\newblock Preprint, {arXiv}:2503.21405, 2025.

\bibitem{Nolasco}
M.~Nolasco.
\newblock A normalized solitary wave solution of the {M}axwell-{D}irac equations.
\newblock {\em Ann. Inst. H. Poincar\'{e} C Anal. Non Lin\'{e}aire}, 38(6):1681--1702, 2021.

\bibitem{SzulkinWeth}
A.~Szulkin and T.~Weth.
\newblock The method of {N}ehari manifold.
\newblock In {\em Handbook of nonconvex analysis and applications}, pages 597--632. Int. Press, Somerville, MA, 2010.

\end{thebibliography}

\noindent {Pan Chen\\
School of Mathematical Sciences,\\
 Shanghai Jiao Tong University, Shanghai 200240, P.R. China
\\
e-mail: chenpan2020@amss.ac.cn }
\medskip
\\
\noindent {Yanheng Ding\\
Academy of Mathematics and Systems Science, \\
Chinese Academy of Sciences, Beijing, 100190, P.R. China\\
School of Mathematics,\\
Jilin University, Changchun, 130012, P.R. China\\
e-mail: dingyh@math.ac.cn }
\medskip
\\
\noindent {Qi Guo\\
School of Mathematics,\\
Renmin University of China, Beijing, 100872, P.R. China\\
e-mail: qguo@ruc.edu.cn}

\end{document}